\documentclass[12pt]{article}
\usepackage{amsfonts}
\usepackage{amsfonts}
\usepackage{amsfonts}
\usepackage{amsfonts}
\usepackage{amsfonts}
\usepackage{amsfonts}
\usepackage{amsfonts}
\usepackage{mathrsfs}
\usepackage{amsfonts}
\usepackage{amsmath, amsthm, amsfonts, amssymb}
\newtheorem{thm}{Theorem}[section]
\newtheorem{cor}[thm]{Corollary}
\newtheorem{lem}[thm]{Lemma}
\newtheorem{prop}[thm]{Proposition}
\newtheorem{example}[thm]{Example}

\theoremstyle{definition}
\newtheorem{defn}{Definition}[section]

\numberwithin{equation}{section} \theoremstyle{remark}
\newtheorem{rem}{Remark}[section]
\def\<{\langle}
\def\>{\rangle}

\def\wh{\widehat}
\def\ra{\rightarrow}
\def\p{\partial}
\def\a{\alpha}

\def\w{\widetilde}

\def\D{{\cal D}}
\def\O{\Omega}
\def\l{{\lambda}}

\def \sm{\setminus}
\def\CC{{\bf C}}

\def\-{\overline}
\def\ale{\mathrel{\mathop{<}\limits_{\sim}}}
\def\o{\omega}

\def\e{\epsilon}

\def\endpf{\hbox{\vrule height1.5ex width.5em}}

\def\b{\beta}
\def\a{\alpha}
\def\endpf{\hbox{\vrule height1.5ex width.5em}}
\def\CC{\bf C}

\def\M*{\wt{M^*}}

\def\-{\overline}
\def\ale{\mathrel{\mathop{<}\limits_{\sim}}}
\def\ld{\lambda}
\def\O{\Omega}
\def\o{\omega}

\def\D{\Delta}

\def\sm{\setminus}

\def\wt{\widetilde}
\def\ra{\rightarrow}

\def\endpf{\hbox{\vrule height1.5ex width.5em}}

\def\a{\alpha}

\def\D{\Delta}

\def\b{\beta}

\def\a{\alpha}

\def\endpf{\hbox{\vrule height1.5ex width.5em}}

\def\beq{\begin{equation}}
\def\nneq{\end{equation}}
\def\beqn{\begin{eqnarray}}
\def\neqn{\end{eqnarray}}
\def\beqna{\begin{eqnarray*}}
\def\neqna{\end{eqnarray*}}
\def\bedis{\begin{displaymath}}
\def\nedis{\end{displaymath}}
\def\-{\overline}

\begin{document}

\title{\bf   Flattening  of  CR singular points and analyticity of local hull of
holomorphy}
\author{Xiaojun Huang \footnote{
Supported in part by NSF-1101481} and Wanke Yin\footnote{Supported
in part by   ANR-09-BLAN-0422,  NSFC-10901123, FANEDD-201117,
RFDP-20090141120010,
 and NSFC-11271291.}}
\date{}
\maketitle \tableofcontents

\bigskip

\section{Introduction} A primary goal in this paper  is to study the
question that asks  when a real analytic submanifold $M$ in
${\CC}^{n+1}$
bounds a real analytic (up to $M$) Levi-flat hypersurface  $\wh{M}$
near $p\in M$ such that $\wh{M}$  is  foliated by a family of
complex hypersurfaces moving along the normal direction of $M$ at
$p$, and gives the invariant local hull of holomorphy of $M$ near
$p$. This question is equivalent to the holomorphic flattening
problem for $M$ near $p$.

To be more precise, we first discuss some basic holomorphic property
for a real submanifold  in a complex space. For a point $q$ in a
real submanifold $M\subset {\mathbb C}^{n+1}$, there is an immediate
holomorphic invariant, namely, the complex dimension $CR_{M}(q)$ of
the tangent space of type $(1,0)$ at $q$. $CR_{M}(q)$ is an upper
semi-continuous function over $M$. $q$ is called a CR point of $M$
if $CR_M(q')\equiv CR_M(q)$ for all $q'(\approx q)\in M$. Otherwise,
$q$ is called a CR singular point of $M$. When $M$ near $p$ bounds a
Levi-flat hypersurface foliated by a family of complex hypersurfaces
moving along the normal direction of $M$ at $p$, then the tangent
space of $M$ at $p$ is a complex hyperplane. In this case $p$ must
be a CR singular point unless we are in the trivial and
uninteresting situation that $M$ is a complex hypersurface itself.

Investigations for CR manifolds and CR singular manifolds have very
different nature. There is a vast amount of work related to the
study of various problems for CR manifolds, which goes back to the
work of Poincar\'e [Po], Cartan [Cat] and  Chern-Moser [CM]. The
study of submanifolds with CR singular points at least dates back to
the fundamental paper of Bishop [Bis] in 1965. Since then, many
efforts have been paid to understand both the geometric and analytic
structures of such manifolds. Here, we mention the papers by
Kenig-Webster [KW1-2], Moser-Webster [MW], Bedford-Gaveau [BG],
Huang-Krantz [HK], Huang [Hu1], Gong [Gon1-3], Huang-Yin [HY1-2],
 Stolovitch [Sto],
Dolbeault-Tomassini-Zaitsev [DTZ1-2], Ahern-Gong [AG], Coffman
[Cof1-2], Lebl [Le1-2], Burcea [Va1], etc,  and many references
therein.

Let $M\subset {\mathbb C}^{n+1}$ be a codimension two real
submanifold with {\it CR singular points}. Then a simple linear
algebra computation shows that $CR_M(q)=n-1$ when $q$ is a CR point,
and $CR_M(q)=n$ when $q$ is a  CR singular point.
The general holomorphic (or, formal) flattening problem is then to
ask when
$M$ can be  transformed, by   a biholomorphic (formal equivalence,
respectively) mapping, to an open piece of the standard Levi-flat
hyperplane $({\mathbb C}^n\times {\mathbb R}^{1})\times \{0\}\subset
{\mathbb C}^{n+1}$.
A good understanding  to this problem  is crucial for understanding
many  geometric, analytic and dynamic properties of the manifolds.
For instance, by a classical theorem of Cartan, solving the problem
when $M$ bounds a real analytic (up to $M$) Levi-flat hypersurface
is equivalent to solving the holomorphic flattening problem of the
manifold.
Here, we refer the reader to the papers by Kenig-Webster [KW1],
Moser-Webster [MW], Huang-Krantz [HK], Gong [Gon1-3], Stolovitch
[Sto], Huang-Yin [HY1],  Dolbeault-Tomassini-Zaitsev[DTZ1], and many
references therein, for investigations along these lines.

The major difficulty for getting the flattening  property for $M$
lies in the complicated nature of the CR singular points. And, in
general, only non-degenerate CR singular points with a rich
geometric structure could be flattened.
To be more precise, we use $(z_1,\cdots,z_n,w)$ for the complex
coordinates of ${\mathbb C}^{n+1}$. We first make the following
definition. For  related concepts and many intrinsic discussions on
this matter, see the work in Stolovitch [Sto],
Dolbeault-Tomassini-Zaitsev [DTZ1], and Huang-Yin [HY2]:

\begin{defn} \label {101}Let $M$ be a codimension two real submanifold in
${\mathbb C}^{n+1}$. We say $q\in M$ is a non-degenerate CR singular
point, or a non-degenerate complex tangent point, if there is a
biholomorphic change of coordinates which maps $p$ to $0$ and in the
new coordinates $(z,w)$, $M$ is defined near $0$ by an equation of
the following form:
\begin{equation}\label{intr-01}
w=\sum_{j=1}^{n}\left(|z_j|^2+\lambda_j(z_j^2+\-{z_j^2})\right)+o(|z|^2)
\end{equation}
Here,  $0\le \lambda_1,\cdots, \lambda_n< \infty$.
$\{\lambda_1,\cdots,\lambda_n\}$ (counting multiplicity)   are
called the Bishop invariants of $M$ at $0$.
We call $\lambda_j$ an elliptic, parabolic or hyperbolic Bishop
invariant of $M$ at $0$ in terms of $\lambda_j<1/2, $
$\lambda_j=1/2$, or $\lambda_j>1/2, $ respectively.
\end{defn}
Notice that the set of Bishop invariants at  a non-degenerate CR
singular point $p\in M$  consists of  the only second order
biholomorphic invariants of $M$ at  $p\in M$. By the results in
Moser-Webster [MW] and Huang-Krantz [HK], in the case of complex
dimension two ($n+1=2$), any real analytic surface near an elliptic
CR singular point can be flattened. On the other hand, a generic
real analytic surface near a parabolic or hyperbolic CR singular
point can not be flattened, though it can be formally flattened
whenever the Bishop invariant is not exceptional. See the work of
Moser-Webster [MW], Gong [Gon 1-3] and a very recent paper by
Ahern-Gong [AG] on many discussions on this matter. Here, we recall
that a Bishop invariant $\lambda$ is called non-exceptional if the
following quadratic equation in $\nu$ has no roots of unity:
\begin{equation}\label {intro-03}
\lambda\nu^2-\nu+\lambda=0.
 \end{equation}
 However, the situation
for $n>1$ is very different. Consider the following codimension two
real analytic submanifold in ${\mathbb C}^3$:
\begin{example}
\label{intro-02}
\begin{equation}
\label{intro-003} M:=\{\
w=\sum_{j=1}^{2}|z_j|^2+2\Re\left(\sum_{j_1+ j_2\ge 3}
a_{j_1j_2}z_1^{j_1}z_2^{j_2}\right)+\sqrt{-1}\sum_{j_1\ge 2, j_2\ge
2}b_{j_1 \-{j_2}}z_1^{j_1}\-{z_2}^{j_2},\ \ b_{j\bar{ l}}=\-{b_{l\-
j}}.\}
\end{equation}
\end{example}
$M$ has  a non-degenerate CR singular point at $0$ and all Bishop
invariants of $M$ at $0$ are $0$ and thus all elliptic. It was shown
in Huang-Yin [HY2] ([Remark 2.7, HY2]) that $(M,0)$ can not even be
flattened to the order $m$ if $b_{j_1\-{j_2}}\not =0$ for some
$j_1+j_2\le m$. Namely,  if $b_{j_1\-{j_2}}\not =0$ for some
$j_1+j_2\le m$, then there is no holomorphic change of variables
(preserving the origin) such that in the new coordinates, $M$ is
defined near $0$ by an equation of the form $w=\rho$ with the
property that $\Im(\rho)$ vanishes at the origin to the order at
least $m$.

Example \ref{intro-02} shows that in higher dimensions, the geometry
from the nearby CR points also play a  role in the flattening
problem, while in the two variables case, the nearby points are
totally real and can  all be locally holomorphically flattened. Thus
the nearby points in the two dimension case has no influence for the
holomorphic property at a non-degenerate CR singular point.  Indeed,
suppose $M$ is already flattened and is defined by an equation of
the form $u=q(z,\-{z}),v=0$, where $w=u+iv$. Then the complex
hypersurface $S_{u_0}=:\{w=u_0+i0\}$ with $u_0\in {\mathbb R}$
intersects $M$ along a CR submanifold $E$ of CR dimension $(n-1)$
near $p_0$ if $S_{u_0}$ intersects $M$ (CR) transversally at $p_0$.
The points where $S_{u_0}$ is (CR) tangent to $M$ are apparently CR
singular points of $M$. Recall a well-known terminology (see [T] and
[Tu]): A point $p$ in a CR submanifold $N$ is called a non-minimal
point if $N$ contains a proper CR submanifold $ S$ containing $p$
such that $T^{(1,0)}_pS=T^{(1,0)}_pN.$ Hence, in such a terminology,
we have the following simple fact:

\medskip
{\it If $M$ can be flattened, then all CR points in $M$ are
non-minimal CR points.}

\medskip
We mention that the  necessary condition  for the non-minimality  of
CR points already appeared in the earlier  work of
Dolbeault-Tomassini-Zaitsev [DTZ1-2] and Lebl [Le1-2] on the study
of the general complex Plateau problem, which looks for  the
Levi-flat varieties (even maybe in the sense of current) bounded by
the given manifolds.

 Our main results, which we  state  below, demonstrate
that, with the non-minimality assumption at  CR points, the
existence of one  Bishop invariant not being parabolic, namely, not
equal to $\frac{1}{2}$, is good enough for the formal flattening and
the existence of just one elliptic Bishop invariant suffices for the
holomorphic flattening:

\begin{thm}\label{thmm1} Let $M\subset {\mathbb C}^{n+1}$ with $n>1$ be a
codimension two smooth  submanifold with $p\in M$ a non-degenerate
complex tangent point $p\in M$.  Suppose that one element $\lambda$
from the set of Bishop invariants of $M$ at $p$ is not parabolic,
namely, not equal to $\frac{1}{2}$. Also assume that all  CR points
of $M$ near $p$ are non-minimal. Then $M$ can be formally flattened
near $p$. Namely, for any positive integer $m$, there is a
holomorphic change of coordinates which maps $p$ to $0$ and maps $M$
to a manifold defined by an equation of the form $w=\rho(z,\-{z})$
with $\Im{\rho}$ vanishing at least to  the order $m$ at the origin.
\end{thm}


We mention  that the result in Theorem \ref{thmm1} holds even if $M$
is assumed just to be a formal submanifold with the same type of
assumptions, or we need only assume that the set of non-minimal CR
points over $M$ forms an open subset $O$ with $p\in \-{O}$.  (See
Theorem \ref{thm2} and Corollary \ref{thm3}.) However, as
demonstrated even in the two dimensional case by Moser-Webster [MW]
and Gong [Gon1-3], more geometric structure is needed to get the
holomorphic flattening in the above theorem. Indeed, making use of
the construction of holomorphic disks in Kenig-Webster [KW1] and
Huang-Krantz [HK], we have the following convergence result for
Theorem \ref{thmm1} under the assumption of at least one ellipticity
for the Bishop invariants:

\begin{thm}\label{thmm2} Let $M\subset {\mathbb C}^{n+1}$ with $n>1$ be a codimension two
real analytic CR manifold with $p\in M$ a non-degenerate complex
tangent point (namely, a non-degenerate CR singular point). Suppose
one of the Bishop invariants $\lambda$ of $M$ at $p$ is elliptic.
%
Then $M$ near $p$ can be holomorphically flattened
 if and only if all CR points of $M$
near $p$ are non-minimal.
\end{thm}

As we mentioned above, by the classical  Cartan theorem ([Cat]),
Theorem \ref{thmm2} is equivalent to the following geometric
theorem:

\begin{thm}\label{thmm3} Let $M\subset {\mathbb C}^{n+1}$ with $n>1$ be a codimension
 two real analytic CR manifold with
$p\in M$ a non-degenerate complex tangent point. Suppose one of the
Bishop invariants $\lambda$ of $M$ at $p$ is elliptic. Also assume
that all CR points of $M$ near $p$ are non-minimal. Then the local
hull of holomorphy $\wh{M}$ of $M$ near $p$ is a real analytic
Levi-flat hypersurface which has  $M$ near $p$ as part of its real
analytic boundary. Moreover $\wh{M}$   is  foliated by a family of
smooth complex hypersurfaces in ${\mathbb C}^{n+1}$, that moves
along the transversal direction of the tangent space of $M$ at $p$.
\end{thm}

\begin{example} \label {exam-beg}{\rm
Define $M\subset {\mathbb C}^3=:\{(z_1,z_2,w)\}$ by the following
equation near $0$:
$$w=q(z,\-{z})+p(z,\-{z})+iE(z,\-{z}).$$
Here
$q=|z_1|^2+\ld_1(z_1^2+\-{z_1^2})+|z_2|^2+\ld_2(z_2^2+\-{z_2^2})$
with $0\le \ld_1,\ld_2<\infty$, and
$$p(z,\-{z})+iE(z,\-{z})=\mu_1|z_1|^2({z_1}+\ld_1\-{z_1})+\mu_2|z_2|^2({z_2}
+\ld_2\-{z_2})+\mu_1z_1(|z_2|^2+\ld_2\-{z_2^2})+
\mu_2{z_2}(|z_1|^2+\ld_1\-{z_1^2}).$$ Here $\mu_1, \mu_2$ are two
complex numbers. Then, $M$ is non-minimal at its CR points near its
non-degenerate CR singular point $0$. (See Example
\ref{example-new}.) Hence, our result says that when one of the
$\ld_1,\ld_2$ is not $\frac{1}{2}$, then $M$ can be formally
flattened at $0$; and when one of the $\ld_1,\ld_2$ is less than
$\frac{1}{2}$, then $M$ can be holomorphically flattened near $0$.

In this example, $M\sm \{0\}$ near $0$ is foliated by a family of
three dimensional strongly pseudoconvex CR manifolds---  the
intersections of $M$ with  real hypersurfaces $K_c: q(z,\-{z})=c$
with $c\in {\mathbb R}.$ (When both $\ld_1,\ld_2$ are elliptic,
$c>0$). Assume that one of the Bishop invariants $\{\ld_1,\ld_2\}$
is not elliptic. Then there is an orbit corresponding to $c=0$, that
extends to the CR singular point  with it as its non-smooth point.
Also none of the orbits closes up near $0$.
 }
\end{example}

We next say a few words about the proof of our main geometric
result: To prove Theorem \ref{thmm3}, we first slice $M$ near the
complex tangent point $p$ by a family of two dimensional complex
planes along the elliptic direction. We then get a family of
elliptic Bishop surfaces. Now each one bounds a three dimensional
Levi flat CR manifold and their union forms a codimension one subset
$\wt{M}$ in ${\CC}^{n+1}$  with $M$ as part of its boundary. An
analysis, based on Bishop disks, similar to that in Kenig-Webster
[KW1], and in particular, in Huang-Krantz [HK], shows that $\wt{M}$
is a real analytic hypersurface with $M$ as part of its real
analytic boundary. However, all we know from this construction is
that $\wt{M}$ has only one Levi-flat direction (along the elliptic
direction). And it is not clear at all if $\wt{M}$ is flat along the
parameter directions. In fact, $\wt{M}$ can not be Levi flat without
the non-minimality property  from  the nearby CR points. Now, the
crucial issue is that, with the assumption of the non-minimality at
the nearby CR points, we can find a formal transformation which
makes $M$ formally flattened, while any finite order truncation of
this transformation preserves $\wt{M}$.  The existence of this
transformation is the content of Theorem \ref{thm1}, which is a more
general but also more technical version of Theorem \ref{thmm1}.
(Notice that in the two dimensional setting, the uniqueness of
$\wt{M}$ is done normally by showing that $\wt{M}$ is the local hull
of holomorphy of $M$ and thus is invariant under biholomorphic
transformation. However, this is more or less equivalent to proving
that  $\wt{M}$ is Levi-flat. Hence it   can not be achieved in this
way  in higher dimensions.) After this is done, we see that the
Levi-form of $\wt{M}$  vanishes to any high order as we like. Since
$\wt{M}$ is real analytic up to $M$, by the unique continuation
property for real-analytic functions, we conclude that the Levi-form
of $\wt{M}$ vanishes everywhere. Thus, $\wt{M}$ must be Levi-flat
everywhere.

 Most part of the paper is
devoted to the proof Theorem \ref{thm1} (a more general and more
technical version of Theorem \ref{thmm1}).
Here one sees an essential  difference from arguments in the two
dimensional case. Indeed, the phenomenon is also different in this
setting as there is no need to impose   the non-exceptional property
for even a single Bishop invariant. Our basic idea for the proof of
Theorem \ref{thm1} goes as follows: Suppose $M$ in Theorem
\ref{thm1} is flattened to order $m-1$. We first normalize the
$m^{\rm th}$-order of the imaginary part of the defining function of
$M$ to fix all possible free choices of coordinates. This will be
done in Theorem \ref{norm}. Then we show that the non-minimality of
the nearby CR points forces the vanishing of such a normal form. In
$\S 2$ and $\S 3$, we will derive three basic equations that must be
satisfied for $M$ under the non-minimality assumption. These will be
used in $\S 4$ and $\S 5$ to prove Theorem \ref{thm1}.


 Theorem \ref{thmm2} is equivalent to
Theorem \ref{thmm3} by  a classical result of Cartan which states
that a real analytic  hypersurface is Levi-flat if and only if it
can be transformed locally to an open piece of the standard
Levi-flat hyperplane defined by $\Im{w}=0$. When all Bishop
invariants at $p$ are elliptic, we mention that Theorem \ref{thmm3}
can also be derived by combining the results obtained in
Dolbeault-Tomassini-Zaitsev [DTZ1-2] and the work in a very recent
preprint by Burcea [Bur2] with  a different approach. (The work in
Dolbeault-Tomassini-Zaitsev [DTZ1-2]  contains other very nice
global results.) The arguments based on Dolbeault-Tomassini-Zaitsev
[DTZ1-2] and Burcea [Bur2] depend strongly on all the ellipticity of
Bishop invariants and  requires that the CR orbits in $M$ near the
CR singular point form a family of compact strongly pseudoconvex
manifolds shrinking down to the complex tangent such that the
Harvey-Lawson theorem applies.
This   is  certainly  not the case even when one non-elliptic Bishop
invariant at the CR singular point appears.

We also include an appendix to give a detailed proof of Theorem
\ref{thm1} in the special case of $n=2$ and $m=3$. The reader may
like to read the Appendix before reading $\S 4-\S 6$. By including
such an appendix, we hope it will help the reader to see the basic
ideas, through a simple case,  the  complicated argument for the
proof of Theorem \ref{thm1} in the general setting in $\S 4-\S6$.

\medskip
{\bf Acknowledgment}: The major part of the paper  was  completed in
the summer of 2011 when  the first author was visiting Wuhan
University. The first author would like to thank the School of
Mathematics and Statistics, Wuhan University for the hospitality
during his stay. Part of the work in the paper was done while the
second author was taking a year long visit at Rutgers University at
New Brunswick in 2009. The second author likes to thank this
institute for the hospitality during his stay. The second author
also likes  very much  to thank Nordine Mir for his many helps both
in his mathematics and in other arrangements during his stay at
Universit\'e de Rouen, through a European Union postdoctoral
fellowship.

\section{An immediate consequence for non-minimality near CR points}

 Let $(M,0)$ be a smooth submanifold of codimension two in
${\mathbb{C} }^{n+1}$ with $0\in M$ as a CR singular point. Assume
that the CR singular point at $0\in M$ is non-degenerate as defined
in Definition \ref{101} such that after a holomorphic change of
coordinates, $M$ near $0$ is defined by an equation of the form:

\begin{equation}
w=q(z,\-{z})+p(z,\-z)+iE(z,\-z), \label{429eq1}
\end{equation}
where $q(z,\-z)=\sum_{i=1}^{n}(|z_i|^2+\lambda_i(z_i^2+\-z_i^2))$
with $0\le \lambda_1,\cdots, \lambda_n<\infty$  being the Bishop
invariants of $M$ at $0$, $\text{Ord}(p),\text{Ord}(E)\geq 3$ and
both $p(z,\-z)$ and $E(z,\-z)$ are real-valued smooth functions. For
convenience of notation, we also write
\begin{equation}
F(z,\-{z})=p(z,\-{z})+iE(z,\-z)\ \text{and}\ G(z,\-z)=
q(z,\-{z})+p(z,\-z). \label{82eq1}
\end{equation}
Then we have
$$
w=q(z,\-{z})+F(z,\-z)=G(z,\-{z})+iE(z,\-z).
$$

In what follows, as is standard in the literature, we write $
\chi_{\a}=\frac{\p \chi}{\p z_\a},\ \chi_{\-{\a}}=\frac{\p \chi}{\p
\-z_{\a}}$ with $1\leq \a\leq n$ for a smooth function $\chi(z,\-z)$
in $z$. For $1\leq j\leq n-1$, we define
\begin{equation}\begin{split}\label{325eq2}
L_j=&(G_n-iE_n)\frac{\p}{\p z_j}-(G_j-iE_j)\frac{\p}{\p
z_n}+2i(G_nE_j-G_jE_n)\frac{\p}{\p w}\\
:=&A\frac{\p}{\p z_j}-B_{(j)}\frac{\p}{\p z_n}+C_{(j)}\frac{\p}{\p
w}.
\end{split}\end{equation}
Then we have
\begin{equation*}\begin{split}
L_j(-w+G+iE)&=(G_n-iE_n)(G_j+iE_j)-(G_j-iE_j)(G_n+iE_n)\\
&\ \ -2i(G_nE_j-G_jE_n)=0,\\
L_j(\-{-w+G+iE})&=(G_n-iE_n)(G_j-iE_j)-(G_j-iE_j)(G_n-iE_n) =0.
\end{split}\end{equation*}
Hence $L_1,\cdots,L_{n-1}$ are complex tangent vector fields of type
$(1,0)$ along $M$ near $0$. Moreover, for $1\le j,k\le n-1,$ a
straightforward computation shows that
\begin{equation}\begin{split}
[L_j,\-L_k]=&\left[A\frac{\p}{\p z_j}-B_{(j)}\frac{\p}{\p
z_n}+C_{(j)}\frac{\p}{\p w},\-A\frac{\p}{\p
\-z_k}-\-{B_{(k)}}\frac{\p}{\p
\-z_n}+\-{C_{(k)}}\frac{\p}{\p \-w}\right]\\
 =&\lambda_{(1jk)}\frac{\p}{\p
\-z_k}+\lambda_{(2jk)}\frac{\p}{\p
\-z_n}+\lambda_{(3jk)}\frac{\p}{\p \-w}+\lambda_{(4jk)}\frac{\p}{\p
z_j}+\lambda_{(5jk)}\frac{\p}{\p z_n}+\lambda_{(6jk)}\frac{\p}{\p
w},
\end{split}\end{equation}
where
\begin{equation}\begin{split}\label{325eq3}
\lambda_{(1jk)}=A\cdot (\-A)_j-B_{(j)}(\-{A})_n,& \ \lambda_{(4jk)}
=-\-A\cdot A_{\-k}+\-{B_{(k)}}{A}_{\-n},\\
\lambda_{(2jk)}=-A\cdot (\-{B_{(k)}})_j+B_{(j)}(\-{B_{(k)}})_n,&\
\lambda_{(5jk)}=\-A\cdot ({B_{(j)}})
_{\-k}-\-{B_{(k)}}({B_{(j)}})_{\-n},\\
\lambda_{(3jk)}=A\cdot (\-{C_{(k)}})_j-B_{(j)}(\-{C_{(k)}})_n,&\
\lambda_{(6jk)}=-\-A\cdot
({C_{(j)}})_{\-k}+\-{B_{(k)}}({C_{(j)}})_{\-n}.
\end{split}\end{equation}
Notice that
\begin{equation}\label{57eq8}
\lambda_{(1jk)}=-\-{\lambda_{(4kj)}},\
\lambda_{(2jk)}=-\-{\lambda_{(5kj)}},\
\lambda_{(3jk)}=-\-{\lambda_{(6kj)}}.
\end{equation}
In what follows, write $\ w_j=z_j+2\lambda_j\-z_j\  \text{for}\
1\leq j\leq n$. Suppose that $E\not\equiv 0$. We write in what
follows that
\begin{equation}\label{eh}
\text{Ord}(E)=m\ \text{and}\ H(z,\-{z}):=E^{(m)}(z,\-z).
\end{equation}
 From (\ref{325eq2}),  we get the following approximation
properties:
\begin{equation}\begin{split}\label{326eq4}
A=\-w_n+O(2),\ B_{(j)}=\-w_j+O(2),\ C_{(j)}=2i\-{\Phi_{(j)}}+O(m+1).
\end{split}\end{equation}
Here  (and also in what follows), we have
\begin{equation}\label{phiexp}
  \Phi_{(j)}=w_nH_{\-j}-w_jH_{\-n},\hbox{ and we write}\  \Phi=\Phi_{(1)}.
\end{equation}
For  future applications, we also write
\begin{equation}\begin{split}\label{psiexp}
\Psi_{(jk)}=w_n\-{w_n}
(\Phi_{(j)})_k-{w_n}\-{w_k}(\Phi_{(j)})_n+\-{w_k}\cdot \Phi_{(j)},\
\Psi=\Psi_{(11)}.
\end{split}\end{equation}

Substituting these approximation properties to (\ref{325eq3}), we
obtain
\begin{equation*}\begin{split}
\lambda_{(1jk)}=&\big(\-w_n+O(2)\big)\cdot
\big((w_n)_j+O(1)\big)-\big(\-w_j+O(2)\big)\cdot
           \big((w_n)_n+O(1)\big)
           =-\-w_j+O(2),\\
\lambda_{(2jk)}=&-\big(\-w_n+O(2)\big)\cdot
\big((w_k)_j+O(1)\big)+\big(\-w_j+O(2)\big)\cdot
           \big((w_k)_n+O(1)\big)=-\delta_{jk}\-w_n+O(2),\\
\lambda_{(3jk)}=&\big(\-w_n+O(2)\big)\cdot
\big(-2i\Phi_{(k)}+O(m+1)\big)_j -\big(\-w_j+O(2)\big)\cdot
\big(-2i\Phi_{(k)}+O(m+1)\big)_n\\
=&-2i\-w_n(\Phi_{(k)})_j+2i\-w_j(\Phi_{(k)})_n+O(m+1).
\end{split}\end{equation*}
Combining these relations with (\ref{57eq8}), we get
\begin{equation}\begin{split}\label{326eq5}
\lambda_{(1jk)}=&-\-w_j+O(2),\ \lambda_{(2jk)}=-\delta_{jk}\-w_n+O(2),\\
\lambda_{(4jk)}=&w_k+O(2),\ \ \ \ \ \lambda_{(5jk)}=\delta_{jk}w_n+O(2),\\
\lambda_{(3jk)}=&-2i\-w_n(\Phi_{(k)})_j+2i\-w_j(\Phi_{(k)})_n+O(m+1),\\
\lambda_{(6jk)}=&-2iw_n(\-{\Phi_{(j)}})_{\-{k}}+2iw_k(\-{\Phi_{(j)}})_{\-{n}}+O(m+1).
\end{split}\end{equation}

In what follows, we further assume  that $M$ is non-minimal at its
CR points. Write ${\cal S}$ for the set of CR singular points of $M$
near $0$. Suppose that $0$ is not an isolated point in ${\cal S}$.
Notice that $T^{(1,0)}_{0}M=\hbox {span} \{\frac{\p }{\p
z_1}|_0,\cdots, \frac{\p }{\p z_n}|_0 \}$. For $p_0\in {\cal S}$
close to $0$, we easily see that  $T^{(1,0)}_{p_0}M=\hbox {span}
\{X_1,\cdots, X_n\}$ for certain  tangent vectors of type $(1,0)$ of
the form: $X_j=\frac{\p }{\p z_j}|_{p_0}+b_j\frac{\p }{\p
w}|_{p_0}$, $j=1,\cdots, n$. Since $X_j(\-{-w+q+F})=0$,  we get
$\-{z_j}+2\lambda_j{z_j}=O(|z|^2)$ when $q_0\approx 0$. Write
$z_j=x_j+\sqrt{-1}y_j$. Then $(1+2\lambda_j)x_j=O(|z|^2),
(1-2\lambda_j)y_j=O(|z|^2)$ for $j=1,\cdots, n$. By the implicit
function theorem,
 we conclude that
${\cal S}$ is contained in a submanifold of $M$ near $0$ which has
at most real dimension $n$; and when none of the Bishop invariants
of $M$ at $0$ is parabolic, the only CR singular point of $M$ near
$0$ is $0$ itself.

We next claim that there is an open dense subset ${\cal O}_i$ of $M$
near $0$ such that for any  $q_0\in {\cal O}_i$, at $q_0$, it holds
that
\begin{equation}\begin{split}\label{326eq2}
[L_i,\-L_i]\not \in \text{Span}\{L_j,\-{L_j}\}_{1\leq j\leq n-1}\
\text{for}\ 1\leq i\leq n-1.
\end{split}\end{equation}
We prove the claim by contradiction. Suppose that we have at $q_0\in
{\cal O}_i$ the following:
\begin{equation}\begin{split}
[L_i,\-L_i]&=\sum_{l=1}^{n-1}\big(\hat{\a}_lL_l+\hat{\b}_l\-L_l\big)\\
           &=\sum_{l=1}^{n-1}\hat{\a}_l\big(A\frac{\p}{\p z_l}-B_{(l)}\frac{\p}{\p
           z_n}+C_{(l)}\frac{\p}{\p w}\big)
           +\sum_{l=1}^{n-1}\hat{\b}_l\big(\-A\frac{\p}{\p \-z_l}-\-{B_{(l)}}\frac{\p}{\p
           \-z_n}+\-{C_{(l)}}\frac{\p}{\p \-w}\big).
\end{split}\end{equation}
Then by considering the coefficients of $\frac{\p}{\p z_l}$ and
$\frac{\p}{\p \-z_l}$ for $1\leq l\leq n-1$, we obtain
\begin{equation*}\begin{split}
\hat{\a}_l=\hat{\b}_l=0\ \text{for}\ l\neq i,
\lambda_{(4ii)}=A\hat{\a}_i,\ \lambda_{(5ii)}=-B_{(i)}\hat{\a}_i.
\end{split}\end{equation*}
Eliminating $\hat{\a}_i$ from the above,  we get $
A\lambda_{(5ii)}+B_{(i)}{\lambda_{(4ii)}}=0 . $ Combining this with
the approximation properties (\ref{326eq4}) and (\ref{326eq5}), we
get $|w_i|^2+|w_n|^2+O(3)=0$ at $q_0\in {\cal O}_i$. Now, write this
equation as
\begin{equation}\label {0*}
|w_i|^2+(1+2\l_n)^2x_n^2+(1-2\l_n)^2y_n^2+h_0+h_1x_n+h_2x_n^2+O(x_n^3)=0.
\end{equation}
Notice that when $M$ is real analytic, it defines a closed proper
analytic variety over $M$ and ${\cal O}_i$ can be simply defined as
its compliment. In general, suppose it defines a subset which
contains an open subset $V_i$ with $0\in \-{V_i}$. Differentiating
(\ref{0*}) with respect to $x_n$, we get the following  over $V_i$:
$$\left(2(1+2\l_n)^2+2h_2\right)x_n=-h_1+O(x_n^2).$$
Since $h_2=o(1)$, by the implicit function theorem, the above
defines a proper submanifold in $M$. This is contradiction.

In the following, we write ${\cal O}=\cap_{i=1}^{n-1}{\cal O}_i\sm
{\cal S},$ which ${\cal O}$ is an open dense subset of $M$ near $0$.
In particular,  at $q_0(\approx 0)\in  {\cal O}$, we have
$$
T:=[L_1,\-L_1]\not \in \text{Span}\{L_j,\-{L_j}\}_{1\leq j\leq n-1}.
$$
Hence, by the Frobenius theorem, the non-minimality at the subset
${\cal O}$ of the CR points (sufficiently close to 0) is equivalent
to the following property when restricted to the subset  ${\cal O}$:
\begin{equation}\label{716eq1}
    [L_i,\-L_j],[[L_i,\-L_j],L_k]\in \text{Span}\big\{\{L_h,\-{L_h}\}_{1\leq h\leq n-1},T\big\}
    \ \text{for}\ 1\leq i,j,k\leq n-1.
\end{equation}
Recall the following notation we set up before:
$$ T=\lambda_{(111)}\frac{\p}{\p
\-z_1}+\lambda_{(211)}\frac{\p}{\p
\-z_n}+\lambda_{(311)}\frac{\p}{\p \-w}+ \lambda_{(411)}\frac{\p}{\p
z_1}+\lambda_{(511)}\frac{\p}{\p z_n}+\lambda_{(611)}\frac{\p}{\p
w}.$$
 Next we give  equivalent conditions for (\ref{716eq1}), which
are much easier to apply.

\medskip
First, since $[L_j,\-L_k]\in \text{Span}\{\{L_h,\-{L_h}\}_{1\leq
h\leq n-1},T\}$ with $1\leq j,k\leq n-1$ over $M\sm {\cal S}$, we
have, over $M\sm {\cal S}$, the following
$$
[L_j,\-L_k]=\sum\limits_{l=1}^{n-1}(\a_l L_l+\b_l \-L_l)+\gamma T\
\text{for some coefficients}\ \alpha_l,\beta_l,\gamma.
$$
Namely, we have
\begin{equation*}\begin{split}
&\lambda_{(1jk)}\frac{\p}{\p \-z_k}+\lambda_{(2jk)}\frac{\p}{\p
\-z_n}+\lambda_{(3jk)}\frac{\p}{\p \-w}+ \lambda_{(4jk)}\frac{\p}{\p
z_j}+\lambda_{(5jk)}\frac{\p}{\p z_n}+\lambda_{(6jk)}\frac{\p}{\p
w}\\
=&\sum\limits_{l=1}^{n-1}\a_l \left(A\frac{\p}{\p
z_l}-B_{(l)}\frac{\p}{\p z_n}+C_{(l)}\frac{\p}{\p
w}\right)+\sum\limits_{l=1}^{n-1}\b_l \left(\-A\frac{\p}{\p
\-z_l}-\-{B_{(l)}}\frac{\p}{\p \-z_n}+\-{C_{(l)}}\frac{\p}{\p \-w}\right)\\
&+\gamma \left(\lambda_{(111)}\frac{\p}{\p
\-z_1}+\lambda_{(211)}\frac{\p}{\p
\-z_n}+\lambda_{(311)}\frac{\p}{\p \-w}+ \lambda_{(411)}\frac{\p}{\p
z_1}+\lambda_{(511)}\frac{\p}{\p z_n}+\lambda_{(611)}\frac{\p}{\p
w}\right).
\end{split}\end{equation*}

Comparing the coefficients of $\{\frac{\p}{\p z_h},\frac{\p}{\p
\-{z}_h}\}_{1\leq h\leq n-1},\frac{\p}{\p w},\frac{\p}{\p \-w}$,
respectively, we  get, over ${\cal O}$,  the following:

(I) If $j\neq 1 \ \hbox{and }\ k\neq 1$, then we have $\a_l,\b_l=0$
for $l\neq 1,l\neq j, l\neq k$. Moreover,  we have
\begin{equation}\begin{split}\label{39eq1}
\-A\cdot \b_1+\gamma\cdot \lambda_{(111)}=0,\ A\cdot\a_1+\gamma\cdot
\lambda_{(411)}=0,\ \lambda_{(1jk)}=\-A\cdot \b_k,\
\lambda_{(4jk)}=A\cdot
\a_j,\\
\lambda_{(2jk)}=-\b_1\cdot
\-{B_{(1)}}-\b_k\cdot\-{B_{(k)}}+\gamma\cdot \lambda_{(211)},\
\lambda_{(3jk)}=\b_1\cdot
\-{C_{(1)}}+\b_k\cdot\-{C_{(k)}}+\gamma\cdot \lambda_{(311)},\\
\lambda_{(5jk)}=-\a_1\cdot {B_{(1)}}-\a_j\cdot{B_{(j)}}+\gamma\cdot
\lambda_{(511)},\ \lambda_{(6jk)}=\a_1\cdot
{C_{(1)}}+\a_j\cdot{C_{(j)}}+\gamma\cdot \lambda_{(611)}.
\end{split}\end{equation}

(II) If $j=1$ but $k\neq1$, then we get (i) $\a_l=0$ for $l\not 1$
and (ii) $\b_l=0$ for $l\neq 1,l\neq k$. Moreover, we have
\begin{equation}\begin{split}\label{39eq2}
&\-A\cdot \b_1+\gamma\cdot \lambda_{(111)}=0,\
\lambda_{(41k)}=A\cdot\a_1+\gamma\cdot \lambda_{(411)},\
\lambda_{(11k)}=\-A\cdot \b_k,\\
&\lambda_{(21k)}=-\b_1\cdot
\-{B_{(1)}}-\b_k\cdot\-{B_{(k)}}+\gamma\cdot \lambda_{(211)},\
\lambda_{(31k)}=\b_1\cdot
\-{C_{(1)}}+\b_k\cdot\-{C_{(k)}}+\gamma\cdot \lambda_{(311)},\\
&\lambda_{(51k)}=-\a_1\cdot B_{(1)}+\gamma\cdot \lambda_{(511)},\
\lambda_{(61k)}=\a_1\cdot C_{(1)}+\gamma\cdot \lambda_{(611)}.
\end{split}\end{equation}

Back to (\ref{39eq1}), we  get from its first line that
$$
\-A\b_1=- \gamma\cdot \lambda_{(111)},\ A\a_1=- \gamma\cdot
\lambda_{(411)},\ \-A\b_k=\lambda_{(1jk)},\ A\a_j=\lambda_{(4jk)}.
$$
Multiplying $\-A$ and $A$ to the  second and the third lines  in
(\ref{39eq1}), respectively, and making use of  the just obtained
relations, we obtain over $ {\cal O}$ the following:
\begin{align*}
&\-A\lambda_{(2jk)}= \gamma\cdot \lambda_{(111)}\cdot
\-{B_{(1)}}- \lambda_{(1jk)}\cdot\-{B_{(k)}}+\-A\gamma\cdot \lambda_{(211)},\\
&\-A\lambda_{(3jk)}=- \gamma\cdot \lambda_{(111)}\cdot
\-{C_{(1)}}+ \lambda_{(1jk)}\cdot\-{C_{(k)}}+\-A\gamma\cdot \lambda_{(311)},\\
&A\lambda_{(5jk)}=\gamma\cdot \lambda_{(411)}\cdot
{B_{(1)}}-\lambda_{(4jk)}\cdot{B_{(j)}}
    +A\gamma\cdot    \lambda_{(511)},\\
&A\lambda_{(6jk)}=-\gamma\cdot \lambda_{(411)}\cdot
  {C_{(1)}}+ \lambda_{(4jk)}\cdot{C_{(j)}}+A\gamma\cdot \lambda_{(611)}.
\end{align*}
After rewriting,  we get
\begin{equation}\begin{split}\label{717eq1}
\big(\-A\cdot\lambda_{(211)}+ \-{B_{(1)}}\cdot
\lambda_{(111)}\big)\cdot
\gamma=&\-A\lambda_{(2jk)}+ \lambda_{(1jk)}\cdot\-{B_{(k)}},\\
\big(\-A\cdot\lambda_{(311)}-
\-{C_{(1)}}\cdot\lambda_{(111)}\big)\cdot
\gamma=&\-A\lambda_{(3jk)}-\lambda_{(1jk)}\cdot\-{C_{(k)}}, \\
\big(A\cdot\lambda_{(511)}+B_{(1)}\cdot \lambda_{(411)}\big)\cdot
\gamma=&A\lambda_{(5jk)}+\lambda_{(4jk)}\cdot{B_{(j)}},\\
\big(A\cdot\lambda_{(611)}-{C_{(1)}\cdot \lambda_{(411)}}\big)\cdot
\gamma=&A\lambda_{(6jk)}-\lambda_{(4jk)}\cdot{C_{(j)}}.
\end{split}\end{equation}
Thus we  get from the first and the last equations in (\ref{717eq1})
the following relation over $ {\cal O}$:
\begin{equation}\begin{split}\label{59eq01}
&(\-A\cdot \lambda_{(2jk)}+\-{B_{(k)}}\cdot \lambda_{(1jk)})\cdot
(A\cdot
\lambda_{(611)}-C_{(1)}\cdot \lambda_{(411)})\\
&=(\-A\cdot \lambda_{(211)}+\-{B_{(1)}}\cdot \lambda_{(111)})\cdot
(A\cdot \lambda_{(6jk)}-C_{(j)}\cdot \lambda_{(4jk)})
\end{split}\end{equation}
After taking a limit, we see the equation in (\ref{59eq01}) holds
over $M$ near $0$.

\medskip
 Next, we solve (\ref{39eq2}) by the same
argument as that used to solve (\ref{39eq1}). In fact, we  get from
the first line of (\ref{39eq2}) that
$$
\-A\b_1=-\gamma\cdot \lambda_{(111)},\
A\a_1=\lambda_{(41k)}-\gamma\cdot \lambda_{(411)},\
\-A\b_k=\lambda_{(11k)}.
$$
Multiplying $\-A$ and $A$ to  the second and the third lines of the
equations in (\ref{39eq2}), respectively, and using the just
obtained relations, we have over $M\sm {\cal S}$
 \begin{align*}
&\-A\lambda_{(21k)}=\gamma\cdot \lambda_{(111)}\cdot
\-{B_{(1)}}- \lambda_{(11k)}\cdot\-{B_{(k)}}+\-A\gamma\cdot \lambda_{(211)},\\
&\-A\lambda_{(31k)}=-\gamma\cdot \lambda_{(111)}\cdot
\-{C_{(1)}}+\lambda_{(11k)}\cdot\-{C_{(k)}}+\-A\gamma\cdot \lambda_{(311)},\\
&A\lambda_{(51k)}=-
(\lambda_{(41k)}-\gamma\cdot \lambda_{(411)})\cdot {B_{(1)}}+A\gamma\cdot    \lambda_{(511)},\\
&A\lambda_{(61k)}=(\lambda_{(41k)}-\gamma\cdot \lambda_{(411)})\cdot
{C_{(1)}}+A\gamma\cdot \lambda_{(611)}.
\end{align*}
Rearranging the terms and replacing $\gamma$ by $\hat{\gamma}$, we
get over $M\sm {\cal S}$
\begin{equation}\begin{split}\label{717eq2}
\big(\-A\lambda_{(211)}+ \lambda_{(111)}\cdot\-{B_{(1)}}\big)\cdot
\hat{\gamma}=&\-A\lambda_{(21k)}+
\lambda_{(11k)}\cdot\-{B_{(k)}},\\
\big(\-A\lambda_{(311)}- \lambda_{(111)}\cdot\-{C_{(1)}}\big)\cdot
\hat{\gamma}=&\-A\lambda_{(31k)}-\lambda_{(11k)}\cdot\-{C_{(k)}}, \\
\big(A\lambda_{(511)}+ \lambda_{(411)}\cdot{B_{(1)}}\big)\cdot
\hat{\gamma}=&A\lambda_{(51k)}+\lambda_{(41k)}\cdot{B_{(1)}},\\
\big(A\lambda_{(611)}-\lambda_{(411)}\cdot{C_{(1)}}\big)\cdot
\hat{\gamma}=&A\lambda_{(61k)}-\lambda_{(41k)}\cdot{C_{(1)}}.
\end{split}\end{equation}
From the first and the last equations in (\ref{717eq2}), we see
that, after taking a limit, the following equation holds  near $0$:
\begin{equation}\begin{split}\label{59eq02}
&(\-A\cdot \lambda_{(21k)}+\-{B_{(k)}}\cdot \lambda_{(11k)})\cdot
(A\cdot \lambda_{(611)}-C_{(1)}\cdot \lambda_{(411)})\\&=(\-A\cdot
\lambda_{(211)}+\-{B_{(1)}}\cdot \lambda_{(111)})\cdot (A\cdot
\lambda_{(61k)}-C_{(1)}\cdot \lambda_{(41k)}).
\end{split}\end{equation}

Next we will examine $[L_1,T]$.  A direct computation shows that
\begin{equation}\begin{split}\label{0509eq1}
[L_1,T]=\Gamma_{(1)}\frac{\p}{\p \-z_1}+\Gamma_{(2)}\frac{\p}{\p
  \-z_n}+\Gamma_{(3)}\frac{\p}{\p \-w}+\Gamma_{(4)}\frac{\p}{\p
  z_1}+\Gamma_{(5)}\frac{\p}{\p z_n}+\Gamma_{(6)}\frac{\p}{\p
  w}.
\end{split}\end{equation}
where
\begin{equation}\begin{split}\label{1.25}
\Gamma_{(1)}&=A(\lambda_{(111)})_1-B_{(1)}(\lambda_{(111)})_n,\\
\Gamma_{(2)}&=A(\lambda_{(211)})_1-B_{(1)}(\lambda_{(211)})_n,\\
\Gamma_{(3)}&=A(\lambda_{(311)})_1-B_{(1)}(\lambda_{(311)})_n,\\
\Gamma_{(4)}&=A(\lambda_{(411)})_1-B_{(1)}(\lambda_{(411)})_n-\lambda_{(111)}A_{\-1}-\lambda_{(211)}
  A_{\-n}-\lambda_{(411)}A_1-\lambda_{(511)}A_n,\\
\Gamma_{(5)}&=A(\lambda_{(511)})_1-B_{(1)}(\lambda_{(511)})_n+\lambda_{(111)}
(B_{(1)})_{\-1}+\lambda_{(211)}(B_{(1)})_{\-n}
    +\lambda_{(411)}(B_{(1)})_1+\lambda_{(511)}(B_{(1)})_n,\\
\Gamma_{(6)}&=A(\lambda_{(611)})_{1}-B_{(1)}(\lambda_{(611)})_{n}
-\lambda_{(111)}(C_{(1)})_{\-1}-\lambda_{(211)}
(C_{(1)})_{\-n}-\lambda_{(411)}(C_{(1)})_{1}-\lambda_{(511)}(C_{(1)})_{n}.
\end{split}\end{equation}

 Suppose that over ${\cal O}$ we have
\begin{equation*}\begin{split}
[L_1,T]=&\sum\limits_{h=1}^{n-1}\left(\kappa_{(h)} L_h+\sigma_{(h)} \-{L_h}\right)+\tau T\\
=&\sum\limits_{h=1}^{n-1}\kappa_{(h)} \big(A\frac{\p}{\p
z_h}-B_{(h)}\frac{\p}{\p z_n}+C_{(h)}\frac{\p}{\p
    w}\big)+\sum\limits_{h=1}^{n-1}\sigma_{(h)} \big(\-A\frac{\p}{\p \-z_h}-\-{B_{(h)}}\frac{\p}{\p
    \-z_n}+\-{C_{(h)}}\frac{\p}{\p\-w}\big)\\
    &+\tau\big(\lambda_{(111)}\frac{\p}{\p \-z_1}+\lambda_{(211)}\frac{\p}{\p
  \-z_n}+\lambda_{(311)}\frac{\p}{\p \-w}+\lambda_{(411)}\frac{\p}{\p
  z_1}+\lambda_{(511)}\frac{\p}{\p z_n}+\lambda_{(611)}\frac{\p}{\p w}\big).
\end{split}\end{equation*}

Then combining this with (\ref{0509eq1}), we  get
$\kappa_{(h)}=0,\sigma_{(h)}=0$ for $h\neq 1$ and
\begin{equation}\begin{split}\label{39eq200}
&\Gamma_{(1)}=\-A\cdot \sigma_{(1)}+\tau\cdot \lambda_{(111)},\
\Gamma_{(4)}=A\cdot \kappa_{(1)}+\tau\cdot \lambda_{(411)},\\
&\Gamma_{(2)}=-\-{B_{(1)}}\cdot \sigma_{(1)}+\tau\cdot
\lambda_{(211)},\
\Gamma_{(3)}=\-{C_{(1)}}\cdot \sigma_{(1)}+\tau\cdot \lambda_{(311)},\\
&\Gamma_{(5)}=-{B_{(1)}}\cdot \kappa_{(1)}+\tau\cdot
\lambda_{(511)},\ \Gamma_{(6)}={C_{(1)}}\cdot \kappa_{(1)}+\tau\cdot
\lambda_{(611)}.
\end{split}\end{equation}
We  get from  the first line in (\ref{39eq200}) that
$$
\-A\sigma_{(1)}=\Gamma_{(1)}-\tau\cdot\lambda_{(111)},\
A\kappa_{(1)}=\Gamma_{(4)}-\tau\cdot \lambda_{(411)}.
$$
Multiplying $\-A$ and $A$ to  the second and the third lines of the
equations in (\ref{39eq200}), respectively, and using the just
obtained relations, we have over ${\cal O}$
\begin{equation*}\begin{split}
&\-A\Gamma_{(2)}=-\-{B_{(1)}}\cdot
(\Gamma_{(1)}-\tau\cdot\lambda_{(111)})+\-A\tau\cdot
\lambda_{(211)},\\
 &\-A\Gamma_{(3)}=\-{C_{(1)}}\cdot
(\Gamma_{(1)}-\tau\cdot\lambda_{(111)})+\-A\tau\cdot
\lambda_{(311)},\\
&A\Gamma_{(5)}=-{B_{(1)}}\cdot
(\Gamma_{(4)}-\tau\cdot\lambda_{(411)})+A\tau\cdot
\lambda_{(511)},\\
 &A\Gamma_{(6)}={C_{(1)}}\cdot
(\Gamma_{(4)}-\tau\cdot\lambda_{(411)})+A\tau\cdot \lambda_{(611)}.
\end{split}\end{equation*}
Rearranging the terms, we get over ${\cal O}$
\begin{equation}\begin{split}\label{717eq3}
\big(\-A\lambda_{(211)}+\lambda_{(111)}\cdot
     \-{B_{(1)}}\big)\cdot \tau=&\-A\Gamma_{(2)}+\Gamma_{(1)}\cdot \-{B_{(1)}},\\
\big(\-A\lambda_{(311)}- \lambda_{(111)}\cdot
     \-{C_{(1)}}\big)\cdot \tau=&\-A\Gamma_{(3)}-\Gamma_{(1)}\cdot
     \-{C_{(1)}},\\
\big(A\lambda_{(511)}+
     \lambda_{(411)}\cdot{B_{(1)}}\big)\cdot \tau=
     &A\Gamma_{(5)}+\Gamma_{(4)}\cdot {B_{(1)}},\\
\big(A\lambda_{(611)}- \lambda_{(411)}\cdot{C_{(1)}}\big)\cdot
    \tau=&A\Gamma_{(6)}-\Gamma_{(4)}\cdot
     {C_{(1)}}.
\end{split}\end{equation}
As before, from the first two equations in (\ref{717eq3}), we
obtain, near $0$, the following:
\begin{equation}  \begin{split}\label{717eq4}
  &(\-A\cdot\Gamma_{(2)}+\Gamma_{(1)}\cdot\-{B_{(1)}})\cdot (\-A\cdot \lambda_{(311)}-\lambda_{(111)}\cdot
  \-{C_{(1)}})\\
  =&(\-A\cdot\Gamma_{(3)}-\Gamma_{(1)}\cdot\-{C_{(1)}})
  \cdot (\-A\cdot \lambda_{{(211)}}+\lambda_{(111)}\cdot
  \-{B_{(1)}}).
\end{split}\end{equation}

At last, if both $[L_j,\-L_k]$ and $[L_1,T]$ are contained in the
span of $\{L_h,\-{L_h}\}_{1\leq h\leq n-1}$ and $T$. Then we have
$$
[L_k,T]=-\big[L_1,[\-{L_1},L_k]\big]-\big[\-{L_1},[L_k,L_1]\big]\in
\text{Span}\{\{L_h,\-{L_h}\}_{1\leq h\leq n-1},T\}.
$$

Summarizing the above, we have proved the following:

\begin{prop}\label{429prop1}
  Let $(M,0)$ be a $2n$-dimensional real manifold in $\mathbb{C}^{n+1}$
  defined by (\ref{429eq1}). Suppose $M$ is non-minimal at the CR
points near the origin. Then there is an open dense subset ${\cal
O}$ of $M$ near $0$ such that   the systems (\ref{717eq1}),
(\ref{717eq2}) and (\ref{717eq3}) are solvable over $ {\cal O}$ with
unknowns $\gamma$, $\hat{\gamma}$, $\tau$, respectively. In
particular, when $M$ is non-minimal at the CR points near the
origin,  we have (\ref{59eq01}), (\ref{59eq02}), and (\ref{717eq4})
near the origin.
\end{prop}

We mention that (\ref{59eq01}), (\ref{59eq02}),  and (\ref{717eq4})
are what we need for the proof of Theorem \ref{thmm1}.

\section{Derivation of three basic equations and  statement of Theorem \ref{thm1}}

Let $(M,0)$ be a $(2n)$-dimensional smooth real submanifold in
$\mathbb{C}^{n+1}$ defined by (\ref{429eq1}). Suppose that $M$ is
non-minimal at its CR points near $0$ and the order of $E(z,\-z)$ is
$m(\geq 3)$.  We first  study three basic relations for terms in
$E^{(m)}$, by making use of  (\ref{59eq01}), (\ref{59eq02}) and
(\ref{717eq4}).

By (\ref{326eq4}) and (\ref{326eq5}), we have
\begin{equation*}\begin{split}
\-A\cdot \lambda_{(2jk)}+\-{B_{(k)}}\cdot
\lambda_{(1jk)}&=(w_n+O(2))\cdot(-\-w_n\delta_{jk}+O(2))+(w_k+O(2))\cdot(-\-w_j+O(2))\\
&=-(|w_n|^2\cdot \delta_{jk}+\-w_jw_k)+O(3),\\
 A\cdot \lambda_{(6jk)}-C_{(j)}\cdot \lambda_{(4jk)}&=(\-w_{n}+O(2))\cdot
\lambda_{(6jk)}-C_{(j)}\cdot (w_k+O(2))\\
&=\-w_n\cdot \lambda_{(6jk)}-w_kC_{(j)}+O(m+2)\\
&=\-w_n\cdot (-2iw_n)\cdot (\-{\Phi_{(j)}})_{\-k}+\-w_n\cdot
2iw_k(\-{\Phi_{(j)}})_{\-n}
-2iw_k\cdot \-{\Phi_{(j)}}+O(m+2)\\
&=-2i\-{\Psi_{(jk)}}+O(m+2).
\end{split}\end{equation*}
Substituting these relations to (\ref{59eq01}), we get
\begin{equation}\begin{split}
&(|w_n|^2\cdot \delta_{jk}+\-w_jw_k)\cdot (-\-{\Psi_{(11)}}+O(m+2))+O(m+4)\\
&=(|w_n|^2+|w_1|^2)\cdot (-\-{\Psi_{(jk)}}+O(m+2))+O(m+4).
\end{split}\end{equation}
Hence we obtain
\begin{equation}\begin{split}\label{2}
&(|w_n|^2\cdot \delta_{jk}+\-w_jw_k)\cdot
\-{\Psi_{(11)}}=(|w_n|^2+|w_1|^2)\cdot \-{\Psi_{(jk)}}.
\end{split}\end{equation}

Notice that (\ref{59eq02}) is the same as (\ref{59eq01}) except that
$j$ is replaced by $1$. By the same method  as that used to handle
(\ref{59eq01}), we  get the following equation which is the same as
(\ref{2}) except that $j$ is replaced by $1$.
\begin{equation}\begin{split}\label{3}
&(|w_n|^2\cdot \delta_{1k}+\-w_1w_k)\cdot
\-{\Psi_{(11)}}=(|w_n|^2+|w_1|^2)\cdot \-{\Psi_{(1k)}}.
\end{split}\end{equation}

 We next  derive an equation from (\ref{717eq4}). From
(\ref{326eq4}), (\ref{326eq5}) and (\ref{1.25}), we obtain
\begin{equation}  \begin{split}\label{gamma}
\Gamma_{(1)}=&(\-w_n+O(2))\cdot(-2\lambda_1+O(1))-(\-w_1+O(2))\cdot
O(1)=-2\lambda_1\-w_n+O(2),\\
\Gamma_{(2)}=&(\-w_n+O(2))\cdot O(1)-(\-w_1+O(2))\cdot
(-2\lambda_n+O(1))=2\lambda_n\-w_1+O(2),\\
\frac{1}{2i}\Gamma_{(3)}=&(\-w_n+O(2))\cdot\big(
             -\-{w_n}\Phi_1+\-{w_1}\Phi_n+O(m+1) \big)_1\\
         &-(\-w_1+O(2))\cdot\big(
             -\-{w_n}\Phi_1+\-{w_1}\Phi_n+O(m+1) \big)_n\\
=&-\-{w_n}\big(\-{w_n}\Phi_1-\-{w_1}\Phi_n \big)_1
         +\-{w_1}\big(\-{w_n}\Phi_1-\-{w_1}\Phi_n\big)_n+O(m+1).
\end{split}\end{equation}
Hence we have
\begin{equation}  \begin{split} \label{0506eq3}
\-A\Gamma_{(2)}+\-B_{(1)}\Gamma_{(1)}&=(w_n+O(2))\cdot
(2\lambda_n\-w_1+O(2))+(w_1+O(2))\cdot
(-2\lambda_1\-w_n+O(2))\\
&=2\lambda_nw_n\-w_1-2\lambda_1w_1\-w_n+O(3).
\end{split}\end{equation}
From (\ref{326eq4}), (\ref{psiexp}) and (\ref{gamma}), we obtain
\begin{equation}\begin{split}\label{0506eq4}
&\frac{1}{2i}(\-A\Gamma_{(3)}-\Gamma_{(1)}\-{C_{(1)}})\\
=&\big(w_n+O(2)\big)\cdot
\Big\{-\-{w_n}\big(\-{w_n}\Phi_1-\-{w_1}\Phi_n \big)_1
         +\-{w_1}\big(\-{w_n}\Phi_1-\-{w_1}\Phi_n\big)_n+O(m+1)\Big\}\\
         &-\big(-2\lambda_1\-w_n+O(2)\big)\cdot\big(-\Phi+O(m+1)\big)\\
=&-\-{w_n}\big(w_n\-{w_n}\Phi_1-w_n\-{w_1}\Phi_n\big)_1
+\-{w_1}\big(w_n\-{w_n}\Phi_1-w_n\-{w_1}\Phi_n\big)_n\\
&-\-{w_1}\big(\-{w_n}\Phi_1-\-{w_1}\Phi_n\big)-2\lambda_1
\-{w_n}\Phi+O(m+2)\\
=&-\-{w_n}\big(\Psi-\-{w_1}\Phi\big)_1+\-{w_1}\big(\Psi-\-{w_1}\Phi\big)_n
-\-{w_1}\big(\-{w_n}\Phi_1-\-{w_1}\Phi_n\big)-2\lambda_1
\-{w_n}\Phi+O(m+2)\\
=&-\-{w_n}\Psi_1+\-{w_1}\Psi_n+O(m+2).
\end{split}\end{equation}
By (\ref{326eq4}) and (\ref{326eq5}), we get
\begin{equation}\begin{split}\label{0506eq1}
\frac{1}{2i}(\-A\lambda_{(311)}-\lambda_{(111)}\-{C_{(1)}})
=&(w_n+O(2))\cdot\big\{
-\-w_n\Phi_{1}+\-w_1\Phi_{n}+O(m+1)\big\}\\
&-(-\-w_1+O(2))\cdot \big\{-\Phi+O(m+1)\big\}\\
=&-w_n\-{w_n}\Phi_{1}+w_n\-{w_1}\Phi_{n}-\-w_1 \Phi
+O(m+2)\\
=&-\Psi+O(m+2),\\
\-A\lambda_{(211)}+\-{B_{(1)}}\lambda_{(111)}=&(w_n+O(2))\cdot
(-\-w_n+O(2))+(w_1+O(2))\cdot (-\-w_1+O(2))\\
=&-(|w_n|^2+|w_1|^2)+O(3).
\end{split}\end{equation}

Substituting  (\ref{0506eq3})-(\ref{0506eq1}) into (\ref{717eq4}),
we obtain
\begin{equation*}  \begin{split}
&-\big(|w_n|^2+|w_1|^2+O(3)\big)\cdot
\big\{-\-{w_n}\Psi_1+\-{w_1}\Psi_n+O(m+2)\big\}\\
&=\big(2\lambda_nw_n\-w_1-2\lambda_1w_1\-w_n+O(3)\big)\cdot
\big\{-\Psi+O(m+2)\big\}.
\end{split}\end{equation*}
Hence we get
\begin{equation}  \begin{split}\label{5}
  &\big(|w_n|^2+|w_1|^2\big)\cdot
\big(\-{w_n}\Psi_1-\-{w_1}\Psi_n\big)+\big(2\lambda_nw_n\-w_1-2\lambda_1w_1\-w_n\big)\cdot
\Psi=0.
\end{split}\end{equation}

\medskip
Now, for convenience of the reader, we put together, in the
following, the equations in (\ref{2}), (\ref{3}) and (\ref{5}), that
will be all we need to use to prove Theorem \ref{thmm1}:

\begin{equation} \label{add-new}
\begin{split}
&(|w_n|^2\cdot \delta_{1k}+\-w_1w_k)\cdot
\-{\Psi_{(11)}}=(|w_n|^2+|w_1|^2)\cdot \-{\Psi_{(1k)}}, \ \  1<k<n \\
&(|w_n|^2\cdot \delta_{jk}+\-w_jw_k)\cdot
\-{\Psi_{(11)}}=(|w_n|^2+|w_1|^2)\cdot \-{\Psi_{(jk)}},\  1,\ j, k< n\\
&\big(|w_n|^2+|w_1|^2\big)\cdot
\big(\-{w_n}\Psi_1-\-{w_1}\Psi_n\big)+\big(2\lambda_nw_n\-w_1-2\lambda_1w_1\-w_n\big)\cdot
\Psi=0 \ \ \hbox{with}\\
&\Psi_{(jk)}=w_n\-{w_n}
(\Phi_{(j)})_k-{w_n}\-{w_k}(\Phi_{(j)})_n+\-{w_k}\cdot \Phi_{(j)},\
\ \Psi=\Psi_{(11)},\
 \hbox{where}\\
&  \Phi_{(j)}=w_nH_{\-j}-w_jH_{\-n},\ \Phi=\Phi_{(1)}, \ H=E^{(m)},\
 \  w_l=z_l+2\lambda_l\-z_l\  \text{for}\ 1\leq l\leq n.\\
\end{split}
\end{equation}
\bigskip

Notice that when $n=2$, the first two equations in (\ref{add-new})
disappear and we only have the third one to use.

\medskip
  We will use (\ref{add-new}) to
prove the following theorem, which includes Theorem \ref{thmm1} as
its special case:

\begin{thm}\label{thm1}
Let $(M,0)$ be a $(2n)$-dimensional smooth real submanifold in
$\mathbb{C}^{n+1}$ defined by (\ref{429eq1}). Suppose that there
exists an $i\in [1,n]$ such that $\lambda_i\neq 1/2$. We further
suppose that $M$ is non-minimal at its CR points near $0$. Then for
any positive integer $N$, there exists a holomorphic transform of
{\it the special form} $(z,w)\rightarrow (z'=z,\ w'=w+o(|z|^2,w))$
such that in the new coordinates, $M$ is defined by an equation of
the form: $w'=\rho(z',\-z')$ with $\Im\rho$ vanishing at least to
the order $N$.
\end{thm}

\begin {rem}

Let $M$ be a formal $(2n)$-manifold in ${\mathbb C}^{n+1}$ near $0$
defined by a formal equation of the form $w=q(z,\-{z})+O(|z|^3)$.
Here,  as before,
$q(z,\-z)=\sum_{i=1}^{n}(|z_i|^2+\lambda_i(z_i^2+\-z_i^2))$ with
$0\le \lambda_1,\cdots, \lambda_n< \infty$. Then we can similarly
define the formal vector fields $\{L_1, \cdots, L_{n-1}, T\}$. We
call that $M$ is formally non-minimal  if (\ref{59eq01}),
(\ref{59eq02}), and (\ref{717eq4}) hold in the formal sense. Then
the exact proof for Theorem \ref{thm1} can be used to prove the
following:
\end {rem}

\begin{thm}\label{thm2}
Let $(M,0)$ be a $(2n)$-dimensional formal submanifold in
$\mathbb{C}^{n+1}$ defined by (\ref{429eq1}). Suppose that there
exists an $i\in [1,n]$ such that $\lambda_i\neq 1/2$. Further assume
that $M$ is formally non-minimal. Then for any positive integer $N$,
there exists a holomorphic transform of {\it the special form}
$(z,w)\rightarrow (z'=z,\ w'=w+o(|z|^2,w))$ such that in the new
coordinates, $M$ is defined by an equation of the form:
$w'=\rho(z',\-z')$ with $\Im\rho$ vanishing at least to the order
$N$.
\end{thm}

As a corollary, when $M$ is smooth near the non-degenerate CR
singular point $p=0$, defined by (\ref{429eq1}),  and when the set
$O$ of non-minimal CR points has $p=0$ in its closure, then one sees
that (\ref{59eq01}), (\ref{59eq02}), and (\ref{717eq4}) hold  in an
open subset of ${\mathbb C}^{n}$    that has $0$ in its boundary.
We see that they must hold for all $z$ in a neighborhood of $0\in
{\mathbb C}^{n}$ in the formal sense. Hence, we see that $M$ is
formally non-minimal as just defined. Thus we have the following:

\begin{cor}\label{thm3}
Let $(M,0)$ be a $(2n)$-dimensional smooth submanifold in
$\mathbb{C}^{n+1}$ defined by (\ref{429eq1}) near $0$. Suppose that
there exists an $i\in [1,n]$ such that $\lambda_i\neq 1/2$. Further
assume that the set of non-minimal CR points of $M$ forms an open
subset with $0$ in its closure . Then for any positive integer $N$,
there exists a holomorphic transform of {\it the special form}
$(z,w)\rightarrow (z'=z,\ w'=w+o(|z|^2,w))$ such that in the new
coordinates, $M$ is defined by an equation of the form:
$w'=\rho(z',\-z')$ with $\Im\rho$ vanishing at least to the order
$N$.
\end{cor}

We notice that under the hypothesis in  Corollary \ref{thm3}, when
$M$ is real analytic, it is easy to see that $M$ must be non-minimal
at any CR point near $p=0$. Hence, Corollary \ref{thm3} does not
give any new result in the real analytic category.

\section{Formal flattening  near a CR singular point,   Proof of Theorem \ref{thm1}---Part I}

Before reading $\S 4$-$\S 6$, the reader is suggested  to read the
Appendix in $\S 8$ for the proof in the special case when $n=2,
m=3$, to see  basic ideas behind all these complicated computations.

 We use
the notations and definitions set up so far for the proof of Theorem
\ref{thm1}. Due to the complicated nature of the argument, we
divided our proof into two parts. In this part, we give an initial
normalization by using biholomorphic change of coordinates without
involving the non-minimality at CR points.

 Throughout this and the next sections, we also set up the
following convention:
\begin{equation}\label{conv}
  E_{(I,J)}=0 \ \text{if one of the indices in $I$ or $J$ is negative.}
\end{equation}
For   quantities $a,b_1,\cdots,b_t$, we write
$$
a=\mathcal{F}\{b_1,\cdots,b_t\}\ \text{or}\
a=\mathcal{F}\{(b_j)_{1\leq j\leq t}\}
$$
if $a=\sum_{j=1}^{t}(c_jb_j+d_j\-{b_j})$. Here, when $b_j's$ are
complex numbers, we require that $c_j, d_j$ are complex numbers.
When $a, b_j$ are polynomials in $(z,\-{z})$, we require $c_j, d_j$
are polynomials in $(z,\-{z})$, too.

For $\S 4-\S 6$, we make the range of indices $j, k\in [2,n-1]$ if
$n\ge 3$. For any homogeneous polynomial $\chi(z,\-{z})$ of degree
$k\ge 1$, write
$$\chi=\sum_{\a\ge 0,\b\ge 0, |\a|+|\b|=k}H_{(\a,\b)}z^{\a}\-{z^{\b}}.$$

Set
\begin{equation}\label{xieta}
\begin{split}
&\xi=2\lambda_n,\ \eta=2\lambda_1,\ \theta=1-\xi^2,\\
&H_{[tsrh]}=H_{(te_n+se_1,re_n+he_1)}\ \text{for}\ t+s+r+h=m,\\
&\Phi_{[tsrh]}=\Phi_{(te_n+se_1,re_n+he_1)}\ \text{for}\ t+s+r+h=m,\\
&\Psi_{[tsrh]}=\Psi_{(te_n+se_1,re_n+he_1)}\ \text{for}\
t+s+r+h=m+1.
\end{split}
\end{equation}
Here $e_j$ is the $n$-tuple with its $j^{\rm{th}}$ element $1$ and
zero for the others. By (\ref{phiexp}), we have
\begin{equation}\label{atsrh}
\begin{split}
\Phi_{[tsrh]}=&\xi(h+1)H_{[ts(r-1)(h+1)]}+(h+1)H_{[(t-1)sr(h+1)]}\\
         &-(r+1)H_{[t(s-1)(r+1)h]}-\eta(r+1)H_{[ts(r+1)(h-1)]}.
\end{split}
\end{equation}
From (\ref{psiexp}), we obtain
\begin{equation}\label{btsrh}
\begin{split}
\Psi_{[tsrh]}=&(s+1)\big\{\xi\Phi_{[t(s+1)(r-2)h]}+(1+\xi^2)\Phi_{[(t-1)(s+1)(r-1)h]}
+\xi\Phi_{[(t-2)(s+1)rh]}\big\}\\
&-\xi(t+1)\Phi_{[(t+1)s(r-1)(h-1)]}
-t\Phi_{[tsr(h-1)]}-\xi\eta(t+1)\Phi_{[(t+1)(s-1)(r-1)h]}\\
& -\eta t\Phi_{[t(s-1)rh]}+\Phi_{[tsr(h-1)]}+\eta\Phi_{[t(s-1)rh]}.
\end{split}
\end{equation}
Notice that $\Phi_{[tsrh]}$ are $\Psi_{[t's'r'h']}$ are understood
as $0$ if one of their indices is negative.

Collecting the coefficients of
$z_n^tz_1^{s-1}\-{z_n}^{r+3}\-{z_1}^h$ for $t\geq 0$, $s\geq 1$,
$r\geq -3$ and $h=m+1-t-s-r\geq 0$ in (\ref{5}), we get
\begin{equation}\label{bequation}
\begin{split}
&s\big\{\xi
\Psi_{[tsrh]}+(2\xi^2+1)\Psi_{[(t-1)s(r+1)h]}+(\xi^3+2\xi)\Psi_{[(t-2)s(r+2)h]}
+\xi^2\Psi_{[(t-3)s(r+3)h]}\big\}\\
&+s\eta\big\{\Psi_{[ts(r+2)(h-2)]}+\xi
\Psi_{[(t-1)s(r+3)(h-2)]}\big\}
+(1+\eta^2)(s-1)\big\{\Psi_{[t(s-1)(r+2)(h-1)]}\\
&+\xi\Psi_{[(t-1)(s-1)(r+3)(h-1)]}\big\}
+(s-2)\eta\big\{\Psi_{[t(s-2)(r+2)h]}+\xi\Psi_{[(t-1)(s-2)(r+3)h]}\big\}\\
&-\big\{(t+1)\xi \Psi_{[(t+1)(s-1)(r+1)(h-1)]}
+t(1+\xi^2)\Psi_{[t(s-1)(r+2)(h-1)]}+(t-1)\xi \Psi_{[(t-1)(s-1)(r+3)(h-1)]}\big\}\\
&-\eta\big\{(t+1)\xi
\Psi_{[(t+1)(s-2)(r+1)h]}+t(1+\xi^2)\Psi_{[t(s-2)(r+2)h]}
+(t-1)\xi \Psi_{[(t-1)(s-2)(r+3)h]}\big\}\\
&-(t+1)\big\{\eta \Psi_{[(t+1)(s-1)(r+3)(h-3)]}+(2\eta^2+1)\Psi_{[(t+1)(s-2)(r+3)(h-2)]}\\
&     +(\eta^3+2\eta)\Psi_{[(t+1)(s-3)(r+3)(h-1)]}+\eta^2\Psi_{[(t+1)(s-4)(r+3)h]}\big\}\\
&+\xi\big\{\xi
\Psi_{[t(s-1)(r+2)(h-1)]}+\Psi_{[(t-1)(s-1)(r+3)(h-1)]}\big\}
+\xi\eta\big\{\xi \Psi_{[t(s-2)(r+2)h]}+\Psi_{[(t-1)(s-2)(r+3)h]}\big\}\\
&-\eta\big\{\eta \Psi_{[t(s-1)(r+2)(h-1)]}+\xi\eta
\Psi_{[(t-1)(s-1)(r+3)(h-1)]}+\Psi_{[t(s-2)(r+2)h]}+\xi
\Psi_{[(t-1)(s-2)(r+3)h]}\big\}=0.
\end{split}
\end{equation}

Notice that (\ref{bequation}) takes  the following form:
\begin{equation*}
\begin{split}
s\big\{&\xi \Psi_{[tsrh]}+(2\xi^2+1)\Psi_{[(t-1)s(r+1)h]}+(\xi^3+2\xi)\Psi_{[(t-2)s(r+2)h]}\\
&+\xi^2\Psi_{[(t-3)s(r+3)h]}\big\}+
\mathcal{F}\{(\Psi_{[t's'r'h']})_{s'+h'\leq s+h-2,s'\leq s,h'\leq
h}\}=0.
\end{split}
\end{equation*}
Thus for $s\geq 1$, by keeping use this property, we can inductively
get
\begin{equation}\label{beq}
\begin{split}
\Psi_{[tsrh]}=\mathcal{F}\{(\Psi_{[t's'r'h']})_{s'+h'\leq
s+h-2,s'\leq s,h'\leq h}\}.
\end{split}
\end{equation}
Substituting (\ref{btsrh}) into (\ref{beq}), we get, for $s\geq 1$,
the following
\begin{equation*}
\begin{split}
&(s+1)\big\{\xi\Phi_{[t(s+1)(r-2)h]}+(1+\xi^2)\Phi_{[(t-1)(s+1)(r-1)h]}
+\xi\Phi_{[(t-2)(s+1)rh]}\big\}\\
&=\mathcal{F}\{(\Phi_{[t's'r'h']})_{s'+h'\leq s+h-1,s'\leq
s+1,h'\leq h}\}.
\end{split}
\end{equation*}
Hence for $s\geq 2$, we can inductively obtain
\begin{equation}\label{111}
\begin{split}
\Phi_{[tsrh]}=\mathcal{F}\{(\Phi_{[t's'r'h']})_{s'+h'\leq
s+h-2,s'\leq s,h'\leq h}\}.
\end{split}
\end{equation}
Substituting (\ref{atsrh}) into (\ref{111}), we get, for $s\geq 2$
and $h\geq0$, the following
\begin{equation*}
\begin{split}
\xi(h+1)H_{[ts(r-1)(h+1)]}+(h+1)H_{[(t-1)sr(h+1)]}=\mathcal{F}\{(H_{[t's'r'h']})_{s'+h'\leq
s+h-1,s'\leq s,h'\leq h+1}\}.
\end{split}
\end{equation*}
Hence for $s\geq 2$ and $h\geq 1$, we inductively get that
\begin{equation*}\begin{split}
H_{[ts(m-t-s-h)h]}=\mathcal{F}\big\{(H_{[t's'(m-t'-s'-h')h']})_{s'+h'\leq
s+h-2,s'\leq s,h'\leq h}\big\}.
\end{split}\end{equation*}
Notice that $H_{[tsrh]}=\-{H_{[rhts]}}$.  Keeping applying the above
until the assumption that $s\ge 2$ and $h\ge 1$ do not hold anymore,
  we can  inductively get the following crucial formula:
\begin{equation}\begin{split}\label{111ind}
H_{[ts(m-t-s-h)h]}=\mathcal{F}\big\{(H_{[t'1(m-t'-2)1]})_{1\leq
t'\leq m-2},(H_{[t'0(m-t'-i)i]})_{i\leq \max(s,h),0\leq t'\leq
m-i}\big\}.
\end{split}\end{equation}

Substituting (\ref{phiexp}) and (\ref{psiexp}) into (\ref{3}), we
get the following equation:
\begin{equation}\begin{split}\label{33}
&\-w_1w_k\cdot(|w_n|^2H_{1\-1}-w_n\-w_1H_{n\-1}-w_1\-w_nH_{1\-n}+|w_1|^2H_{n\-n})\\
&=(|w_n|^2+|w_1|^2)\cdot(|w_n|^2H_{1\-k}-w_n\-w_1H_{n\-k}-w_k\-w_nH_{1\-n}+w_k\-w_1H_{n\-n}).
\end{split}\end{equation}
Notice that it takes  the form
\begin{equation*}
\begin{split}
-|w_n|^4H_{1\-k}+\sum\limits_{i_n+j_n\leq
3}z_1^{i_1}z_k^{i_k}z_n^{i_n}\-{z_1}^{j_1}\-{z_k}^{j_k}\-{z_n}^{j_n}\frac{
\p^{h_1}}{\p z_1^{h_1}}\frac{ \p^{h_k}}{\p z_k^{h_k}}\frac{
\p^{h_n}}{\p z_n^{h_n}}\frac{ \p^{l_1}}{\p \-{z_1}^{l_1}}\frac{
\p^{l_k}}{\p \-{z_k}^{l_k}}\frac{ \p^{l_n}}{\p \-{z_1}^{l_n}}H=0,
\end{split}
\end{equation*}
where $i_1+i_k+i_n+j_1+j_k+j_n-(h_1+h_k+h_n+l_1+l_k+l_n)=2$. Hence
we get
\begin{equation}\begin{split}\label{1kind}
H_{(te_n+e_1+I,re_n+e_k+J)}=\mathcal{F}\big(\{H_{(t'e_n+I',r'e_n+J')}\}_{t'+r'>t+r
}\big).
\end{split}\end{equation}

Similarly, substituting (\ref{phiexp}) and (\ref{psiexp}) into
(\ref{2})  and setting $j=k(\neq 1)$, we get the following equation:
\begin{equation}\begin{split}\label{11}
&(|w_n|^2+|w_k|^2)\cdot(|w_n|^2H_{1\-1}-w_n\-w_1H_{n\-1}-w_1\-w_nH_{1\-n}+|w_1|^2H_{n\-n})\\
&=(|w_n|^2+|w_1|^2)\cdot(|w_n|^2H_{k\-k}-w_n\-w_kH_{n\-1}-w_1\-w_nH_{k\-n}+|w_k|^2H_{n\-n}).
\end{split}\end{equation}
Similar to (\ref{1kind}), for any fixed $s,h\geq 1$,  we  get
\begin{equation}\begin{split}\label{22}
&shH_{(te_n+se_k,re_n+he_k)}-H_{(te_n+(s-1)e_k+e_1,re_n+(h-1)e_k+e_1)}\\
&+\mathcal{F}\{(H_{(t'e_n+I,r'e_n+J)})_{t'+r'>t+r}\}=0.
\end{split}\end{equation}

Next we prove the following lemma, which is only needed for $n\geq
3$.

\begin{lem}\label{lem1} Suppose that $n\geq 3$. For any given $j$ with $j\geq 1$ and any given
$I=(i_1,\cdots,i_n)$ with $i_1=i_k=i_n=0$, suppose that
$H_{(te_n+se_1,re_n+he_1+I+(j'-2)e_k)}=0$ for
 all $t,s,r,h\geq 0$, $j'\leq j$ and $t+s+r+h=m+2-|I|-j'$. Then
 \begin{equation}\begin{split}
&H_{(te_n+se_1,re_n+he_1+I+je_k)}
=\mathcal{F}\{(H_{(t'e_n,r'e_n+h'e_1+I+je_k)})_{r'\geq t'}\},\\
&\text{where}\ t+s+r+h=t'+r'+h'=m-|I|-j.
\end{split}\end{equation}
\end{lem}

\begin{proof}[Proof of Lemma \ref{lem1}]
Set
\begin{equation*}  \begin{split}
P^{(l)}&=\{\text{homogeneous  polynomials of degree}\ l\}\ \text{and}\\
P^{(l)}_{(1n\-1\-n)}&=\{\text{homogeneous  polynomials   of degree
$l$ in $(z_1,z_n,\-{z_1}, \-{z_n})$} \}.
\end{split}\end{equation*}

In (\ref{33}), the coefficients of  terms other than
$(|w_n|^2+|w_1|^2)(|w_n|^2H_{1\-k}-w_n\-{w_1}H_{n\-k})$, when
projected to the space of  polynomials of the form:
$\-z^{I+(j-1)e_k}P^{(m+3-|I|-j)}_{(1n\-1\-n)}$, is a linear
combination of $H_{(t'e_n+s'e_1,r'e_n+h'e_1+I+(j'-2)e_k)}$ with
$j'\leq j$, which are $0$ by our assumption. Here and in what
follows we equip the space of polynomials in $(z,\-{z})$ with
$\{z^\a\-{z^\b}\}$ as an ortho-normal
 basis. Hence by considering
terms projected to the space
$\-z^{I+(j-1)e_k}P^{(m+3-|I|-j)}_{(1n\-1\-n)}$ in (\ref{33}), we get
$$
|w_n|^2(|w_n|^2+|w_1|^2)H_{1\-k}-w_n\-{w_1}(|w_n|^2+|w_1|^2)H_{n\-k}=0\
\text{mod}\left(
\{\-z^{I+(j-1)e_k}P^{(m+3-|I|-j)}_{(1n\-1\-n)}\}^c\right).
$$
Here for a subspace $A$ of the space of polynomials, we write $A^c$
for its compliment.
 Namely, we have
\begin{equation}\label{326eq3}
(\-z_n+2\lambda_nz_n)H_{1\-k}-(\-z_1+2\lambda_1z_1)H_{n\-k}=0\
\text{mod}\left(
\{\-z^{I+(j-1)e_k}P^{(m-|I|-j)}_{(1n\-1\-n)}\}^c\right ).
\end{equation}
Considering the coefficients of
$z_1^{s-1}{\-z}^{I+(j-1)e_k}\-z_1^{h}{\-z_n}^{t+r+1}$ and
$z_1^{s-1}z_n^t{\-z}^{I+(j-1)e_k}\-z_1^{h}{\-z_n}^{r+1}$,
respectively, with $r=m-t-s-|I|-j-h$, $t\geq 0$, $h\geq 0$, $s\geq
1$ in (\ref{326eq3}), we get
\begin{equation}\begin{split}\label{227eq3}
&sH_{(se_1,((t+r)e_n+he_1+I+je_k)}\\
=&H_{(e_n+(s-1)e_1,(t+r+1)e_n+(h-1)e_1+I+je_k)}
+2\lambda_1H_{(e_n+(s-2)e_1,(t+r+1)e_n+he_1+I+je_k)}, \ \text{and}\\
  &s\big(H_{(te_n+se_1,re_n+he_1+I+je_k)}+2\lambda_n
  H_{((t-1)e_n+se_1,(r+1)e_n+he_1+I+je_k)}\big)\\
=&(t+1)\big(H_{((t+1)e_n+(s-1)e_1,(r+1)e_n+(h-1)e_1+I+je_k)}
+2\lambda_1H_{((t+1)e_n+(s-2)e_1,(r+1)e_n+he_1+I+je_k)}\big).
\end{split}\end{equation}
Hence for $s\geq 1$, we obtain
\begin{equation}\begin{split}\label{29eq3}
H_{(te_n+se_1,re_n+he_1+I+je_k)}=\mathcal{F}
\{&(H_{(t'e_n+(s-1)e_1,r'e_n+(h-1)e_1+I+je_k)})_{r'-t'\geq r-t},\\
&( H_{(t'e_n+(s-2)e_1,r'e_n+he_1+I+je_k)})_{r'-t'\geq r-t}\}.
\end{split}\end{equation}

Next we prove by induction that
\begin{equation}\begin{split}\label{226eq4}
H_{(te_n+se_1,re_n+he_1+I+je_k)}=\mathcal{F}\{(H_{(t'e_n,r'e_n+h'e_1+I+je_k)})_{r'-t'\geq
r-t}\}.
\end{split}\end{equation}

In fact, the claim holds automatically  for $s=0$. If $s=1$,
(\ref{226eq4}) follows from (\ref{29eq3}). Now we suppose that
(\ref{226eq4}) holds for $s<s_0$, we can get by (\ref{29eq3}) that
\begin{equation}\begin{split}\label{29eq05}
H_{(te_n+s_0e_1,re_n+he_1+I+je_k)}&=\mathcal{F}
\{(H_{(t'e_n+(s_0-1)e_1,r'e_n+(h-1)e_1+I+je_k)})_{r'-t'>r-t},\\
&\hskip 1.8cm ( H_{(t'e_n+(s_0-2)e_1,r'e_n+he_1+I+je_k)})_{r'-t'\geq r-t}\}\\
&=\mathcal{F}\big\{(H_{(t'e_n,r'e_n+h'e_1+I+je_k)})_{r'-t'\geq
r-t}\}.
\end{split}\end{equation}
The last equality follows from our assumption. Hence (\ref{226eq4})
also holds for $s_0$. This finishes the proof of (\ref{226eq4}).

By interchanging the role of $z_1$ and $z_n$ in (\ref{226eq4}), we
can get
\begin{equation}\label{226eq04}
H_{(te_n+se_1,re_n+he_1+I+je_k)}
=\mathcal{F}\{(H_{(s'e_1,r'e_n+h'e_1+I+je_k)})_{h'-s'\geq h-s}\}.
\end{equation}
As a special case of (\ref{226eq04}) or (\ref{226eq4}), we obtain
the following:
\begin{equation}\begin{split}\label{0509eq3}
&H_{(t'e_n,r'e_n+h'e_1+I+je_k)}=\mathcal{F}\{(H_{(s''e_1,r''e_n+h''e_1+I+je_k)})_{h''-s''\geq h'}\}.\\
&H_{(s''e_1,r''e_n+h''e_1+I+je_k)}=\mathcal{F}\{(H_{(t'''e_n,r'''e_n+h'''e_1+I+je_k)})_{r'''-t'''\geq
r''}\}.
\end{split}\end{equation}

Now we  conclude from (\ref{226eq4}) and (\ref{0509eq3}) that
\begin{equation*}
H_{(te_n+se_1,re_n+he_1+I+je_k)}
=\mathcal{F}\{(H_{(t'''e_n,r'''e_n+h'''e_1+I+je_k)})_{r'''-t'''\geq
0}\}.
\end{equation*}
This completes the proof of Lemma \ref{lem1}.\end{proof}

For the rest of this section, for simplicity of notation, we assume
that $\lambda_n$ is the smallest non-parabolic Bishop invariant,
namely, the smallest one that is not equal to $ \frac{1}{2}$.
Then we have the following normalization for $E(z,\-z)$. We notice
that the following result holds in general even without assuming the
non-minimality condition at CR points.  Also, in this result, there
is no need
 to assume that $\l_n\not = \frac{1}{2}$.

\begin{thm}\label{norm}
 For any given $l\geq 3$, there exists a holomorphic transformation
near the origin $(z,w)\rightarrow (z'=z,w'=w+o(|z|^2,w))$ such that
in the new coordinates, the $E(z,\-z)$ defined in (\ref{429eq1})
satisfies
\begin{equation}\label{ng}
  E_{(I,0)}=E_{(te_n+J,se_n)}=0\ \text{for}\ t\geq s,\ |J|\neq 0,\ |I|=t+s+|J|\leq
  l.
\end{equation}
Moreover, we have the following  normalizations:\\
 (I) When $\lambda_n=0$, we have
\begin{equation}\begin{split}\label{lambda0}
E_{(te_n,se_n)}=0 \ \text{for}\ t\geq s.
\end{split}\end{equation}
 (II) When $\lambda_n\neq 0$, for any $m_0\leq l$, the
normalization is divided into the following six
cases:\\

(II$_{-3}$)
 If $m_0=6\hat{m}-3$, then we have
\begin{equation}\begin{split}\label{n-3}
&E_{(te_n,se_n)}=0\ \text{for}\  4\hat{m}-1\leq t\leq m_0-1,\\
&E_{((2t+1)e_n+e_1,(m_0-2t-3)e_n+e_1)}=0\ \text{for}\ 2\hat{m}-2\leq
t\leq 3\hat{m}-3.
\end{split}\end{equation}

(II$_{-2}$) If $m_0=6\hat{m}-2$, then we have
\begin{equation}\begin{split}\label{n-2}
&E_{(te_n,se_n)}=0\ \text{for}\  4\hat{m}-1\leq t\leq m_0-1,\\
&E_{((2t+1)e_n+e_1,(m_0-2t-3)e_n+e_1)}=0\ \text{for}\ 2\hat{m}-1\leq
t\leq 3\hat{m}-3,\\
&{{\Re}}E_{((4\hat{m}-3)e_n+e_1,(2\hat{m}-1)e_n+e_1)}=0.
\end{split}\end{equation}

(II$_{-1}$)
 If $m_0=6\hat{m}-1$, then we have
\begin{equation}\begin{split}\label{n-1}
&E_{(te_n,se_n)}=0\ \text{for}\  4\hat{m}\leq t\leq m_0-1,\\
&E_{((2t+1)e_n+e_1,(m_0-2t-3)e_n+e_1)}=0\ \text{for}\ 2\hat{m}-1\leq
t\leq 3\hat{m}-2.
\end{split}\end{equation}

(II$_{0}$) If $m_0=6\hat{m}$, then we have
\begin{equation}\begin{split}\label{n0}
&E_{(te_n,se_n)}=0\ \text{for}\  4\hat{m}+1\leq t\leq m_0-1,\\
&E_{((2t+1)e_n+e_1,(m_0-2t-3)e_n+e_1)}=0\ \text{for}\ 2\hat{m}-1\leq
t\leq 3\hat{m}-2,\\
&{{\Re}}E_{(4\hat{m}e_n,2\hat{m}e_n)}=0.
\end{split}\end{equation}

(II$_{1}$)
 If $m_0=6\hat{m}+1$, then we have
\begin{equation}\begin{split}\label{n1}
&E_{(te_n,se_n)}=0\ \text{for}\  4\hat{m}+1\leq t\leq m_0-1,\\
&E_{((2t+1)e_n+e_1,(m_0-2t-3)e_n+e_1)}=0\ \text{for}\  2\hat{m}\leq
t\leq 3\hat{m}-1.
\end{split}\end{equation}

(II$_{2}$) If $m_0=6\hat{m}+2$, then we have
\begin{equation}\begin{split}\label{n2}
&E_{(te_n,se_n)}=0\ \text{for}\  4\hat{m}+2\leq t\leq m_0-1,\\
&E_{((2t+1)e_n+e_1,(m_0-2t-3)e_n+e_1)}=0\ \text{for}\  2\hat{m}\leq
t\leq 3\hat{m}-1,\\
&{{\Re}}E_{((4\hat{m}+1)e_n,(2\hat{m}+1)e_n)}=0.
\end{split}\end{equation}


\end{thm}

\begin{proof}[Proof of Theorem \ref{norm}]

Suppose that $z'=z,\ w'=w+B(z,w)$ transforms
$w=q(z,\-z)+p(z,\-z)+iE(z,\-z)$ to
$w'=q(z',\-z')+p'(z',\-z')+iE'(z',\-z')$, where $p'(z',\-z')$ and
$E'(z',\-z')$ are real valued and both of their orders are at least
three. Then
\begin{equation}
    q(z,\-z)+p(z,\-z)+iE(z,\-z)+B(z,w)=q(z,\-z)+p'(z,\-z)+iE'(z,\-z).
\end{equation}
Hence we get
\begin{equation}
    {\Im}(B(z,w))=E'(z,\-z)-E(z,\-z).
\end{equation}
Set
$$
B^{(m_0)}(z,w)=\sum\limits_{|I|+2j=m_0} b_{(Ij)}z^Iw^{j}.
$$
We further normalize $B(z,\-z)$ such that
Re$(b_{(0\frac{m_0}{2})})=0$ if $m_0$ is even. Then the real
dimension of  $B^{(m_0)}$ is
 \begin{equation}\begin{split}\label{bdim}
&2\cdot \sharp\big\{(i_1,\cdots,i_{n},j)\in \mathbb{R}^{n+1}:\
i_1,\cdots,i_{n},j\geq 0,\  i_1+\cdots+i_{n}+2j=m_0,\ 2j\neq m_0
\big\}+\delta_{m_0}\\
=&2\cdot \sharp\big\{(i_1,\cdots,i_{n},j)\in \mathbb{R}^{n+1}:\
I'\neq 0,\ i_1+\cdots+i_{n}+2j=m_0
\big\}+2[\frac{m_0+1}{2}]+\delta_{m_0}.
\end{split}\end{equation}
Here $\delta_{m_0}=1$ when $m_0$ is even and $0$, otherwise.\\
 The
dimension of the term $\sum\limits_{I'\neq
0,|I|+2j=m_0}a_{(Ij)}(z')^{I'}z_n^{i_n}|z_n|^{2j}$ is
\begin{equation}\begin{split}\label{in0dim}
&2\cdot \sharp\big\{(i_1,\cdots,i_{n},j)\in \mathbb{R}^{n+1}:\
i_1,\cdots,i_{n},j\geq 0,\ I'\neq 0,\ i_1+\cdots+i_{n}+2j=m_0\}.
\end{split}\end{equation}

{\bf (I)} Assume that $\lambda_n=0$. Set
$$
\hat{P}^{(m_0)}=\big\{\text{polynomials  of the form}\
2{\Re}\sum_{|I|+2j=m_0}
    a_{(Ij)}z^I|z_n|^{2j},\ {\Im}(a_{0,[m_0/2]})=0\ \text{for $m_0$ even}\big\}.
$$
To get the normalization condition (\ref{ng}) and (\ref{lambda0}),
we only need to prove that
\begin{equation}\label{227eq4}
{\Im}\big(B^{(m_0)}(z,q(z,\-z))\big)\big|_{\hat{P}^{(m_0)}}=Q^{(m_0)}(z,\-z)
\end{equation}
is solvable for any  $Q^{(m_0)}(z,\-z)\in \hat{P}^{(m_0)}$. Notice
that $\hat{P}^{(m_0)}$ and the space  $\{B^{(m_0)}(z,q(z,\-{z}))\}$
have the same dimension. Here, we recall that for a polynomial $A$
and a subspace of polynomials, we write $A|_P$ for the projection of
$A$ to $P$. Hence to prove (\ref{227eq4}), we need to show that
$$
 {\Im}\big(B^{(m_0)}(z,{q}(z,\-z))\big)\big|_{\hat{P}^{(m_0)}}=0\
  \hbox{and}\ {\Re}(b_{(0\frac{m_0}{2})})=0\
 \text{for $m_0$  even}\Longleftrightarrow B=0.
$$

By considering the coefficients of terms involving only $z_n$ and
$\-z_n$, we get
$$
{\Im}\big({\sum}_{i+2j=m_0}b_{(0i_nj)} z_n^{i_n}|z_n|^{2j}\big)=0.
$$
Thus we get $b_{(0i_nj)}=0$. Suppose that $b_{(I'i_nj)}=0$ for
$|I'|\leq k_0$. Considering  terms of the form:
$z^{I'}z_n^{i_n}|z_n|^j$ with $|I'|=k_0+1$, we get
$$
{\sum}_{|I'|=k_0+1,i+2j=m}b_{(I'i_nj)}z'^{I'} z_n^{i_n}|z_n|^{2j}=0,
$$
from which it follows that $b_{(I'i_nj)}=0$. Thus we get
$B^{(m_0)}(z,\-z)=0$.
\bigskip

{\bf (II)} Assume that $\lambda_n\neq 0$. Write $\w{P}$ for the
space of polynomials of the form:
$|z_1|^2P_1(z_n,\-z_n)+P_2(z_n,\-z_n)$. Then
\begin{equation}
\begin{split}
{\Im}B(z,q(z,\-z))\big|_{\w{P}}=&{\Im}B(0,z_n,|z_n|^2+\lambda_nz_n^2
+\lambda_n\-z_n^2+|z_1|^2)\big|_{\w{P}}\\
=&\sum\limits_{h+2j=m_0}
\frac{1}{2i}(b_{(0hj)}z_n^h-\-{b_{(0hj)}}\-z_n^h)(|z_n|^2+\lambda_nz_n^2
+\lambda_n\-z_n^2+|z_1|^2)^j\big|_{\w{P}}\\
=&\sum\limits_{h+2j=m_0}
\frac{1}{2i}{b_{(0hj)}}z_n^h(|z_n|^2+\lambda_nz_n^2+\lambda_n\-z_n^2+|z_1|^2)^j\big|_{\w{P}}\\
&-\sum\limits_{h+2j=m_0}
\frac{1}{2i}\-{{b_{(0hj)}}}\-{z_n}^h(|z_n|^2+\lambda_nz_n^2+\lambda_n\-z_n^2+|z_1|^2)^j\big|_{\w{P}}\\
:=&\w{I}\big|_{\w{P}}-\w{J}\big|_{\w{P}}.
\end{split}
\end{equation}
Here, for a subspace $A$ of the space of polynomials and for a
polynomial $X$, we write $X|_A$ for the projection of $X$ to $A$.
Write
\begin{equation}\label{ij0}
\begin{split}
I_{kl}=\sum\limits_{h+2j=m_0,j\geq
k+l}\frac{1}{2i}{b_{(0hj)}}(_{j-k-l}^j)\lambda_n^{j-k}(_l^{k+l}),\
J_{kl}=\sum\limits_{h+2j=m_0,j\geq
k+l}\frac{1}{2i}\-{{b_{(0hj)}}}(_{j-k-l}^j)\lambda_n^{j-k}(_l^{k+l}).
\end{split}
\end{equation}
Then we have
\begin{equation}\label{ij}
\begin{split}
J_{kl}=-\-{I_{kl}},\ I_{kl}=(_k^{k+l})\lambda_n^lI_{k+l,0}.
\end{split}
\end{equation}
A direct computation shows that
\begin{equation}
\begin{split}
\w{I}\big|_{\w{P}}=&\sum\limits_{h+2j=m_0}\frac{1}{2i}{b_{(0hj)}}
z_n^h\sum\limits_{0\leq k+l\leq
j}(_{j-k-l}^j)(\lambda_nz_n^2)^{j-k-l}
              (_l^{k+l})(\lambda_n \-z_n^2)^l|z_n|^{2k}\\
              &+\sum\limits_{h+2j=m_0}\frac{1}{2i}{b_{(0hj)}}z_n^h
              \sum\limits_{0\leq k+l\leq j}(_{j-k-l}^j)(\lambda_nz_n^2)^{j-k-l}
              (_l^{k+l})(\lambda_n \-z_n^2)^l|z_n|^{2(k-1)}\cdot
              k|z_1|^2\\
              =&\sum\limits_{0\leq k+2l\leq m_0,k+l\leq \frac{m_0}{2}}I_{kl}z_n^{m_0-k-2l}\-z_n^{k+2l}
                +\sum\limits_{k\geq 1,0\leq k+2l\leq m_0,\atop{k+l\leq \frac{m_0}{2}}}kI_{kl}|z_1|^2
                z_n^{m_0-k-2l-1}\-z_n^{k+2l-1}.
\end{split}
\end{equation}
Similarly, we have
\begin{equation}
\begin{split}
\w{J}\big|_{\w{P}}=&\sum\limits_{0\leq k+2l\leq m\atop{k+l\leq
\frac{m_0}{2}}}J_{kl}z_n^{k+2l}\-{z_n}^{m_0-k-2l}
                +\sum\limits_{k\geq 1,0\leq k+2l\leq m_0\atop{k+l\leq \frac{m_0}{2}}}kJ_{kl}|z_1|^2
                z_n^{k+2l-1}\-{z_n}^{m_0-k-2l-1}.
\end{split}
\end{equation}
Hence the coefficients of $z_n^t\-z_n^s(t\geq s, t+s=m_0)$ and
$z_n^{t-1}\-z_n^{s-1}|z_1|^2(t\geq s, t+s=m_0)$ in
Im$(B(z,q(z,\-z)))$ are, respectively, the following:
\begin{equation}\label{00coe}
\sum\limits_{k+2l=s}I_{kl}-\sum\limits_{k+2l=t\atop {k+l\leq
m_0/2}}J_{kl}\ \ \ \text{and}\ \ \
\sum\limits_{k+2l=s}kI_{kl}-\sum\limits_{k+2l=t\atop {k+l\leq
m_0/2}}kJ_{kl}.
\end{equation}

{\bf (II$_{-3}$)}: In this case, we have set $m_0=6\hat{m}-3$. Write
 \begin{equation*}\begin{split}
\hat{P}_{-3}^{(m_0)}=\big\{&\text{polynomials of the form}
\sum\limits_{t\geq
4\hat{m}-1}a_tz_n^t\-z_n^{m_0-t}\\&+\sum\limits_{t\geq
2\hat{m}-2}b_tz_n^{2t+1}\-z_n^{m_0-2t-3}|z_1|^2+\sum\limits_{I'\neq
0}C_{I't}(z')^{I'}z_n^{m-|I'|-2t}|z_n|^{2t}\big\}.
\end{split}\end{equation*}
To get the normalization condition (\ref{ng}) and (\ref{n-3}), we
only need to prove that
\begin{equation}\label{227eq-3}
{\Im}\big(B^{(m_0)}(z,q(z,\-z))\big)\big|_{\hat{P}_{-3}^{(m_0)}}=Q^{(m_0)}(z,\-z)
\end{equation}
is solvable for any real valued polynomial $Q^{(m_0)}(z,\-z)\in
\hat{P}_{-3}^{(m_0)}$.
 Notice that the dimension of the space of polynomials of the form:
 $\sum\limits_{t\geq 4\hat{m}-1}a_tz_n^t\-z_n^{m_0-t}+\sum\limits_{t\geq
2\hat{m}-2}b_tz_n^{2t+1}\-z_n^{m_0-2t-3}|z_1|^2$ is
 \begin{equation}\begin{split}
2\big(6\hat{m}-3-(4\hat{m}-2)\big)+2\big(\frac{6\hat{m}-6}{2}-(2\hat{m}-3)\big)=6\hat{m}-2
=2\big[\frac{6\hat{m}-2}{2}\big].
\end{split}\end{equation}
Combining this with (\ref{bdim}) and (\ref{in0dim}), we know that
the space $\{B^{(m_0)}(z,q(z,\-{z}))\}$ and the space
$\hat{P}_{-3}^{(m_0)}$ have the same dimension. Hence to prove
(\ref{227eq-3}),  we  need to show that $B^{(m_0)}=0$ if
\begin{equation}\label{-3e00}
\begin{split}
 {\Im}\big(B^{(m_0)}(z,{q}(z,\-z))\big)\big|_{\hat{P}_{-3}^{(m_0)}}=0.
\end{split}
\end{equation}

By (\ref{00coe}), the condition (\ref{n-3}) gives  that
\begin{equation}\label{-3e}
\begin{split}
&\sum\limits_{k+2l=2t-1}kI_{kl}=\sum\limits_{k+2l=m_0-2t+1\atop
{k+l\leq m_0/2}}kJ_{kl},\
\sum\limits_{k+2l=2t-1}I_{kl}=\sum\limits_{k+2l=m_0-2t+1\atop
{k+l\leq m_0/2}}J_{kl},\\
&\sum\limits_{k+2l=2t}I_{kl}=\sum\limits_{k+2l=m_0-2t\atop {k+l\leq
m_0/2}}J_{kl}\ \text{for} \ 1\leq t\leq  \hat{m}-1,
\end{split}
\end{equation}
and
\begin{equation}\label{-3eee}
\begin{split}
&\sum\limits_{k+2l=2\hat{m}-1}kI_{kl}=\sum\limits_{k+2l=m_0-2\hat{m}+1\atop
{k+l\leq m_0/2}}kJ_{kl}.
\end{split}
\end{equation}

Next we prove by induction that, for $1\leq t\leq \hat{m}-1$, we
have
\begin{equation}\label{-3ind}
I_{2t-1,0}=I_{2t,0}=J_{0,3\hat{m}-1-t}=0
\end{equation}

Setting $t=1$ in (\ref{-3e}), we get
$$
I_{10}=0,\ I_{1,0}=J_{0,3\hat{m}-2},\
I_{20}+I_{01}=J_{1,3\hat{m}-3}.
$$
Together with (\ref{ij}), we obtain
$I_{10}=I_{20}=J_{0,3\hat{m}-2}=0$.

Suppose that  (\ref{-3ind})
 holds
 for  $t\leq t_0\in [1, \hat{m}-2]$. Setting $t=t_0+1$ in
(\ref{-3e}) and making use of the relations above, we get
$$
I_{2t_0+1,0}=0,\ I_{2t_0+1,0}=J_{0,3\hat{m}-t_0-2},\
I_{2t_0+2,0}+I_{2t_0,1}=J_{1,3\hat{m}-t_0-3}.
$$
Combining this with (\ref{ij}), we obtain (\ref{-3ind}) for
$t=t_0+1$. Hence (\ref{-3ind}) holds for $1\leq t\leq \hat{m}-1$.
Substituting the relations in (\ref{-3ind}) to (\ref{-3eee}), we get
$I_{2\hat{m}-1,0}=0$. Thus we get
\begin{equation}\label{-3indc}
I_{t,0}=J_{0,t'}=0\ \text{for}\ 1\leq t\leq 2\hat{m}-1,\
2\hat{m}\leq t'\leq 3\hat{m}-2.
\end{equation}
Namely, we have $I_{t,0}=0$ for $1\leq t\leq 3\hat{m}-2$. By
(\ref{ij0}), we know that
\begin{equation}\label{-3ik0}
I_{k0}=\frac{1}{2i}b_{(0(m_0-2k)k)}+\mathcal{F}\{(b_{(0(m_0-2t)t)})_{t>k}\}.
\end{equation}
In particular, we have $b_{(01(3\hat{m}-2))}=2iI_{(3\hat{m}-2)0}$.
Combining this with (\ref{-3indc}) and (\ref{-3ik0}), we inductively
get $b_{(0(m_0-2k)k)}=0$.

Suppose that $b_{(I'i_nj)}=0$ for $|I'|\leq k_0$. Next we will prove
$b_{(I'i_nj)}=0$ for $|I'|=k_0+1$. Considering all terms of  forms
$z^{I'}z_n^{h}|z_n|^{2j}$ with $|I'|=k_0+1,$ $|I'|+h+2j=m_0$ in
(\ref{-3e00}), we get
\begin{equation}\label{i005}
\begin{split}
{\sum}_{|I'|=k_0+1,h+2j=m_0-|I'|}b_{(I'hj)}z'^{I'}
z_n^{h}\big(|z_n|^{2}+\lambda_nz_n^2+\lambda_n\-z_n^2\big)^j\big|_{\hat{P}_{-3}^{(m_0)}}=0,
\end{split}
\end{equation}
Write
\begin{equation}\label{i002}
\begin{split}
\hat{I}_{kl}=\sum\limits_{h+2j=m_0-|I'|,j\geq
k+l}{b_{(I'hj)}}(_{j-k-l}^j)\lambda_n^{j-k}(_l^{k+l}).
\end{split}
\end{equation}
Then we have
\begin{equation}\label{ij00}
\begin{split}
 \hat{I}_{kl}=(_k^{k+l})\lambda_n^l\hat{I}_{k+l,0}.
\end{split}
\end{equation}
A direct computation shows that
\begin{equation*}
\begin{split}
&{\sum}_{|I'|=k_0+1,h+2j=m_0-|I'|}b_{(I'hj)}z'^{I'}
z_n^{h}\big(|z_n|^{2}+\lambda_nz_n^2+\lambda_n\-z_n^2\big)^j\big|_{\hat{P}_{-3}^{(m_0)}}\\
=&{\sum}_{|I'|=k_0+1,h+2j=m_0-|I'|}b_{(I'hj)}z'^{I'}
z_n^h\sum\limits_{0\leq k+l\leq j}(_{j-k-l}^j)(\lambda_nz_n^2)^{j-k-l}
              (_l^{k+l})(\lambda_n \-z_n^2)^l|z_n|^{2k}\big|_{\hat{P}_{-3}^{(m_0)}}\\
     =&\sum\limits_{|I'|=k_0+1\atop{0\leq k+2l\leq \frac{m_0-|I'|}{2}}}\hat{I}_{kl}z'^{I'}
     z_n^{m_0-|I'|-k-2l}\-z_n^{k+2l}.
\end{split}
\end{equation*}
Thus (\ref{i005}) is equivalent to
\begin{equation}\label{i001}
\begin{split}
\sum\limits_{k+2l=\check{m}}\hat{I}_{kl}=0\ \text{for}\ 1\leq
\check{m}\leq \big[\frac{m_0-|I'|}{2}\big].
\end{split}
\end{equation}
Setting $\check{m}=1$ in (\ref{i001}), we get $\hat{I}_{10}=0$.
Combining this with (\ref{i001}) and (\ref{ij00}), we inductively
obtain $\hat{I}_{t,0}=0$ for $1\leq t\leq [\frac{m_0-|I'|}{2}]$.
From (\ref{i002}), we know that
\begin{equation}\label{-3ik00}
\hat{I}_{k0}=b_{(I'(m_0-|I'|-2k)k)}+\mathcal{F}\{(b_{(I'(m_0-|I'|-2t)t)})_{t>k}\}.
\end{equation}
In particular, we have
$$
b_{\big(I'(m_0-|I'|-2[\frac{m_0-|I'|}{2}])[\frac{m_0-|I'|}{2}]\big)}=\hat{I}_{[\frac{m_0-|I'|}{2}],0}.
$$
Combining this with  (\ref{-3ik00}), we inductively get
$b_{(I'hj)}=0$ for $|I'|= k_0+1,h+2j=m_0-|I'|$. Thus we get
$B^{(m_0)}(z,\-z)=0$.
\bigskip

{\bf (II$_{-2}$)} In this case, we have set $m_0=6\hat{m}-2$. Write
 \begin{equation*}\begin{split}
\hat{P}_{-2}^{(m_0)}=\big\{&\text{polynomials of the form}\
\sum\limits_{t\geq
4\hat{m}-1}a_tz_n^t\-z_n^{m_0-t}+\sum\limits_{t\geq
2\hat{m}-1}b_tz_n^{2t+1}\-z_n^{m_0-2t-3}|z_1|^2\\
&+2{\Re}\big(b_{4\hat{m}-3}z_n^{4\hat{m}-3}\-z_n^{2\hat{m}-1}|z_1|^2\big)+\sum\limits_{I'\neq
0,0\leq 2t\leq m_0-|I'|}C_{I't}(z')^{I'}z_n^{m-|I'|-2t}|z_n|^{2t}
\big\}.
\end{split}\end{equation*}
To get the normalization condition (\ref{ng}) and (\ref{n-2}), we
only need to prove that
\begin{equation}\label{227eq-2}
{\Im}\big(B^{(m_0)}(z,q(z,\-z))\big)\big|_{\hat{P}_{-2}^{(m_0)}}=Q^{(m_0)}(z,\-z)
\end{equation}
is solvable for any real valued formal power series
$Q^{(m_0)}(z,\-z)\in \hat{P}_{-2}^{(m_0)}$.

 The dimension of
 $$\sum\limits_{t\geq 4\hat{m}-1}a_tz_n^t\-z_n^{m_0-t}+\sum\limits_{t\geq
2\hat{m}-1}b_tz_n^{2t+1}\-z_n^{m_0-2t-3}|z_1|^2
+2{\Re}\big(b_{4\hat{m}-3}z_n^{4\hat{m}-3}\-z_n^{2\hat{m}-1}|z_1|^2\big)
$$
is
 \begin{equation}\begin{split}
2\big(6\hat{m}-2-(4\hat{m}-2)\big)+2\big(\frac{6\hat{m}-6}{2}
-(2\hat{m}-2)\big)+1=6\hat{m}-1=2\big[\frac{6\hat{m}-1}{2}\big]+1.
\end{split}\end{equation}
Combining this with (\ref{bdim}) and (\ref{in0dim}), we know that
$B^{(m_0)}(z,\-z)$ and  $\hat{P}_{-2}^{(m_0)}$ have the same
dimension. Hence to prove (\ref{227eq-2}), now we only need to prove
that $B^{(m_0)}\equiv 0$ if
 \begin{equation}\begin{split}\label{-2b0}
  {\Im}\big(B^{(m_0)}(z,{q}(z,\-z))\big)
 \big|_{\hat{P}_{-2}^{(m_0)}}=0,\ {\Re}(b_{0(3\hat{m}-1)})=0.
\end{split}\end{equation}

By (\ref{00coe}), the condition (\ref{n-2}) means that
\begin{equation}\label{-2e}
\begin{split}
&\sum\limits_{k+2l=2t-1}I_{kl}=\sum\limits_{k+2l=m_0-2t+1\atop
{k+l\leq m_0/2}}J_{kl},\
\sum\limits_{k+2l=2t}kI_{kl}=\sum\limits_{k+2l=m_0-2t\atop
{k+l\leq m_0/2}}kJ_{kl},\\
&\sum\limits_{k+2l=2t}I_{kl}=\sum\limits_{k+2l=m_0-2t\atop {k+l\leq
m_0/2}}J_{kl}\ \text{for} \ 1\leq t\leq  \hat{m}-1,
\end{split}
\end{equation}
and
\begin{align}
&\sum\limits_{k+2l=2\hat{m}-1}I_{kl}=\sum\limits_{k+2l=m_0-2\hat{m}+1\atop
{k+l\leq m_0/2}}J_{kl}, \
\sum\limits_{k+2l=2\hat{m}}{\Re}\big(kI_{kl}\big)=\sum\limits_{k+2l=m_0-2\hat{m}\atop
{k+l\leq m_0/2}}{\Re}\big(kJ_{kl}\big).       \label{-2e2}
\end{align}

Next we prove by induction that  the following hods for $1\leq t\leq
\hat{m}-1$:
\begin{equation}\label{-2ind}
\begin{split}
I_{j0}&=(_j^{3\hat{m}-1})\lambda_n^{3\hat{m}-1-j}J_{3\hat{m}-1,0}\
\text{for}\ 1\leq j\leq 2t,\\
J_{j0}&=(_j^{3\hat{m}-1})\lambda_n^{3\hat{m}-1-j}J_{3\hat{m}-1,0}\
\text{for}\ 3\hat{m}-1-t\leq j\leq 3\hat{m}-2.
\end{split}
\end{equation}

Setting $t=1$ in (\ref{-2e}), we get
$$
I_{10}=J_{1,3\hat{m}-2},\ 2I_{2,0}=2J_{2,3\hat{m}-3},\
I_{20}+I_{01}=J_{2,3\hat{m}-3}+J_{0,3\hat{m}-2}.
$$
Hence we obtain
\begin{equation}\label{t1cal}
\begin{split}
I_{10}&=(3\hat{m}-1)\lambda_n^{3\hat{m}-2}J_{3\hat{m}-1,0},\
I_{20}=(_2^{3\hat{m}-1})\lambda_n^{3\hat{m}-3}J_{3\hat{m}-1,0},\\
J_{3\hat{m}-2,0}&=\lambda_n^{-3\hat{m}+2}I_{01}
=\lambda_n^{-3\hat{m}+2}\lambda_n(3\hat{m}-1)\lambda_n^{3\hat{m}-2}J_{3\hat{m}-1,0}
=(_{3\hat{m}-2}^{3\hat{m}-1})\lambda_nJ_{3\hat{m}-1,0}.
\end{split}
\end{equation}
This proves (\ref{-2ind}) for $t=1$.

Suppose that (\ref{-2ind}) holds for some $t\leq t_0\in [1,
\hat{m}-2]$. Next we will prove it also holds for $t=t_0+1$.

By our assumption, we get, for $k+l\leq 2t_0$ and $l\leq t_0$, the
following
\begin{equation*}
\begin{split}
I_{kl}&=(_k^{k+l})\lambda_n^lI_{k+l,0}=(_k^{k+l})\lambda_n^l \cdot
(_{k+l}^{3\hat{m}-1})\lambda_n^{3\hat{m}-1-k-l}
   J_{3\hat{m}-1,0}\\
   &=(_{k}^{3\hat{m}-1})(_{l}^{3\hat{m}-1-k})\lambda_n^{3\hat{m}-1-k}J_{3\hat{m}-1,0},\\
J_{k,3\hat{m}-1-k-l}&=(_k^{3\hat{m}-1-l})\lambda_n^{3\hat{m}-1-k-l}J_{3\hat{m}-1-l,0}=(_k^{3\hat{m}-1-l})
\lambda_n^{3\hat{m}-1-k-l} \cdot
(_{3\hat{m}-1-l}^{3\hat{m}-1})\lambda_n^{l}
   J_{3\hat{m}-1,0}\\
   &=(_{k}^{3\hat{m}-1})(_{l}^{3\hat{m}-1-k})\lambda_n^{3\hat{m}-1-k}J_{3\hat{m}-1,0}.
\end{split}
\end{equation*}
Hence we get, for $k+l\leq 2t_0$ and $l\leq t_0$, the following
\begin{equation}\label{-2med}
\begin{split}
I_{kl}=
   J_{k,3\hat{m}-1-k-l}.
\end{split}
\end{equation}
 Setting $t=t_0+1$ in (\ref{-2e}) and making use of
(\ref{-2med}), we obtain
\begin{equation}
\begin{split}
&I_{2t_0+1,0}=J_{2t_0+1,3\hat{m}-2t_0-2},\\
&(2t_0+2)I_{2t_0+2,0}+2t_0I_{2t_0,1}=(2t_0+2)J_{2t_0+2,3\hat{m}-2t_0-3}+2t_0J_{2t_0,3\hat{m}-2t_0-2},\\
&I_{2t_0+2,0}+I_{2t_0,1}+I_{0,t_0+1}=J_{2t_0+2,3\hat{m}-2t_0-3}+J_{2t_0,3\hat{m}-2t_0-2}+J_{0,3\hat{m}-t_0-2}.
\end{split}
\end{equation}
From the first equation, we get
$$I_{2t_0+1,0}=(_{2t_0+1}^{3\hat{m}-1})\lambda_n^{3\hat{m}-2t_0-2}J_{3\hat{m}-1,0}.$$
Then (\ref{-2med}) holds for $k+l=2t_0+1,l\leq t_0$. Namely, we
obtain $I_{2t_0,1}=J_{2t_0,3\hat{m}-2t_0-2}$. Hence we have
\begin{equation}
\begin{split}
&I_{2t_0+2,0}=J_{2t_0+2,3\hat{m}-2t_0-3}=(_{2t_0+2}^{3\hat{m}-1})\lambda_n^{3\hat{m}-2t_0-3}J_{3\hat{m}-1,0},\\
&J_{3\hat{m}-t_0-2,0}=\lambda_n^{-3\hat{m}+t_0+2}J_{0,3\hat{m}-t_0-2}=\lambda_n^{-3\hat{m}+t_0+2}I_{0,t_0+1}
  =(_{3\hat{m}-t_0-2}^{3\hat{m}-1})\lambda_n^{t_0+1}J_{3\hat{m}-1,0}.
\end{split}
\end{equation}
This proves (\ref{-2ind}) for $t=t_0+1$.  Hence we get
\begin{equation}\label{923new}
\begin{split}
I_{j0}&=(_j^{3\hat{m}-1})\lambda_n^{3\hat{m}-1-j}I_{3\hat{m}-1,0}\
\text{for}\ 1\leq j\leq 2\hat{m}-2,\\
J_{j0}&=(_j^{3\hat{m}-1})\lambda_n^{3\hat{m}-1-j}J_{3\hat{m}-1,0}\
\text{for}\ 2\hat{m}\leq j\leq 3\hat{m}-2.
\end{split}
\end{equation}
Notice that now (\ref{-2med}) holds for $k+l\leq 2\hat{m}-2,l\leq
\hat{m}-1$. Substituting these relations to (\ref{-2e2}) and making
use of (\ref{-2med}) for $k+l\leq 2\hat{m}-2,l\leq \hat{m}-1$, we
get
\begin{equation*}
\begin{split}
&I_{2\hat{m}-1,0}=J_{2\hat{m}-1,\hat{m}},\\
&{\Re}\big(2\hat{m}I_{2\hat{m},0}+(2\hat{m}-2)I_{2\hat{m}-2,1}\big)
={\Re}\big(2\hat{m}J_{2\hat{m},\hat{m}-1}+(2\hat{m}-2)J_{2\hat{m}-2,\hat{m}}\big).
\end{split}
\end{equation*}
From (\ref{923new}) and the first equation above, we get
\begin{equation}\label{-2new}
I_{2\hat{m}-1,0}=(_{2\hat{m}-1}^{3\hat{m}-1})\lambda^{\hat{m}}J_{3\hat{m}-1,0}.
\end{equation}
Combining this with (\ref{ij}), we obtain
$I_{2\hat{m}-2,1}=J_{2\hat{m}-2,\hat{m}}$. Thus we obtain
$${\Re}\big(I_{2\hat{m},0}-2\hat{m}J_{2\hat{m},\hat{m}-1}\big)=0.$$

Since $b_{(0,3\hat{m}-1)}$ is purely imaginary, we know
$I_{3\hat{m}-1,0}=\frac{1}{2i}b_{(0,3\hat{m}-1)}$ is real. Hence
\begin{equation*}
\begin{split}
&I_{2\hat{m},0}=-\-{J_{2\hat{m},0}}=-(_{2\hat{m}}^{3\hat{m}-1})\lambda^{\hat{m}-1}J_{3\hat{m}-1,0},\
J_{2\hat{m},\hat{m}-1}=(_{\hat{m}-1}^{3\hat{m}-1})\lambda^{\hat{m}-1}J_{3\hat{m}-1,0}.
\end{split}
\end{equation*}
Thus we obtain $J_{3\hat{m}-1,0}=0$. Combining this with
(\ref{923new}) and (\ref{-2new}), we get $I_{k,0}=0$ for $1\leq
k\leq 3\hat{m}-1$. By (\ref{ij0}), we know that
\begin{equation}\label{-2ik0}
I_{k0}=\frac{1}{2i}b_{(0(m_0-2k)k)}+\mathcal{F}\{(b_{(0(m_0-2t)t)})_{t>k}\}.
\end{equation}
In particular, we have $b_{(00(3\hat{m}-1))}=2iI_{3\hat{m}-1,0}$.
Hence we inductively get $b_{(0(m_0-2k)k)}=0$.

By a similar induction argument as that used in the (II$_{-3}$)
case, we  get $b_{(Ij)}=0$. Hence we obtain $B^{(m_0)}(z,\-z)=0$.
\medskip

The cases (II$_{-1}$) and (II$_{1}$) can be  similarly  done as for
(II$_{-3}$), while the cases (II$_{0}$) and (II$_{2}$) can be
similarly done  as in the case (II$_{-2}$). This completes the proof
of Theorem \ref{norm}.
\end{proof}

\section{Proof of Theorem \ref{thm1}--- Part II}

We continue our proof of Theorem \ref{thm1}. In this part, we assume
that $M$ is non-minimal at its CR points near the origin. We will
prove $H\equiv 0$ when it satisfies the normalization in Theorem
\ref{norm}.

The crucial step is to prove the following proposition, which is
more or less the content of Theorem \ref{thm1} when $n=3$:

\begin{prop}\label{lem3}
Suppose that $\lambda_n\neq 1/2$. Then for $t,r,s\geq 0$ with
$t+r+s\leq m$, we have
\begin{equation}\label{324eq1}
H_{(te_n+re_1,(m-t-r-s)e_n+se_1)}=0.
\end{equation}
\end{prop}

{\it Proof.} The proof of Proposition \ref{lem3} is carried out in
three steps, according to $\lambda_1=\lambda_n=0$ or $\lambda_n=0,
\lambda_1\not =0$ or $\lambda_n\neq 0, \lambda_1\neq 0$. We notice
that when $\lambda_n\not = 0$, it must hold that $\lambda_1\not =0$
by our choice of $\ld_n$.

\medskip
{\bf Step I:}  In this case, we assume that $\lambda_n=\lambda_1=0$.
Then (\ref{5}) has the following form:
\begin{equation}\label{500}
\begin{split}
\-z_n\Psi_1=\-{z_1}\Psi_n.
\end{split}
\end{equation}
By considering the coefficients of
$z_n^tz^{s-1}\-{z_n}^{r+1}\-{z_1}^{h}$ for $t\geq 0,$ $s\geq 1$,
$r\geq 0$ and $h=m+1-t-s-r\geq 0$ in (\ref{500}), we get
\begin{equation}\label{001}
\begin{split}
s\Psi_{[tsrh]}=(t+1)\Psi_{[(t+1)(s-1)(r+1)(h-1)]}.
\end{split}
\end{equation}
Setting $h=0$ in (\ref{001}), we get $\Psi_{[tsr0]}=0$ for $s\geq
1$. Combining this with (\ref{001}), we  inductively get
$\Psi_{[tsrh]}=0$ for $s\geq h+1$. Together with  (\ref{btsrh}), we
obtain:
\begin{equation}\label{002}
\begin{split}
(s+1)\Phi_{[(t-1)(s+1)(r-1)h]}=(t-1)\Phi_{[tsr(h-1)]}\ \text{for}\
s\geq h+1.
\end{split}
\end{equation}
Setting $h=0$ in (\ref{002}), we get $ \Phi_{[tsr0]}=0\ \text{for}\
s\geq 2.$ Combining this with (\ref{002}), we  inductively get $
\Phi_{[tsrh]}=0$ for $s\geq h+2$. Together with (\ref{atsrh}), we
get
\begin{equation}\label{003}
\begin{split}
(h+1)H_{[(t-1)sr(h+1)]}=(r+1)H_{[t(s-1)(r+1)h]}\ \text{for}\ s\geq
h+2.
\end{split}
\end{equation}
Setting $t=0$, we get $H_{[0srh]}=0$ for $s\geq h+1, \ r\geq 1$.
Then we  inductively get $H_{[tsrh]}=0$ for $s\geq h+1, \ r\geq
t+1$. When $s\geq h+1, r\leq t$,   from (\ref{003}), we inductively
get $H_{[tsrh]}=\mathcal{F}\{(H_{[t's'r'0]})_{t'\geq r'}\}$, which
is $0$ by our normalization in (\ref{ng}). Thus we have proved
\begin{equation}\label{004}
\begin{split}
H_{[tsrh]}=0\ \text{for}\ s\geq h+1.
\end{split}
\end{equation}
Next we  will prove that $H_{[tsrs]}=0$. Setting $s=h\geq 1$, $t\geq
0$ and $r=-1$ in (\ref{001}), we get $ \Psi_{[ts0s]}=0\ \text{for}\
t\geq 1. $ Substituting it back to (\ref{001}), we inductively get
$$
\Psi_{[tsrs]}=0\ \text{for}\ t\geq r+1.
$$
Substituting  (\ref{btsrh}) into this equation, we get
\begin{equation}\label{005}
\begin{split}
(s+1)\Phi_{[(t-1)(s+1)(r-1)s]}=(t-1)\Phi_{[tsr(s-1)]}\ \text{for}\
t\geq r+1.
\end{split}
\end{equation}
Setting $s=0$, we get $\Phi_{[t1r0]}=0$ for $t\geq r+1$.
Substituting this back to (\ref{005}), we get $\Phi_{[t(s+1)rs]}=0$
for $t\geq r+1$. Together with (\ref{atsrh}), we get
\begin{equation*}
\begin{split}
(s+1)H_{[(t-1)(s+1)r(s+1)]}=(r+1)H_{[ts(r+1)s]}\ \text{for}\ t\geq
r+1.
\end{split}
\end{equation*}
Notice that $H_{[t0r0]}=0$ by our normalization. Hence we
inductively get
\begin{equation}\label{007}
\begin{split}
H_{[tsrs]}=0\ \text{for}\ t\geq r.
\end{split}
\end{equation}
Since $H_{[tsrh]}=\-{H_{[rhts]}}$, (\ref{004}) and (\ref{007}) imply
(\ref{324eq1}) for the case $\lambda_n=\lambda_1=0$.

\bigskip

{\bf Step II:}  In this step, we assume that $\lambda_n=0$ and
$\lambda_1\neq 0$. Proposition \ref{lem3} is an immediate
consequence of the following lemma:

\begin{lem}\label{lem01}
  Suppose that $\lambda_n=0$ and $\lambda_1\neq 0$. Assume that there exists an $h_0\geq
  -1$ such that
\begin{equation}\label{01a}
\Psi_{[tsrh]}=\Phi_{[tsrh]}=0\ \text{for}\ h\leq h_0,\ H_{[tsrh]}=0\
\text{for}\ \max(s,h)\leq h_0+1.
\end{equation}
Then we have
\begin{equation}\label{01c}
\Psi_{[tsrh]}=\Phi_{[tsrh]}=0\ \text{for}\ h\leq h_0+1,\
H_{[tsrh]}=0\ \text{for}\ \max(s,h)\leq h_0+2.
\end{equation}
\end{lem}

Once we have Lemma \ref{lem01} at our disposal, since (\ref{01a})
holds for $h_0=-1$ by our normalization, hence (\ref{01c}) holds for
$h_0=-1$. Then by an induction,  we see that (\ref{01c}) holds for
all $h_0\leq m-2$. This will complete the proof of Proposition
\ref{lem3} in this setting.
\medskip

\begin{proof}[Proof of Lemma \ref{lem01}]
Setting $\xi=0$ in (\ref{bequation}) and making use of the assumptions in Lemma \ref{lem01}, we
get:
\begin{equation*}
\begin{split}
&s\Psi_{[(t-1)s(r+1)(h_0+1)]}+(s-2)\eta\Psi_{[t(s-2)(r+2)(h_0+1)]}-t\eta\Psi_{[t(s-2)(r+2)(h_0+1)]}\\
&-(t+1)\eta^2\Psi_{[(t+1)(s-4)(r+3)(h_0+1)]}-\eta\Psi_{[t(s-2)(r+2)(h_0+1)]}=0.
\end{split}
\end{equation*}
Namely, we have
\begin{equation}\label{012}
\begin{split}
&s\Psi_{[(t-1)s(r+1)(h_0+1)]}=(t+3-s)\eta\Psi_{[t(s-2)(r+2)(h_0+1)]}
+(t+1)\eta^2\Psi_{[(t+1)(s-4)(r+3)(h_0+1)]}.
\end{split}
\end{equation}
By setting $r=-3$ in (\ref{012}), we get $\Psi_{[ts0(h_0+1)]}=0$ for
$t\geq 1$. Substituting this back to (\ref{012}), we  inductively
get that  $\Psi_{[tsr(h_0+1)]}=0$ for $t\geq r+1$. Combining this
with (\ref{btsrh}) and (\ref{01a}), we obtain
\begin{equation}\label{014}
\begin{split}
(s+1)\Phi_{[(t-1)(s+1)(r-1)(h_0+1)]}=(t-1)\eta\Phi_{[t(s-1)r(h_0+1)]}\
\text{for}\ t\geq r+1.
\end{split}
\end{equation}
Setting $r=0$ in (\ref{014}), we get $\Phi_{[ts0(h_0+1)]}=0$ for
$t\geq 2$. Hence we  inductively get  $\Phi_{[tsr(h_0+1)]}=0$ for
$t\geq r+2$. In particular, we have $\Phi_{[t0r(h_0+1)]}=0$ for
$t\geq r+2$. Combining this with (\ref{atsrh}), (\ref{01a}) and
$\lambda_n=0$, we get $ (h_0+2)H_{[(t-1)0r(h_0+2)]}=0$ for $t\geq
r+2$. Namely, we obtain $ H_{[t0r(h_0+2)]}=0$ for $t\geq r+1$.
Together with our normalization (\ref{ng}), we obtain:
\begin{equation}\label{016}
\begin{split}
H_{[t0r(h_0+2)]}=0.
\end{split}
\end{equation}

{\bf Case I of Step II:} When $h_0=-1$, setting $s=0$ in
(\ref{014}), we get $\Phi_{[(t-1)1(r-1)0]}=0$ for $t\geq r+1$.
Together with (\ref{atsrh}) and (\ref{lambda0}), we get $
H_{[(t-1)1r1]}=(r+1)H_{[t0(r+1)0]}=0$ for $t\geq r+1$.  Namely, we
obtain $H_{[t1r1]}=0$ for $t\geq r$. By the reality of $H$, we get
\begin{equation}\label{0016}
\begin{split}
H_{[t1r1]}=0.
\end{split}
\end{equation}
From (\ref{atsrh}), (\ref{016}) and (\ref{0016}), we obtain
$\Phi_{[t0r0]}=\Phi_{[t1r0]}=0$. Together with (\ref{btsrh}), we see
that $\Psi_{[t0r0]}=0$.

Setting $h=0$ in (\ref{beq}) and making use of $\Psi_{[t0r0]}=0$, we first get $\Psi_{[t1r0]}=\Psi_{[t2r0]}=0$, then
 inductively get $\Psi_{[tsr0]}=0$. Combining this with
(\ref{btsrh}), we get
\begin{equation}\label{0017}
(s+1)\Phi_{[(t-1)(s+1)(r-1)0]}=(t-1)\eta \Phi_{[t(s-1)r0]},
\end{equation}
Setting $s=0$ in (\ref{0017}), we obtain $\Phi_{[t1r0]}=0$. By an
induction argument, we  get $\Phi_{[tsr0]}=0$. This proves
(\ref{01c}) for the case $h_0=-1$.  \medskip

{\bf Case II  of Step II:} When $h_0\geq 0$, from
(\ref{111ind}),(\ref{016}) and (\ref{0016}), we inductively get
$H_{[tsr(h_0+2)]}=0$ for $s\leq h_0+2$. Combining this with
(\ref{atsrh}) and (\ref{0016}), we get
$\Phi_{[t0r(h_0+1)]}=\Phi_{[t1r(h_0+1)]}=0$. Substituting this back
to (\ref{btsrh}), we obtain $\Psi_{[t0r(h_0+1)]}=0$. Together with
(\ref{beq}), we inductively get $\Psi_{[tsr(h_0+1)]}=0$. Combining
this with (\ref{btsrh}), we obtain
\begin{equation*}
(s+1)\Phi_{[(t-1)(s+1)(r-1)(h_0+1)]}=(t-1)\eta
\Phi_{[t(s-1)r(h_0+1)]}.
\end{equation*}
As in Case I, we inductively get $\Phi_{[tsr(h_0+1)]}=0$. This
proves (\ref{01c}) for the case $h_0\geq 0$ and  thus completes the
proof of Lemma \ref{lem01}.
\end{proof}


\bigskip
{\bf Step III:} In this step, we assume that $\lambda_n\neq 0$ and
$\lambda_1\neq 0$. Similar to the situation in Step II, Proposition
\ref{lem3} in this setting follows from the following lemma:

\begin{lem}\label{lem11}
  Suppose that $\lambda_n\neq 0$ and $\lambda_1\neq 0$. Then we have
  the following:

(I) \begin{equation}
H_{(te_n+e_1,(m-t-2)e_n+e_1)}=H_{(te_n,(m-t)e_n)}=0.
\end{equation}

(II) Assume that there exists an $h_0\geq
  -1$ such that
\begin{equation}\label{11a}
\Psi_{[tsrh]}=\Phi_{[tsrh]}=0\ \text{for}\ h\leq h_0,\ H_{[tsrh]}=0\
\text{for}\ \max(s,h)\leq h_0+1.
\end{equation}
Then we have
\begin{equation}\label{11c}
\Psi_{[tsrh]}=\Phi_{[tsrh]}=0\ \text{for}\ h\leq h_0+1,\
H_{[tsrh]}=0\ \text{for}\ \max(s,h)\leq h_0+2.
\end{equation}
\end{lem}

Before proving Lemma \ref{lem11}, we first make the needed
preparations. For any fixed $s,h\geq 0$ and nonnegative integer $k$,
set
\begin{equation}\label{psiphik}
\begin{split}
&\Psi^{(k)}_{[sh]}={\sum}_{t=k}^{m+1-s-h}(-\xi)^{m+1-t-s-h}(_k^t)\Psi_{[ts(m+1-t-s-h)h]},\\
&\Phi^{(k)}_{[sh]}={\sum}_{t=k}^{m-s-h}(-\xi)^{m-t-s-h}(_k^t)\Phi_{[ts(m-t-s-h)h]},\\
&H^{(k)}_{[sh]}={\sum}_{t=k}^{m-s-h}(-\xi)^{m-t-s-h}(_k^t)H_{[ts(m-t-s-h)h]}.
\end{split}
\end{equation}
Next we would like to transfer the relations among $\Psi$, $\Phi$
and $H$ into the relations among $\Psi^{(k)}_{[s(h_0+1)]}$,
$\Phi^{(k)}_{[s(h_0+1)]}$ and $H^{(k)}_{[s(h_0+2)]}$.

\begin{lem}\label{lemk}
Assume that there exists an $h_0\geq
  -1$ such that
\begin{equation}\label{11aa}
\Psi_{[tsrh]}=\Phi_{[tsrh]}=0\ \text{for}\ h\leq h_0,\ H_{[tsrh]}=0\
\text{for}\ \max(s,h)\leq h_0.
\end{equation}
Then we have
\begin{align}
\Phi^{(k)}_{[s(h_0+1)]}=&(h_0+2)\theta
H^{(k)}_{[s(h_0+2)]}+(h_0+2)H^{(k-1)}_{[s(h_0+2)]}   \nonumber  \\
&+\frac{1}{\xi}\Big((m-s-h_0-k)H^{(k)}_{[(s-1)(h_0+1)]}
-(k+1)H^{(k+1)}_{[(s-1)(h_0+1)]}\Big), \label{ahexp}\\
 \Psi^{(k)}_{[s(h_0+1)]}=&(s+1)\Big(\xi\theta
\Phi^{(k-1)}_{[(s+1)(h_0+1)]}+\xi \Phi^{(k-2)}_{[(s+1)(h_0+1)]}\Big)                \nonumber   \\
        &-\eta(k+1)\theta\Phi^{(k+1)}_{[(s-1)(h_0+1)]}
        -\eta(k-1)\Phi^{(k)}_{[(s-1)(h_0+1)]}.  \label{baexp}
\end{align}
Moreover, $\Psi^{(k)}_{[s(h_0+1)]}$ satisfies the following
equation:
\begin{equation}\label{bkh0}
\begin{split}
s\xi^2\theta\Psi^{(k-2)}_{[s(h_0+1)]}+s\xi^2\Psi^{(k-3)}_{[s(h_0+1)]}=
&\xi\eta\Big\{(k-1)\theta\Psi^{(k)}_{[(s-2)(h_0+1)]}+(k+1-s)\Psi^{(k-1)}_{[(s-2)(h_0+1)]}\Big\}\\
&+(k+1)\eta^2\Psi^{(k+1)}_{[(s-4)(h_0+1)]}.
\end{split}
\end{equation}
\end{lem}

\begin{proof}[Proof of Lemma \ref{lemk}]

Under the assumption in (\ref{11aa}), we easily conclude that
(\ref{atsrh}) and (\ref{btsrh}) have the following expressions:
\begin{align}
\Phi_{[tsr(h_0+1)]}=&(h_0+2)(\xi
H_{[ts(r-1)(h_0+2)]}+H_{[(t-1)sr(h_0+2)]})-(r+1)H_{[t(s-1)(r+1)(h_0+1)]}.
\label{phitsrh}    \\
\Psi_{[tsr(h_0+1)]}=&(s+1)\big\{\xi\Phi_{[t(s+1)(r-2)(h_0+1)]}
+(1+\xi^2)\Phi_{[(t-1)(s+1)(r-1)(h_0+1)]}\nonumber\\
&+\xi\Phi_{[(t-2)(s+1)r(h_0+1)]}\big\}-\big\{\xi\eta(t+1)\Phi_{[(t+1)(s-1)(r-1)(h_0+1)]}
\label{psitsrh}\\
& +\eta (t-1)\Phi_{[t(s-1)r(h_0+1)]}\big\}.\nonumber
\end{align}

For fixed $s\geq 0$, a direct computation shows that
\begin{equation}\label{phic}
\begin{split}
&{\sum}_{t=k}^{m-s-h_0-1}(-\xi)^{m-t-s-h_0-1}(_k^t)H_{[ts(m-t-s-h_0-2)(h_0+2)]}\\
=&(-\xi){\sum}_{t=k}^{m-s-h_0-1}(-\xi)^{m-t-s-h_0-2}(_k^t)H_{[ts(m-t-s-h_0-2)(h_0+2)]}\\
=&-\xi
H^{(k)}_{[s(h_0+2)]},\\
&{\sum}_{t=k}^{m-s-h_0-1}(-\xi)^{m-t-s-h_0-1}(_k^t)H_{[(t-1)s(m-t-s-h_0-1)(h_0+2)]}\\
=&{\sum}_{t=k}^{m-s-h_0-1}(-\xi)^{m-t-s-h_0-1}
\big((_k^{t-1})+(_{k-1}^{t-1})\big)H_{[(t-1)s(m-t-s-h_0-1)(h_0+2)]}\\
=&H^{(k)}_{[s(h_0+2)]}+H^{(k-1)}_{[s(h_0+2)]}.
\end{split}
\end{equation}
Notice that
$$
(m-t-s-h_0)(_k^t)=\big((m-s-k-h_0)-(t-k)\big)(_k^t)=(m-s-k-h_0)(_k^t)-(k+1)(_{k+1}^t).
$$
Thus we have
\begin{equation}\label{phic1}
\begin{split}
&{\sum}_{t=k}^{m-s-h_0-1}(-\xi)^{m-t-s-h_0-1}(_k^t)(m-t-s-h_0)H_{[t(s-1)(m-t-s-h_0)(h_0+1)]}\\
=&\frac{1}{-\xi}{\sum}_{t=k}^{m-s-h_0}(-\xi)^{m-t-s-h_0}
\big((m-s-k-h_0)(_k^t)-(k+1)\big)(_{k+1}^t)\big)H_{[t(s-1)(m-t-s-h_0)(h_0+1)]}\\
=&\frac{1}{-\xi}\big((m-s-k-h_0)H^{(k)}_{[(s-1)(h_0+1)]}-(k+1)H^{(k+1)}_{[(s-1)(h_0+1)]}\big).
\end{split}
\end{equation}
 Substituting (\ref{phic}) and (\ref{phic1}) into (\ref{phitsrh}), we get
(\ref{ahexp}).

Next we prove (\ref{baexp}). A direct computation shows that
\begin{equation}\label{11k1}
\begin{split}
&{\sum}_{t=k}^{m-s-h_0}(-\xi)^{m-t-s-h_0}(_k^t)\big\{\xi\Phi_{[t(s+1)(m-t-s-h_0-2)(h_0+1)]}
+(1+\xi^2)\Phi_{[(t-1)(s+1)(m-t-s-h_0-1)(h_0+1)]}\\
&  \hskip 2cm   +\xi\Phi_{[(t-2)(s+1)(m-t-s-h_0)(h_0+1)]}\big\}\\
=&{\sum}_{t=k-2}^{m-s-h_0-2}(-\xi)^{m-t-s-h_0-2}\big\{\xi\cdot\xi^2(_k^t)-\xi(1+\xi^2)(_k^{t+1})
+\xi(_k^{t+2})\big\}\Phi_{[t(s+1)(m-t-s-h_0-2)(h_0+1)]}\\
=&{\sum}_{t=k-2}^{m-s-h_0-2}(-\xi)^{m-t-s-h_0-2}\big\{\xi\theta(_{k-1}^{t})
+\xi(_{k-2}^{t})\big\}\Phi_{[t(s+1)(m-t-s-h_0-2)(h_0+1)]}\\
=&\xi\theta\Phi^{(k-1)}_{[(s+1)(h_0+1)]}+\xi
\Phi^{(k-2)}_{[(s+1)(h_0+1)]}.
\end{split}
\end{equation}
We also have
\begin{equation}\label{11k2}
\begin{split}
&{\sum}_{t=k}^{m-s-h_0}(-\xi)^{m-t-s-h_0}(_k^t)\big\{\xi(t+1)\Phi_{[(t+1)(s-1)(m-t-s-h_0-1)(h_0+1)]}\\
&\hskip 1.5cm + (t-1)\Phi_{[t(s-1)(m-t-s-h_0)(h_0+1)]}\big\}\\
=&-\xi^2{\sum}_{t=k}^{m-s-h_0}(-\xi)^{m-t-s-h_0-1}(k+1)(_{k+1}^{t+1})
\Phi_{[(t+1)(s-1)(m-t-s-h_0-1)(h_0+1)]}\\
&+{\sum}_{t=k}^{m-s-h}(-\xi)^{m-t-s-h_0}\big\{(k+1)(_{k+1}^{t})+(k-1)(_k^t)\big\}
\Phi_{[t(s-1)(m-t-s-h_0)(h_0+1)]}\\
=&-\xi^2(k+1)\Phi^{(k+1)}_{[(s-1)(h_0+1)]}+(k+1)\Phi^{(k+1)}_{[(s-1)(h_0+1)]}
+(k-1)\Phi^{(k)}_{[(s-1)(h_0+1)]}\\
=&\theta(k+1)\Phi^{(k+1)}_{[(s-1)(h_0+1)]}+(k-1)\Phi^{(k)}_{[(s-1)(h_0+1)]}.
\end{split}
\end{equation}
Substituting (\ref{11k1}) and (\ref{11k2}) into (\ref{psitsrh}), we
get (\ref{baexp}).

Now we turn to the proof of (\ref{bkh0}). Under the assumption
(\ref{11aa}), (\ref{bequation}) has the following form:
\begin{equation}\label{beq*}
\begin{split}
&s\big\{\xi
\Psi_{[tsr(h_0+1)]}+(2\xi^2+1)\Psi_{[(t-1)s(r+1)(h_0+1)]}
+(\xi^3+2\xi)\Psi_{[(t-2)s(r+2)(h_0+1)]}+\xi^2\Psi_{[(t-3)s(r+3)(h_0+1)]}\big\}\\
&+\eta\big\{(s-2)\Psi_{[t(s-2)(r+2)(h_0+1)]}+\xi(s-2)\Psi_{[(t-1)(s-2)(r+3)(h_0+1)]}\big\}\\
&-\eta\big\{(t+1)\xi
\Psi_{[(t+1)(s-2)(r+1)(h_0+1)]}+t(1+\xi^2)\Psi_{[t(s-2)(r+2)(h_0+1)]}
+(t-1)\xi \Psi_{[(t-1)(s-2)(r+3)(h_0+1)]}\big\}\\
&-(t+1)\eta^2\Psi_{[(t+1)(s-4)(r+3)(h_0+1)]}-\eta\theta
\Psi_{[t(s-2)(r+2)(h_0+1)]} =0.
\end{split}
\end{equation}
Notice that
\begin{equation*}
\begin{split}
&(-\xi)^3\xi
(_k^t)+(-\xi)^2(2\xi^2+1)(_k^{t+1})+(-\xi)(\xi^3+2\xi)(_k^{t+2})+\xi^2(_k^{t+3})\\
=&\xi^2\big\{(_k^{t+3})-2(_k^{t+2})+(_k^{t+1})\big\}-\xi^4\big\{(_k^{t+2})-2(_k^{t+1})+(_k^{t})\big\}\\
=&\xi^2\big\{(_{k-2}^{t})+(_{k-3}^{t})\big\}-\xi^4(_{k-2}^{t})
=\xi^2\theta(_{k-2}^{t})+\xi^2(_{k-3}^{t}).
\end{split}
\end{equation*}
Hence we have
\begin{equation}\label{psik1}
\begin{split}
&{\sum}_{t=k}^{m-s-h_0+3}(-\xi)^{m-t-s-h_0+3}(_k^t)\big\{\xi
\Psi_{[ts(m-t-s-h_0)(h_0+1)]}+(2\xi^2+1)\Psi_{[(t-1)s(m-t-s-h_0+1)(h_0+1)]}\\
&+(\xi^3+2\xi)\Psi_{[(t-2)s(m-t-s-h_0+2)(h_0+1)]}+\xi^2\Psi_{[(t-3)s(m-t-s-h_0+3)(h_0+1)]}\big\}\\
=&{\sum}_{t=k-3}^{m-s-h_0}(-\xi)^{m-t-s-h_0}\big\{(-\xi)^3\xi
(_k^t)+(-\xi)^2(2\xi^2+1)(_k^{t+1})\\
&+(-\xi)(\xi^3+2\xi)(_k^{t+2})+\xi^2(_k^{t+3})\big\}\Psi_{[ts(m-t-s-h_0)(h_0+1)]}\\
=&{\sum}_{t=k-3}^{m-s-h_0}(-\xi)^{m-t-s-h_0}\big\{\xi^2\theta(_{k-2}^{t})+\xi^2(_{k-3}^{t})\big\}
\Psi_{[ts(m-t-s-h_0)(h_0+1)]}\\
=&\xi^2\theta\Psi^{(k-2)}_{[s(h_0+1)]}+\xi^2\Psi^{(k-3)}_{[s(h_0+1)]}.
\end{split}
\end{equation}
A direct computation shows that
\begin{equation}\label{psik2}
\begin{split}
&{\sum}_{t=k}^{m-s-h_0+3}(-\xi)^{m-t-s-h_0+3}(_k^t)\big\{\Psi_{[t(s-2)(m-t-s-h_0+2)(h_0+1)]}
      +\xi\Psi_{[(t-1)(s-2)(m-t-s-h_0-3)(h_0+1)]}\big\}\\
=&{\sum}_{t=k-1}^{m-s-h_0+2}(-\xi)^{m-t-s-h_0+2}\big\{-\xi(_k^t)+\xi(_k^{t+1})\big\}
\Psi_{[t(s-2)(m-t-s-h_0+2)(h_0+1)]}\\
=&{\sum}_{t=k-1}^{m-s-h_0+2}(-\xi)^{m-t-s-h_0+2}\xi(_{k-1}^{t})\Psi_{[t(s-2)(m-t-s-h_0+2)(h_0+1)]}\\
=&\xi\Psi^{(k-1)}_{[(s-2)(h_0+1)]}.
\end{split}
\end{equation}
We also obtain the following formulas:
\begin{equation}\label{psik3}
\begin{split}
&{\sum}_{t=k}^{m-s-h_0+3}(-\xi)^{m-t-s-h_0+3}(_k^t)\big\{(t+1)\xi
\Psi_{[(t+1)(s-2)(m-t-s-h_0+1)(h_0+1)]}\\
&+t(1+\xi^2)\Psi_{[t(s-2)(m-t-s-h_0+2)(h_0+1)]}+(t-1)\xi \Psi_{[(t-1)(s-2)(m-t-s-h_0+3)(h_0+1)]}\big\}\\
=&{\sum}_{t=k-1}^{m-s-h_0+2}(-\xi)^{m-t-s-h_0+2}\big\{(-\xi)^2t\xi(_k^{t-1})-\xi
(1+\xi^2)t(_k^t)+t\xi(_k^{t+1})
\big\}\Psi_{[t(s-2)(m-t-s-h_0+2)(h_0+1)]}\\
=&{\sum}_{t=k-1}^{m-s-h_0+2}(-\xi)^{m-t-s-h_0+2}\big\{k\xi\theta(_k^t)+(k-1)\xi(_{k-1}^t)
\big\}\Psi_{[t(s-2)(m-t-s-h_0+2)(h_0+1)]}\\
=&k\xi\theta\Psi^{(k)}_{[(s-2)(h_0+1)]}+(k-1)\xi\Psi^{(k-1)}_{[(s-2)(h_0+1)]}.
\end{split}
\end{equation}
\begin{equation}\label{psik4}
\begin{split}
&{\sum}_{t=k}^{m-s-h_0+3}(-\xi)^{m-t-s-h_0+3}(_k^t)(t+1)\Psi_{[(t+1)(s-4)(m-t-s-h_0+3)(h_0+1)]}
=(k+1)\Psi^{[(k+1)}_{[(s-4)(h_0+1)]},\\
&{\sum}_{t=k}^{m-s-h_0+3}(-\xi)^{m-t-s-h_0+3}(_k^t)\Psi_{[t(s-2)(m-t-s-h_0+2)(h_0+1)]}
=-\xi\Psi^{(k)}_{[(s-2)(h_0+1)]}.
\end{split}
\end{equation}
Combining (\ref{psik1})-(\ref{psik4}) with (\ref{beq*}), we obtain
\begin{equation*}
\begin{split}
&s\xi^2\theta\Psi^{(k-2)}_{[s(h_0+1)]}+s\xi^2\Psi^{(k-3)}_{[s(h_0+1)]}
+(s-2)\xi\eta\Psi^{(k-1)}_{[(s-2)(h_0+1)]}-\eta\big(k\xi\theta\Psi^{(k)}_{[(s-2)(h_0+1)]}
 +(k-1)\xi\Psi^{(k-1)}_{[(s-2)(h_0+1)]}\big)\\
&-(k+1)\eta^2\Psi^{(k+1)}_{[(s-4)(h_0+1)]}
+\xi\eta\theta\Psi^{(k)}_{[(s-2)(h_0+1)]}=0.
\end{split}
\end{equation*}
This finishes the proof of (\ref{bkh0}).
\end{proof}

Now we are in a position to prove Lemma \ref{lem11}.
\begin{proof}[Proof of Lemma \ref{lem11}] (I)
Setting $h_0=-1$, $k=0$ and $h_0=-1$, $k=1$ in (\ref{bkh0}),
respectively, we get
\begin{equation*}
  \begin{split}
&\xi\eta(-1)\theta
\Psi^{(0)}_{[s-2,0]}+\eta^2\Psi^{(1)}_{[s-4,0]}=0,\ \xi\eta(2-s)
\Psi^{(0)}_{[s-2,0]}+2\eta^2\Psi^{(2)}_{[s-4,0]}=0.
  \end{split}
\end{equation*}
Namely,  we have
\begin{equation}\label{odd012}
  \begin{split}
&\xi\theta \Psi^{(0)}_{[s-2,0]}=\eta\Psi^{(1)}_{[s-4,0]},\ (s-2)\xi
\Psi^{(0)}_{[s-2,0]}=2\eta\Psi^{(2)}_{[s-4,0]}.
  \end{split}
\end{equation}

Next we prove  that for $2s\leq m+2,$ $2k\leq m+2-2s$, we have
\begin{equation}\label{odd}
  \begin{split}
  &2s\xi\theta
  \Psi^{(2k)}_{[2s,0]}=(2s+2k)\eta\Psi^{(2k+1)}_{[2s-2,0]},\
  2s\xi \Psi^{(2k)}_{[2s,0]}=(2k+2)\eta\Psi^{(2k+2)}_{[2s-2,0]}.
     \end{split}
\end{equation}
Notice that $\Psi_{[ij]}^{(k)}=0$ for $i+j+k\geq m+2$. Hence  all
the terms in (\ref{odd}) are $0$ when $2k>m+2-2s$. Thus (\ref{odd})
holds for all $k\geq 0$. Also notice that (\ref{odd}) implies
\begin{equation}\label{odd00}
  \begin{split}
  (2s+2k)\Psi^{(2k+1)}_{[2s-2,0]}=(2k+2)\theta\Psi^{(2k+2)}_{[2s-2,0]}.
     \end{split}
\end{equation}

We prove (\ref{odd}) by induction on $s$.\medskip

If  $m=2\hat{m}$  then the largest possible $s$ is $s=\hat{m}+1$.
 In this case,
(\ref{odd}) is has only one nontrivial equation:
$\Psi^{(1)}_{[2\hat{m},0]}=0$.  This can be got by setting
$s=2\hat{m}+4$ and $k=0$ in the first equation of (\ref{odd012}).

If  $m=2\w{m}+1$,  then the largest possible $s$ is $s=\w{m}+1$. In
this case, we must have $k=0$. Hence (\ref{odd}) is the same as
(\ref{odd012}).\medskip

Suppose that (\ref{odd}) holds for $s\geq s_0(\geq 2)$.  Since
(\ref{odd}) holds for $s=s_0-1$ and $k=0$, we can further suppose
that (\ref{odd}) holds for $s=s_0-1$ and $k\leq k_0$. Next we will
prove (\ref{odd}) for $s=s_0-1$ and $k=k_0+1$.

Setting $h_0=-1,$ $s=2s_0$, $k=2k_0+2$ in (\ref{bkh0}), we get
\begin{equation}\label{oddind}
  \begin{split}
  &2s_0\xi^2\theta \Psi^{(2k_0)}_{[2s_0,0]}+2s_0\xi^2\Psi^{(2k_0-1)}_{[2s_0,0]}\\
  =&\xi\eta\big\{(2k_0+1)\theta \Psi^{(2k_0+2)}_{[2s_0-2,0]}+(2k_0+3-2s_0)\Psi^{(2k_0+1)}_{[2s_0-2,0]}\big\}
  +\eta^2(2k_0+3)\Psi^{(2k_0+3)}_{[2s_0-4,0]}.
     \end{split}
\end{equation}
By our assumption, (\ref{odd00}) holds for $s=s_0+1, \ k=k_0-1$ and
$s=s_0, \ k=k_0$, respectively. Hence we get:
\begin{equation}\label{odd1}
  \begin{split}
  &(2s_0+2k_0)\Psi^{(2k_0-1)}_{[2s_0,0]}=2k_0\theta \Psi^{(2k_0)}_{[2s_0,0]}, \
  (2s_0+2k_0)\Psi^{(2k_0+1)}_{[2s_0-2,0]}=(2k_0+2)\theta
\Psi^{(2k_0+2)}_{[2s_0-2,0]}.
     \end{split}
\end{equation}
By  (\ref{odd}) with $s=s_0, k=k_0$ and (\ref{odd1}), we obtain
\begin{equation}\label{odd2}
  \begin{split}
  2s_0\xi^2\theta \Psi^{(2k_0)}_{[2s_0,0]}+2s_0\xi^2\Psi^{(2k_0-1)}_{[2s_0,0]}
  =&\big(2s_0\xi^2\theta+2s_0\xi^2\frac{2k_0}{2s_0+2k_0}\theta\big)\Psi^{(2k_0)}_{[2s_0,0]}\\
  =&\big(1+\frac{2k_0}{2s_0+2k_0}\big)2s_0\xi^2\theta \Psi^{(2k_0)}_{[2s_0,0]}\\
  =&\big(1+\frac{2k_0}{2s_0+2k_0}\big)(2k_0+2)\xi\eta\theta
  \Psi^{(2k_0+2)}_{[2s_0-2,0]}.
     \end{split}
\end{equation}
By (\ref{odd1}), we have
\begin{equation}\label{odd3}
  \begin{split}
  &\xi\eta\big\{(2k_0+1)\theta \Psi^{(2k_0+2)}_{[2s_0-2,0]}+(2k_0+3-2s_0)\Psi^{(2k_0+1)}_{[2s_0-2,0]}\big\}
  +\eta^2(2k_0+3)B^{(2k_0+3)}_{[2s_0-4,0]}\\
=&\xi\eta\big\{2k_0+1+(2k_0+3-2s_0)\frac{2k_0+2}{2s_0+2k_0}\big\}\theta
\Psi^{(2k_0+2)}_{[2s_0-2,0]}+(2k_0+3)\eta^2\Psi^{(2k_0+3)}_{[2s_0-4,0]}.
     \end{split}
\end{equation}
Substituting (\ref{odd2})-(\ref{odd3}) back into (\ref{oddind}), we
get
\begin{equation*}
  \begin{split}
  \eta^2(2k_0+3)\Psi^{(2k_0+3)}_{[2s_0-4,0]}&=\big(1+\frac{2k_0+2}{2s_0+2k_0}(2s_0-3)\big)\xi\eta\theta
  \Psi^{(2k_0+2)}_{[2s_0-2,0]}\\
  &=\frac{(2k_0+3)(2s_0-2)}{2s_0+2k_0}\xi\eta\theta
  \Psi^{(2k_0+2)}_{[2s_0-2,0]}
     \end{split}
\end{equation*}
This is just the first equation of (\ref{odd}).

Setting $h_0=-1$, $s=2s_0$ and  $k=2k_0+3$ in (\ref{bkh0}), we get
\begin{equation}\label{odd5}
  \begin{split}
  &2s_0\xi^2\theta \Psi^{(2k_0+1)}_{[2s_0,0]}+2s_0\xi^2\Psi^{(2k_0)}_{[2s_0,0]}\\
  =&\xi\eta\big\{(2k_0+2)\theta \Psi^{(2k_0+3)}_{[2s_0-2,0]}+(2k_0+4-2s_0)\Psi^{(2k_0+2)}_{[2s_0-2,0]}\big\}
  +\eta^2(2k_0+4)\Psi^{(2k_0+4)}_{[2s_0-4,0]}.
     \end{split}
\end{equation}
By our assumption, (the second equation in) (\ref{odd}) holds for
$s=s_0$, $k=k_0+1$ and for $s=s_0$, $k=k_0$. Namely, we have
\begin{equation}\label{odd500}
2s_0\xi\Psi^{(2k_0+2)}_{[2s_0,0]}=(2k_0+4)\eta\Psi^{(2k_0+4)}_{[2s_0-2,0]},\
2s_0\xi\Psi^{(2k_0)}_{[2s_0,0]}=(2k_0+2)\eta\Psi^{(2k_0+2)}_{[2s_0-2,0]}.
\end{equation}
By our assumption, (\ref{odd00}) holds for $s=s_0+1, \ k=k_0$ and
$s=s_0, \ k=k_0+1$, respectively. Hence we get:
\begin{equation*}
  \begin{split}
  &(2s_0+2k_0+2)\Psi^{(2k_0+1)}_{[2s_0,0]}=(2k_0+2)\theta \Psi^{(2k_0+2)}_{[2s_0,0]}, \
  (2s_0+2k_0+2)\Psi^{(2k_0+3)}_{[2s_0-2,0]}=(2k_0+4)\theta
\Psi^{(2k_0+4)}_{[2s_0-2,0]}.
     \end{split}
\end{equation*}
Combining this with (\ref{odd500}), we get
\begin{equation}\label{odd6}
\begin{split}
2s_0\xi^2
\Psi^{(2k_0+1)}_{[2s_0,0]}&=\frac{2s_0}{2s_0+2k_0+2}(2k_0+2)\xi^2\theta
\Psi^{(2k_0+2)}_{[2s_0,0]}\\
&=\frac{2k_0+2}{2s_0+2k_0+2}\xi\theta
(2k_0+4)\eta\Psi^{(2k_0+4)}_{[2s_0-2,0]}\\
&=(2k_0+2)\xi\eta\Psi^{(2k_0+3)}_{[2s_0-2,0]}.
\end{split}
\end{equation}
Substituting (\ref{odd500}) and (\ref{odd6}) into (\ref{odd5}), we
get
$$
\xi(2s_0-2)\Psi^{(2k_0+2)}_{[2s_0-2,0]}=\eta(2k_0+4)\Psi^{(2k_0+4)}_{[2s_0-4,0]}.
$$
This completes the proof of (\ref{odd}).
\bigskip

Setting $s=1$ in (\ref{odd00}), we get
\begin{equation}\label{s1}
\begin{split}
\theta \Psi^{(2k+2)}_{[0,0]}=\Psi^{(2k+1)}_{[0,0]}\ \text{for}\
0\leq 2k\leq m.
\end{split}
\end{equation}
By (\ref{baexp}), we have $ \Psi^{(k)}_{[0,0]}=\xi\theta
\Phi^{(k-1)}_{[1,0]}+\xi\Phi^{(k-2)}_{[1,0]}. $ Hence (\ref{s1}) is
equivalent to
\begin{equation}\label{as1}
\begin{split}
\theta \big(\xi\theta
\Phi^{(2k+1)}_{[1,0]}+\xi\Phi^{(2k)}_{[1,0]}\big)=\xi\theta
\Phi^{(2k)}_{[1,0]}+\xi \Phi^{(2k-1)}_{[1,0]}\ \text{for}\ 0\leq
2k\leq m.
\end{split}
\end{equation}
Namely, we have $ \theta^2
\Phi^{(2k+1)}_{[1,0]}=\Phi^{(2k-1)}_{[1,0]}$. Setting $k=0$ in this
equation, we  get $\Phi^{(1)}_{[1,0]}=0$. Thus we
 inductively get $\Phi^{(2k-1)}_{[1,0]}=0$. Together with
(\ref{ahexp}), we obtain
\begin{equation}\label{lasteq}
(-\xi)\big(\theta
H^{(2k-1)}_{[11]}+H^{(2k-2)}_{[11]}\big)-(m+1-2k)H^{(2k-1)}_{[00]}+2kH^{(2k)}_{[00]}=0.
\end{equation}

{\bf (II$_{-3}$)}: In this case, we have $m=6\hat{m}-3$.

First, we prove by induction  that
\begin{equation}\label{-311}
H_{[t1(m-t-2)1]}=0\ \text{for} t\geq 4\hat{m}-3.
\end{equation}
In fact, from (\ref{n-3}), we  get
\begin{equation}\label{00t}
H_{[00]}^{(t)}=0\ \text{for}\ t\geq 4\hat{m}-1\ \text{and}\
H_{[11]}^{(6\hat{m}-5)}=0.
\end{equation}
Setting $2k=6\hat{m}-4$ in (\ref{lasteq}), we  get
$H_{[11]}^{(6\hat{m}-6)}=0$. Together with
$H_{[(6\hat{m}-5)101]}=0$, we obtain $H_{[(6\hat{m}-6)111]}=0$.

Suppose that we have obtained $H^{(2t)}_{[11]}=0$ for $t\geq
t_0(\geq 2\hat{m})$. Then $H_{[11]}^{(2t_0-1)}=0$. Setting $2k=2t_0$
in (\ref{lasteq}), we get $H_{[11]}^{(2t_0-2)}=0$. Since
$H_{[t1(m-t-2)1]}=0$ for $t\geq 2t_0-1$. Hence we obtain
$H_{[(2t_0-2)1(m-2t_0)1]}=0$. This completes the proof of
(\ref{-311}).

Now (\ref{lasteq}) takes the following form
\begin{equation}\label{phieq}
\begin{split}
&(-\xi)\sum\limits_{2\hat{m}-1\leq t\leq
4\hat{m}-4}\Big\{\theta(_{2k-1}^{t})(-\xi)^{m-t-2}H_{[t1(m-t-2)1]}
+(_{2k-2}^{t})(-\xi)^{m-t-2}H_{[t1(m-t-2)1]}\Big\}\\
&-\sum\limits_{2\hat{m}-1\leq t\leq
4\hat{m}-2}\Big\{(m+1-2k)(_{2k-1}^{t})(-\xi)^{m-t}H_{[t0(m-t)0]}
-2k(_{2k}^{t})(-\xi)^{m-t}H_{[t0(m-t)0]}\Big\}=0.
\end{split}
\end{equation}
Notice that
\begin{equation*}
\begin{split}
&\theta(_{2k-1}^{t})+(_{2k-2}^{t})=(_{2k-1}^{t+1})-(_{2k-1}^{t})\xi^2,\\
&(m+1-2k)(_{2k-1}^{t})-2k(_{2k}^{t})=(m+1-2k)(_{2k-1}^{t})-(t-2k+1)(_{2k-1}^{t})=(m-t)(_{2k-1}^{t}).
\end{split}\end{equation*}
Hence (\ref{phieq}) takes the following form
\begin{equation}\label{phieq1}
\begin{split}
&(-\xi)\sum\limits_{2\hat{m}-1\leq t\leq
4\hat{m}-4}\Big\{\big((_{2k-1}^{t+1})-(_{2k-1}^{t})\xi^2\big)(-\xi)^{m-t-2}H_{[t1(m-t-2)1]}\Big\}\\
&-\sum\limits_{2\hat{m}-1\leq t\leq
4\hat{m}-2}\Big\{(m-t)(_{2k-1}^{t})(-\xi)^{m-t}H_{[t0(m-t)0]}\Big\}=0.
\end{split}
\end{equation}
Recall that $H_{IJ}=\-{H_{JI}}$. By considering the real part and
the imaginary part in (\ref{phieq1}), respectively, we obtain:
\begin{equation}\label{-3last}
\begin{split}
&(-\xi)\sum\limits_{3\hat{m}-2\leq t\leq
4\hat{m}-4}\Big((_{2k-1}^{t+1})-(_{2k-1}^{t})\xi^2
\pm(-\xi)^{2t+2-m}\big((_{2k-1}^{m-t-1})-(_{2k-1}^{m-t-2})
\xi^2\big)\Big)(-\xi)^{m-t-2}\hat{H}_{[t1]}^{\pm}\\
&-\sum\limits_{3\hat{m}-1\leq t\leq
4\hat{m}-2}\Big((m-t)(_{2k-1}^{t})\pm(-\xi)^{2t-m}t(_{2k-1}^{m-t})\Big)(-\xi)^{m-t}\hat{H}_{[t0]}^{\pm}=0.
\end{split}
\end{equation}
Here we write
\begin{equation}\label{reim}
\begin{split}
&\hat{H}_{[t1]}^{+}={\Re}(H_{[t1(m-t-2)1]}),\
\hat{H}_{[t1]}^{-}={\Im}(H_{[t1(m-t-2)1]}), \\
&\hat{H}_{[t0]}^{+}={\Re}(H_{[t0(m-t)0]}),\
\hat{H}_{[t0]}^{-}={\Im}(H_{[t0(m-t)0]}).
\end{split}
\end{equation}
Then to prove that $\hat{H}_{[t1]}^{\pm}=\hat{H}_{[t0]}^{\pm}=0$, we
only need to prove that the  matrices $(R^{\pm}_{ij})_{1\le i,j\le
2\hat{m}-1}$ are invertible, where $R^{\pm}_{ij}$ are defined as
follows:
\begin{equation}\label{matrixb}
\begin{split}
R^{\pm}_{ij}(\xi)= \begin{cases}
&(^{4\hat{m}-2-j}_{2i-1})-(^{4\hat{m}-3-j}_{2i-1})\xi^2\pm
(-\xi)^{2\hat{m}-1-2j}\big((^{2\hat{m}-1+j}_{2i-1})-(^{2\hat{m}-2+j}_{2i-1})\xi^2\big)\\
&\hskip 3 cm \text{for}\
1\leq j\leq \hat{m}-1,\\
&(\hat{m}-1+j)(^{5\hat{m}-2-j}_{2i-1})\pm
(-\xi)^{4\hat{m}-1-2j}(5\hat{m}-2-j)(^{\hat{m}-1+j}_{2i-1})\\
&\hskip 3 cm  \text{for}\ \hat{m}\leq j\leq 2\hat{m}-1.
\end{cases}
\end{split}
\end{equation}
This  will be done in Lemma \ref{oddnon} of the next section. The
proof for the  Case (II$_{-3}$) is complete.
\medskip

{\bf (II$_{-2}$)}: In this case, we have $m=6\hat{m}-2$. First, we
prove by induction that
\begin{equation}\label{-211}
H_{t1(m-t-2)1}=0\ \text{for}\ t\geq 4\hat{m}-2.
\end{equation}
In fact, from (\ref{n-2}), we  get
\begin{equation}\label{-200t}
H_{[00]}^{(t)}=0\ \text{for}\ t\geq 4\hat{m}-1.
\end{equation}
Setting $2k=6\hat{m}-2$ in (\ref{lasteq}) and noticing that
$H_{[11]}^{(6\hat{m}-3)}=0$, we get $H_{[11]}^{(6\hat{m}-4)}=0$,
which gives that $H_{[(6\hat{m}-4)101]}=0$.

Suppose that we know that $H^{(2t)}_{[11]}=0$ for $t\geq t_0(\geq
2\hat{m})$. Then $H_{[11]}^{(2t_0-1)}=0$. Setting $2k=2t_0$ in
(\ref{lasteq}), we get $H_{[11]}^{(2t_0-2)}=0$. Since
$H_{[t1(m-t-2)1]}=0$ for $t\geq 2t_0-1$. Hence we obtain
$H_{[(2t_0-2)1(m-2t_0)1]}=0$. This proves (\ref{-211}).

Now (\ref{lasteq}) takes the form:
\begin{equation}\label{-2phieq}
\begin{split}
&(-\xi)\sum\limits_{2\hat{m}-1\leq t\leq
4\hat{m}-3}\Big\{\theta(_{2k-1}^{t})(-\xi)^{m-t-2}H_{[t1(m-t-2)1]}
+(_{2k-2}^{t})(-\xi)^{m-t-2}H_{[t1(m-t-2)1]}\Big\}\\
&-\sum\limits_{2\hat{m}\leq t\leq
4\hat{m}-2}\Big\{(m+1-2k)(_{2k-1}^{t})(-\xi)^{m-t}H_{[t0(m-t)0]}
-2k(_{2k}^{t})(-\xi)^{m-t}H_{[t0(m-t)0]}\Big\}=0.
\end{split}
\end{equation}
As for (\ref{phieq}), it takes the form:
\begin{equation}\label{-2phieq1}
\begin{split}
&(-\xi)\sum\limits_{2\hat{m}-1\leq t\leq
4\hat{m}-3}\Big\{\big((_{2k-1}^{t+1})-(_{2k-1}^{t})\xi^2\big)(-\xi)^{m-t-2}H_{[t1(m-t-2)1]}\Big\}\\
&-\sum\limits_{2\hat{m}\leq t\leq
4\hat{m}-2}\Big\{(m-t)(_{2k-1}^{t})(-\xi)^{m-t}H_{[t0(m-t)0]}\Big\}=0.
\end{split}
\end{equation}
Recall that $H_{IJ}=\-{H_{JI}}$. By considering the real part and
imaginary part in (\ref{-2phieq1}), respectively, we obtain:
\begin{equation}\label{-2last1}
\begin{split}
&(-\xi)\sum\limits_{3\hat{m}-1\leq t\leq
4\hat{m}-4}\Big((_{2k-1}^{t+1})-(_{2k-1}^{t})\xi^2
+(-\xi)^{2t+2-m}\big((_{2k-1}^{m-t-1})-(_{2k-1}^{m-t-2})\xi^2\big)\Big)
(-\xi)^{m-t-2}\hat{H}_{[t1]}^{+}\\
&+(-\xi)(-\xi)^{m-(3\hat{m}-2)-2}\big((_{2k-1}^{3\hat{m}-1})
-(_{2k-1}^{3\hat{m}-2})\xi^2\big)\hat{H}_{[(3\hat{m}-2)1]}^{+}\\
 &-\sum\limits_{3\hat{m}\leq t\leq
4\hat{m}-2}\Big((m-t)(_{2k-1}^{t})+(-\xi)^{2t-m}t(_{2k-1}^{m-t})\Big)
(-\xi)^{m-t}\hat{H}_{[t0]}^{+}\\
&-(3\hat{m}-1)(_{2t-1}^{3\hat{m}-1})(-\xi)^{3\hat{m}-1}\hat{H}_{[(3\hat{m}-1)0]}^{+}=0
\end{split}
\end{equation}
and
\begin{equation}\label{-2last2}
\begin{split}
&(-\xi)\sum\limits_{3\hat{m}-1\leq t\leq
4\hat{m}-3}\Big((_{2k-1}^{t+1})-(_{2k-1}^{t})\xi^2
-(-\xi)^{2t+2-m}\big((_{2k-1}^{m-t-1})-(_{2k-1}^{m-t-2})\xi^2\big)\Big)
(-\xi)^{m-t-2}\hat{H}_{[t1]}^{-}\\
 &-\sum\limits_{3\hat{m}\leq t\leq
4\hat{m}-2}\Big((m-t)(_{2k-1}^{t})-(-\xi)^{2t-m}t(_{2k-1}^{m-t})\Big)
(-\xi)^{m-t}\hat{H}_{[t0]}^{-}=0.
\end{split}
\end{equation}

To prove that $\hat{H}_{[t1]}=\hat{H}_{[t0]}=0$, we only need to
prove that the   matrices $(N_{ij})_{1\le i,j\le 2\hat{m}-1}$
$(T_{ij})_{1\le i,j\le 2\hat{m}-2}$ are nonsingular, where $N_{ij}$
and $T_{ij}$ are defined by:
\begin{equation}\begin{split}\label{matrixnt}
N_{ij}=\begin{cases}
&(^{4\hat{m}-2-j}_{2i-1})-(^{4\hat{m}-3-j}_{2i-1})\xi^2+
\xi^{2\hat{m}-2-2j}\big((^{2\hat{m}+j}_{2i-1})-(^{2\hat{m}-1+j}_{2i-1})\xi^2\big)\
\text{for}\
1\leq j\leq \hat{m}-2,\\
&(^{3\hat{m}-1}_{2i-1})-(^{3\hat{m}-2}_{2i-1})\xi^2\ \text{for}\ j=\hat{m}-1,\\
&(\hat{m}+j)(^{5\hat{m}-2-j}_{2i-1})+
\xi^{4\hat{m}-2-2j}(5\hat{m}-2-j)(^{\hat{m}+j}_{2i-1})\ \text{for}\
\hat{m}\leq
j\leq 2\hat{m}-2,\\
&(3\hat{m}-1)(^{3\hat{m}-1}_{2i-1})\ \text{for}\ j=2\hat{m}-1,
\end{cases}\\
T_{ij}=\begin{cases}
&(^{4\hat{m}-1-j}_{2i+1})-(^{4\hat{m}-2-j}_{2i+1})\xi^2-
\xi^{2\hat{m}-2j}\big((^{2\hat{m}-1+j}_{2i+1})-(^{2\hat{m}-2+j}_{2i+1})\xi^2\big)\
\text{for}\
1\leq j\leq \hat{m}-1,\\
&(\hat{m}+j)(^{5\hat{m}-2-j}_{2i+1})-
\xi^{4\hat{m}-2-2j}(5\hat{m}-2-j)(^{\hat{m}+j}_{2i+1})\ \text{for}\
\hat{m}\leq j\leq 2\hat{m}-2.
\end{cases}\end{split}
\end{equation}
These will be the content of Lemma \ref{evennon} of the next
section. Hence we complete the proof of Lemma \ref{lem11} for the
case (II$_{-2}$).
\bigskip

Case(II$_{-1}$) and Case (II$_{1}$) can be done in a similar way as
that for Case (II$_{-3}$), while  Case (II$_{0}$) and case
(II$_{2}$) can be done in a similar way as that  for Case
(II$_{-2}$). Hence we see the proof of Part (I) of Lemma
\ref{lem11}.
\bigskip

Now we turn to the proof of  Part (II) of Lemma \ref{lem11}. From
(\ref{bequation}), we can conclude that
\begin{equation}\label{bindexp}
\begin{split}
\Psi_{[t(2s+1)r(h_0+1)]}=\mathcal{F}\{(\Psi_{[t'(2s'+1)r'(h_0+1)]})_{s'<s},
(\Psi_{[t''s''r''h'']})_{h''\leq h_0} \}\ \text{for}\ s\geq 0.
\end{split}
\end{equation}
In particular, we obtain
$$
\Psi_{[t1r(h_0+1)]}=\mathcal{F}\{(\Psi_{[t's'r'h']})_{h'\leq h_0}
\}=0.
$$
The last equality follows from the assumptions in (\ref{11a}). By an
induction argument, we obtain
\begin{equation}
  \Psi_{[t(2s+1)r(h_0+1)]}=0 \ \text{for}\ 1\leq 2s+1\leq m-h_0.
\end{equation}
Combining this with (\ref{baexp}), we get
\begin{equation}\label{evenind}
  (2s+2)\xi\big(\theta \Phi^{(k-1)}_{[2s+2,h_0+1]}+\Phi^{(k-2)}_{[2s+2,h_0+1]}\big)=\eta(k+1)\theta
  \Phi^{(k+1)}_{[2s,h_0+1]}+\eta(k-1)\Phi^{(k)}_{[2s,h_0+1]}.
\end{equation}
Setting $k=0$ in (\ref{evenind}), we obtain
\begin{equation}\label{even10}
  \theta \Phi^{(1)}_{[2s,h_0+1]}=\Phi^{(0)}_{[2s,h_0+1]}\ \text{for}\ 0\leq
  2s\leq m-h_0-1.
\end{equation}

Next we prove by induction that
\begin{equation}\label{even}
  \theta \Phi^{(2k+1)}_{[2s,h_0+1]}=\Phi^{(2k)}_{[2s,h_0+1]}\ \text{for}\ 0\leq 2s\leq m-h_0-1,
  \ 0\leq 2k\leq m-2s-h_0-1.
\end{equation}
Notice that (\ref{even}) also holds for $2k>m-2s-h_0-1$, in which
case, all the terms in (\ref{even}) are $0$.

When $m-h_0-1=2\hat{m}$ ($2\w{m}+1$, respectively), then the largest
possible $s$ is $s=\hat{m}$ ($\w{m}$, respectively). In this case,
$k=0$, (\ref{even}) reduces to (\ref{even10}).

Suppose that we already have (\ref{even})  for $s\geq s_0$. By
(\ref{even10}), we see that (\ref{even}) also holds for
$k=0,s=s_0-1$. Hence we can suppose that (\ref{even}) holds for
$k\in [0,k_0]$, $s=s_0-1$, Next we will prove that (\ref{even}) also
holds for $s=s_0-1$, $k=k_0+1$.

Setting $s=s_0$, $k=2k_0,2k_0+1,2k_0+2$ in (\ref{evenind}),
respectively, we get
\begin{align}
& 2s_0\xi\big(\theta
  \Phi^{(2k_0-1)}_{[2s_0,h_0+1]}+\Phi^{(2k_0-2)}_{[2s_0,h_0+1]}\big)=\eta(2k_0+1)\theta
  \Phi^{(2k_0+1)}_{[2s_0-2,h_0+1]}+\eta(2k_0-1)\Phi^{(2k_0)}_{[2s_0-2,h_0+1]},  \label{even1}    \\
& 2s_0\xi\big(\theta
  \Phi^{(2k_0)}_{[2s_0,h_0+1]}+\Phi^{(2k_0-1)}_{2s_0,h_0+1}\big)=\eta(2k_0+2)\theta
  \Phi^{(2k_0+2)}_{[2s_0-2,h_0+1]}+\eta 2k_0\Phi^{(2k_0+1)}_{[2s_0-2,h_0+1]},      \label{even2}     \\
&2s_0\xi\big(\theta
  \Phi^{(2k_0+1)}_{[2s_0,h_0+1]}+\Phi^{(2k_0)}_{[2s_0,h_0+1]}\big)=\eta(2k_0+3)\theta
  \Phi^{(2k_0+3)}_{[2s_0-2,h_0+1]}+\eta(2k_0+1)\Phi^{(2k_0+2)}_{[2s_0-2,h_0+1]}.   \label{even3}
\end{align}
Since (\ref{even}) holds for $s=s_0$, $k=k_0-1,k_0, k_0+1$,
respectively, we conclude  that the left hand side of
$(\ref{even1})-2\theta(\ref{even2})+\theta^2 (\ref{even3})$ can be
written as follows:
\begin{equation}\label{even4}
\begin{split}
  2s_0\xi\big\{2\theta\Phi^{(2k_0-1)}_{[2s_0,h_0+1]}-2\theta\big(\theta
  \Phi^{(2k_0)}_{[2s_0,h_0+1]}+\Phi^{(2k_0-1)}_{[2s_0,h_0+1]}\big)
  +2\theta^2\Phi^{(2k_0)}_{[2s_0,h_0+1]}\big\}=0.
\end{split}
\end{equation}

Setting $s=s_0-1$, $k=k_0$ in (\ref{even}), we get $\theta
\Phi^{(2k_0+1)}_{[2s_0-2,h_0+1]}=\Phi^{(2k_0)}_{[2s_0-2,h_0+1]}$.
Thus we obtain:
\begin{equation}\label{even5}
\begin{split}
(2k_0+1)\theta
  \Phi^{(2k_0+1)}_{[2s_0-2,h_0+1]}+(2k_0-1)\Phi^{(2k_0)}_{[2s_0-2,h_0+1]}-2\theta
  \cdot 2k_0
  \Phi^{(2k_0+1)}_{[2s_0-2,h_0+1]}=0.
\end{split}
\end{equation}
Hence by calculating $(\ref{even1})-2\theta(\ref{even2})+\theta^2
(\ref{even3})$ and making use of (\ref{even4})-(\ref{even5}), we
obtain:
\begin{equation}
\begin{split}
\theta^2\big\{(2k_0+3)\theta
  \Phi^{(2k_0+3)}_{[2s_0-2,h_0+1]}+(2k_0+1)\Phi^{(2k_0+2)}_{[2s_0-2,h_0+1]}\big\}-2\theta
  \cdot (2k_0+2)\theta
  \Phi^{(2k_0+2)}_{[2s_0-2,h_0+1]}=0.
\end{split}
\end{equation}
Thus we get
$$
\theta
\Phi^{(2k_0+3)}_{[2s_0-2,h_0+1]}=\Phi^{(2k_0+2)}_{[2s_0-2,h_0+1]}.
$$
This completes the proof of (\ref{even}). \medskip

Setting $s=0$ in (\ref{even}), we get
\begin{equation}\label{sis0}
  \theta \Phi^{(2k+1)}_{[0,h_0+1]}=\Phi^{(2k)}_{[0,h_0+1]}\ \text{for}\ 0\leq 2k\leq m-h_0-1.
\end{equation}
Substituting (\ref{ahexp}) into (\ref{sis0}) and making use of the
assumptions in (\ref{01a}), we get
$$
\theta(\theta H^{(2k+1)}_{[0,h_0+2]}+H^{(2k)}_{[0,h_0+2]})=\theta
H^{(2k)}_{[0,h_0+2]}+H^{(2k-1)}_{[0,h_0+2]}.
$$
Hence we get
\begin{equation}\label{h0odd}
\theta^2 H^{(2k+1)}_{[0,h_0+2]}=H^{(2k-1)}_{[0,h_0+2]}\ \text{for}\
0\leq 2k\leq m-h_0-1.
\end{equation}

Setting $k=0$ in (\ref{h0odd}), we get $H^{(1)}_{[0,h_0+2]}=0$. By
an induction, we get
\begin{equation}\label{h01}
  H^{(2k+1)}_{[0,h_0+2]}=0\ \text{for}\ 0\leq 2k\leq m-h_0-3.
\end{equation}
Next, we will use the just obtained (\ref{h01}) to show that
$H_{[t0(m-t-h_0-2)(h_0+2)]}=0$. We will proceed in terms of the even
or odd property of $m-h_0-1$.

\bigskip {\bf (1)} In this case, we assume $m-h_0-1=2\hat{m}$. By
$H^{(2\hat{m}-1)}_{[0(h_0+2)]}=0$, we get
$H_{[(2\hat{m}-1)00(h_0+2)]}=0.$ By our normalization (\ref{ng}), we
have $H_{[t0r(h_0+2)]}=0$ for $t\leq \hat{m}-1$. Hence (\ref{h01})
with $0\leq k\leq \hat{m}-2$ takes the following form:
$$
\sum\limits_{j=1}^{\hat{m}-1}
S^{(2\hat{m}-2)}_{ij}(-\xi)^{j-1}H_{[(2\hat{m}-1-j)0j(h_0+2)]}=0.
$$
Here we have set
\begin{equation}\label{ms}
S=(S^{(2\hat{m}-2)}_{ij})=\big((_{2i-1}^{2\hat{m}-1-j})\big)_{1\leq
i,j\leq \hat{m}-1}.
\end{equation}
By Lemma \ref{lem4}, $S=(S^{(2\hat{m}-2)}_{ij})$  is nonsingular.
Hence we have $H_{[(2\hat{m}-1-j)0j(h_0+2)]}=0$.

{\bf (2)} In this case, we assume $m-h_0-1=2\w{m}+1$. By our
normalization (\ref{ng}), we have $H_{[t0r(h_0+2)]}=0$ for $t\leq
\w{m}$. Hence (\ref{h01}) has the following form:
$$
\sum\limits_{j=1}^{\w{m}} S_{ij}
(-\xi)^{j-1}H_{[(2\w{m}+1-j)0(j-1)(h_0+2)]}=0\ \text{for}\ 1\leq
i\leq \w{m}.
$$
Here  $ S=(S_{ij})=\big((_{2i-1}^{2\w{m}+1-j})\big)_{1\leq i,j\leq
\w{m}}. $ Now, by Lemma \ref{lem4}, we conclude that
$H_{[t0r(h_0+2)]}=0$.

Thus we  got $H_{[t0r(h_0+2)]}=0$. By (\ref{111ind}), we  get
$H_{[tsr(h_0+2)]}=0$ for $s\leq h_0+2$. Combining this with
(\ref{phitsrh}), we get
\begin{equation}\label{new0}
\Phi_{[t0r(h_0+2)]}=\Phi_{[t1r(h_0+2)]}=0.
\end{equation}
Substituting this back to (\ref{btsrh}), we obtain
$\Psi_{[t0r(h_0+2)]}=0$. By (\ref{beq}), we inductively get
$\Psi_{[tsr(h_0+2)]}=0$. Combining (\ref{111}) with (\ref{new0}), we
inductively get $\Phi_{[tsr(h_0+2)]}=0$. This proves (\ref{11c}) for
the case $h_0\geq 0$ and completes the proof of Lemma \ref{lem11}.
This also finishes the proof of Proposition \ref{lem3}.
\end{proof}

Finally,  we are in a position to complete the proof of Theorem
\ref{thm1}.

\begin{proof}[Proof of Theorem \ref{thm1}] By an induction argument,
we need only to that $H\equiv 0$.

When $n=3$, Theorem \ref{thm1} is the content of Proposition
\ref{lem3}. Next we suppose that $n>3$. We prove $H=0$ by induction
according to the order of $z_n$ in $H$. By Lemma \ref{lem1} and
Proposition \ref{lem3}, we have $H_{(te_n+se_k,re_n+he_n)}=0$ for
$t+s=m$ or $m-1$. Suppose that
\begin{equation}\begin{split}\label{hyp}
H_{(te_n+I,re_n+J)}=0\ \text{for}\ t+r\geq m_0\ (m_0\leq m-1).
\end{split}\end{equation}
Next we will prove that $H_{(te_n+I,re_n+J)}=0$ for $t+r\geq m_0-1$.
The terms of $H$ can be divided into the following four types:
\begin{equation*}\begin{split}
 H_{(te_n,re_n+I)},\  \ H_{(te_n+se_1,re_n+he_1)},\
\  H_{(te_n+se_k,re_n+he_k)}\ \text{with}\ s,h\geq 1,\  \
H_{(te_n+e_j+I,re_n+e_k+J)}.
\end{split}\end{equation*}
By Lemma \ref{lem1} and Proposition \ref{lem3}, terms of the first
two types are $0$. $H_{(te_n+se_k,re_n+he_k)}=0$ follows from
(\ref{22}) and (\ref{hyp}),  while $H_{(te_n+e_j+I,re_n+e_k+J)}=0$
follows from (\ref{1kind}) and (\ref{hyp}).
 Thus we get $H\equiv 0$. This completes the proof of
Theorem \ref{thm1}.
\end{proof}

\section{ Computation of  determinants}

In this section, we will prove that the matrices $S^{(m)}$,
$R^{\pm(m)}$, $N^{(m)}$ and $T^{(m)}$ in the previous section are
nonsingular when $\lambda_n\neq 0,1/2$.

\begin{lem}\label{lem4} The matrices $D^{(2\hat{m})}=\big((_{2i-2}^{2\hat{m}-j})\big)_{1\leq i,j\leq
 \hat{m}}$ and
${S}^{(2\hat{m})}=\big((_{2i-1}^{2\hat{m}+1-j})\big)_{1\leq i,j\leq
 \hat{m}}$ are nonsingular.
\end{lem}

\begin{proof}[Proof of Lemma \ref{lem4}] Set
\begin{equation}\begin{split}\label{323eq3}
&\w{S}_{ij}^{(2\hat{m})}=\frac{(2i-1)}{2\hat{m}-j+1}{S}_{ij}^{(2\hat{m})}\
\text{for}\ 1\leq i\leq \hat{m}, \ 1\leq j\leq \hat{m},\\
&\check{S}_{ij}^{(2\hat{m})}={\w{S}_{ij}^{(2\hat{m})}-\w{S}_{i,j+1}^{(2\hat{m})}}\
\text{for}\ 1\leq i\leq \hat{m},\ 1\leq j\leq \hat{m}-1,\\
&\check{S}_{i,\hat{m}}^{(2\hat{m})}=\w{S}_{i,\hat{m}}^{(2\hat{m})}
\text{for}\ 1\leq i\leq \check{m}.
\end{split}\end{equation}
Then
 we have
\begin{equation}\begin{split}\label{323eq4}
&\w{S}_{ij}^{(2\hat{m})}=\frac{2i-1}{2\hat{m}-j+1}(_{2i-1}^{2\hat{m}-j+1})=(_{2i-2}^{2\hat{m}-j})
=D_{ij}^{(2\hat{m})}\
\text{for}\ 1\leq i\leq \hat{m}.\\
 &\check{S}^{(2\hat{m})}_{1,j}=0\
\text{for}\ 1\leq j\leq \hat{m}-1,\
\check{S}^{(2\hat{m})}_{1,\check{m}}=1,\\
&\check{S}_{ij}^{(2\hat{m})}=(_{2i-2}^{2\hat{m}-j})
-(_{2i-2}^{2\hat{m}-j-1})=(_{2i-3}^{2\hat{m}-j-1})=
 {S}_{i-1,j}^{(2\hat{m}-2)} \ \text{for}\ 2\leq
i\leq \hat{m}.
\end{split}\end{equation}

For ${1\leq i,j\leq \hat{m}}$, we write $
\w{S}_{ij}^{(2\hat{m})}=(\w{S}_{ij}^{(2\hat{m})}),\
\check{S}_{ij}^{(2\hat{m})}=(\check{S}_{ij}^{(2\hat{m})}). $ By
(\ref{323eq4}), we obtain
\begin{equation*}\begin{split}
\check{S}^{(2\hat{m})}=\left(\begin{array}{cc}
0&1\\
 ({S}_{i-1,j}^{(2\hat{m}-2)})_{2\leq i\leq \hat{m},1\leq j\leq \hat{m}-1}&*
\end{array}\right).
\end{split}\end{equation*}
Hence we have
\begin{equation}\label{323eq2}
\det(\check{S}^{(2\hat{m})})=(-1)^{\hat{m}+1}\det{{S}^{(2\hat{m}-2)}}.
\end{equation}
By (\ref{323eq3})-(\ref{323eq2}), we get
\begin{equation}\begin{split}
\det({S}^{(2\hat{m})})&=\prod\limits_{i,j=1}^{\hat{m}}
(2\hat{m}-j+1)
({2i-1})^{-1}\cdot\det(\w{S}^{(2\hat{m})})\\
&=\prod\limits_{i,j=1}^{\hat{m}} (2\hat{m}-j+1) ({2i-1})^{-1}
(-1)^{\hat{m}+1}\det({S}^{(2\hat{m}-2)}).
\end{split}\end{equation}
Notice that $\det({S}^{(2)})=(_1^2)=2$. Thus ${S}^{(2\hat{m})}$ is
nonsingular. Combining this with the first equation in
(\ref{323eq4}), we also conclude that $D^{(2\hat{m})}$ is
nonsingular.
\end{proof}

\begin{lem}\label{oddnon}
Assume that $\xi\neq 0,\frac{1}{2}$. Then the matrices
$R^{\pm}(\xi)$ defined by (\ref{matrixb}) are nonsingular.
\end{lem}

\begin{proof}
Set
$$
R^{[1]}_{ij}=(2i-1)R^{+}_{ij}\ \text{for}\ j\leq \hat{m}-1;\
R^{[1]}_{ij}=\frac{2i-1}{(5\hat{m}-2-j)(\hat{m}-1+j)}R^{+}_{ij}\
\text{for}\ \hat{m}\leq j\leq 2\hat{m}-1.
$$
Notice that $(2i-1)(_{2i-1}^t)=t(_{2i-2}^{t-1})$. Thus for $1\leq
j\leq \hat{m}-1$, we get
\begin{equation}\label{r1}
\begin{split}
R^{[1]}_{ij}=&(2i-1)(_{2i-1}^{4\hat{m}-2-j})-(2i-1)(_{2i-1}^{4\hat{m}-3-j})\xi^2\\
&+(-\xi)^{2\hat{m}-1-2j}\big\{(2i-1)(_{2i-1}^{2\hat{m}-1+j})
-(2i-1)(_{2i-1}^{2\hat{m}-2+j})\big\}\\
=&(4\hat{m}-2-j)(_{2i-2}^{4\hat{m}-3-j})-(4\hat{m}-3-j)(_{2i-2}^{4\hat{m}-4-j})\xi^2\\
&+(-\xi)^{2\hat{m}-1-2j}\big\{(2\hat{m}-1+j)(_{2i-2}^{2\hat{m}-2+j})
-(2\hat{m}-2+j)(_{2i-2}^{2\hat{m}-3+j})\xi^2\big\}.
\end{split}
\end{equation}
For $\hat{m}\leq j\leq 2\hat{m}-1$, we obtain:
\begin{equation}\label{r2}
\begin{split}
R^{[1]}_{ij}=&\frac{2i-1}{5\hat{m}-2-j}(_{2i-1}^{5\hat{m}-2-j})
  +(-\xi)^{4\hat{m}-1-2j}\frac{2i-1}{\hat{m}-1+j}(_{2i-1}^{\hat{m}-1+j})\\
=&(_{2i-2}^{5\hat{m}-3-j})+(-\xi)^{4\hat{m}-1-2j}(_{2i-2}^{\hat{m}-2+j}).
\end{split}
\end{equation}

For $1\leq i\leq 2\hat{m}-1$, write
\begin{equation}
\begin{split}
R^{[2]}_{ij}=&R^{[1]}_{ij}-(4\hat{m}-2-j)R^{[1]}_{i,j+\hat{m}}
+(4\hat{m}-3-j)R^{[1]}_{i,j+\hat{m}+1}\xi^2\
\text{for}\ 1\leq j\leq \hat{m}-2,\\
R^{[2]}_{i,\hat{m}-1}=&R^{[1]}_{i,\hat{m}-1}-\big(3\hat{m}-1
+(3\hat{m}-2)\xi\big)R^{[1]}_{i,2\hat{m}-1},\\
R^{[2]}_{ij}=&(\hat{m}-1+j)R^{[1]}_{ij}+(4\hat{m}-3-2j)R^{[1]}_{i,j+1}\\
    &-(5\hat{m}-4-j)R^{[1]}_{i,j+2}\xi^2+R^{[2]}_{i,j-\hat{m}+1}\
\text{for}\ \hat{m}\leq j\leq 2\hat{m}-3,\\
R^{[2]}_{i,2\hat{m}-2}=&(3\hat{m}-3)R^{[1]}_{i,2\hat{m}-2}
+\big(1+(3\hat{m}-2)\xi\big)
R^{[1]}_{i,2\hat{m}-1}+R^{[2]}_{i,\hat{m}-1},\
R^{[2]}_{i,2\hat{m}-1}=R^{[1]}_{i,2\hat{m}-1}.
\end{split}
\end{equation}
Then for $1\leq i \leq 2\hat{m}-1$ and  $1\leq j \leq \hat{m}-2$ ,
we have
\begin{equation}\label{rij1}
\begin{split}
R^{[2]}_{ij}=&(4\hat{m}-2-j)(_{2i-2}^{4\hat{m}-3-j})-(4\hat{m}-3-j)(_{2i-2}^{4\hat{m}-4-j})\xi^2\\
&+(-\xi)^{2\hat{m}-1-2j}\big\{(2\hat{m}-1+j)(_{2i-2}^{2\hat{m}-2+j})
-(2\hat{m}-2+j)(_{2i-2}^{2\hat{m}-3+j})\xi^2\big\}\\
&-(4\hat{m}-2-j)(_{2i-2}^{4\hat{m}-3-j})-(4\hat{m}-2-j)(-\xi)^{4\hat{m}-1-2(\hat{m}+j)}
(_{2i-2}^{\hat{m}-2+\hat{m}+j})\\
&+(4\hat{m}-3-j)(_{2i-2}^{4\hat{m}-4-j})\xi^2+(4\hat{m}-3-j)(-\xi)^{4\hat{m}-1-2(\hat{m}+j+1)}
(_{2i-2}^{\hat{m}-2+\hat{m}+j+1})\xi^2\\
=&(-\xi)^{2\hat{m}-1-2j}\Big\{(4\hat{m}-3-j)(_{2i-2}^{2\hat{m}-1+j})
-(2\hat{m}-1-2j)(_{2i-2}^{2\hat{m}-2+j})\\
             &-(2\hat{m}-2+j)(_{2i-2}^{2\hat{m}-3+j})\xi^2\Big\}.
\end{split}
\end{equation}
When $j=\hat{m}-1$, we obtain
\begin{equation}
\begin{split}
R^{[2]}_{i,\hat{m}-1}=&(4\hat{m}-2-\hat{m}+1)(_{2i-2}^{4\hat{m}-3-\hat{m}+1})
-(4\hat{m}-3-\hat{m}+1)(_{2i-2}^{4\hat{m}-4-\hat{m}+1})\xi^2\\
&+(-\xi)^{2\hat{m}-1-2\hat{m}+2}\big\{(2\hat{m}-1+\hat{m}-1)(_{2i-2}^{2\hat{m}-2+\hat{m}-1})
-(2\hat{m}-2+\hat{m}-1)(_{2i-2}^{2\hat{m}-3+\hat{m}-1})\xi^2\big\}\\
&-\big(3\hat{m}-1+(3\hat{m}-2)\xi\big)\big\{(_{2i-2}^{5\hat{m}-3-2\hat{m}+1})
+(-\xi)^{4\hat{m}-1-4\hat{m}+2}(_{2i-2}^{\hat{m}-2+2\hat{m}-1})\big\}\\
=&(-\xi)\big\{(3\hat{m}-2)(_{2i-2}^{3\hat{m}-2})-(_{2i-2}^{3\hat{m}-3})
-(3\hat{m}-3)(_{2i-2}^{3\hat{m}-4})\xi^2\big\}.
\end{split}
\end{equation}
For $1\leq i \leq 2\hat{m}-1$ and  $\hat{m}\leq j \leq 2\hat{m}-3$ ,
we get
\begin{equation}
\begin{split}
R^{[2]}_{ij}=&(\hat{m}-1+j)\big\{(_{2i-2}^{5\hat{m}-3-j})
+(-\xi)^{4\hat{m}-1-2j}(_{2i-2}^{\hat{m}-2+j})\big\}\\
         &+(4\hat{m}-3-2j)\cdot\big\{(_{2i-2}^{5\hat{m}-4-j})
         +(-\xi)^{4\hat{m}-3-2j}(_{2i-2}^{\hat{m}-1+j})\big\}\\
         &-(5\hat{m}-4-j)\xi^2\big\{(_{2i-2}^{5\hat{m}-5-j})
         +(-\xi)^{4\hat{m}-5-2j}(_{2i-2}^{\hat{m}+j})\big\}\\
         &+(-\xi)^{2\hat{m}-1-2j+2\hat{m}-2}\big\{(4\hat{m}-3-j+\hat{m}-1)
         (_{2i-2}^{2\hat{m}-1+j-\hat{m}+1})\\
         &-(2\hat{m}-1-2j+2\hat{m}-2)(_{2i-2}^{2\hat{m}-2+j-\hat{m}+1})
         -(2\hat{m}-2+j-\hat{m}+1)(_{2i-2}^{2\hat{m}-3+j-\hat{m}+1})\big\}\\
=&(\hat{m}-1+j)(_{2i-2}^{5\hat{m}-3-j})+(4\hat{m}-3-2j)(_{2i-2}^{5\hat{m}-4-j})
             -(5\hat{m}-4-j)(_{2i-2}^{5\hat{m}-5-j})\xi^2.
\end{split}
\end{equation}
When $j=2\hat{m}-2$, we obtain
\begin{equation}\label{rij2}
\begin{split}
R^{[2]}_{i,2\hat{m}-2}=&(3\hat{m}-3)\big\{(_{2i-2}^{5\hat{m}-3-2\hat{m}+2})
+(-\xi)^{4\hat{m}-1-4\hat{m}+4}
                   (_{2i-2}^{\hat{m}-2+2\hat{m}-2})\big\}\\
&+\big(1+(3\hat{m}-2)\xi\big)\big\{(_{2i-2}^{5\hat{m}-3-2\hat{m}+1})
+(-\xi)^{4\hat{m}-1-4\hat{m}+2}
                   (_{2i-2}^{\hat{m}-2+2\hat{m}-1})\big\}\\
&+(-\xi)\big\{(3\hat{m}-2)(_{2i-2}^{3\hat{m}-2})-(_{2i-2}^{3\hat{m}-3})
         -(3\hat{m}-3)(_{2i-2}^{3\hat{m}-4})\xi^2\big\}\\
=&(3\hat{m}-3)(_{2i-2}^{3\hat{m}-1})+(_{2i-2}^{3\hat{m}-2})
-(3\hat{m}-2) (_{2i-2}^{3\hat{m}-3})\xi^2.
\end{split}
\end{equation}

Set
\begin{equation}
\begin{split}
R^{[3]}_{ij}&=R^{[2]}_{i,\hat{m}-1+j}\ \text{for}\ 1\leq j\leq
\hat{m}-1,\\
R^{[3]}_{ij}&=\frac{-1}{\xi^{2j+1-2\hat{m}}}R^{[2]}_{i,2\hat{m}-1-j}
\ \text{for}\ \hat{m}\leq j\leq 2\hat{m}-2,\
R^{[3]}_{i,2\hat{m}-1}=R^{[2]}_{i,2\hat{m}-1}.
\end{split}
\end{equation}
Then for $1\leq i\leq 2\hat{m}-1$, $1\leq j\leq \hat{m}-1$, we get
\begin{equation}
\begin{split}
R^{[3]}_{ij}=&(\hat{m}-1+\hat{m}-1+j)(_{2i-2}^{5\hat{m}-3-\hat{m}+1-j})+(4\hat{m}-3-2\hat{m}+2-2j)
(_{2i-2}^{5\hat{m}-4-\hat{m}+1-j})\\
&             -(5\hat{m}-4-\hat{m}+1-j)(_{2i-2}^{5\hat{m}-5-\hat{m}+1-j})\xi^2\\
=&(2\hat{m}-2+j)(_{2i-2}^{4\hat{m}-2-j})+(2\hat{m}-1-2j)(_{2i-2}^{4\hat{m}-3-j})
             -(4\hat{m}-3-j)(_{2i-2}^{4\hat{m}-4-j})\xi^2.
\end{split}
\end{equation}
For $1\leq i\leq 2\hat{m}-1$, $\hat{m}\leq j\leq 2\hat{m}-2$, we get
\begin{equation}
\begin{split}
R^{[3]}_{ij}=&(4\hat{m}-3-2\hat{m}+1+j)(_{2i-2}^{2\hat{m}-1+2\hat{m}-1-j})
-(2\hat{m}-1-4\hat{m}+2+2j)
(_{2i-2}^{2\hat{m}-2+2\hat{m}-1-j})\\
&             -(2\hat{m}-2+2\hat{m}-1-j)(_{2i-2}^{2\hat{m}-3+2\hat{m}-1-j})\xi^2\\
=&(2\hat{m}-2+j)(_{2i-2}^{4\hat{m}-2-j})+(2\hat{m}-1-2j)(_{2i-2}^{4\hat{m}-3-j})
             -(4\hat{m}-3-j)(_{2i-2}^{4\hat{m}-4-j})\xi^2.
\end{split}
\end{equation}
Thus for $1\leq i\leq 2\hat{m}-1$, $1\leq j\leq 2\hat{m}-2$, we get
\begin{equation}\label{rijf}
\begin{split}
&R^{[3]}_{ij}=(2\hat{m}-2+j)(_{2i-2}^{4\hat{m}-2-j})+(2\hat{m}-1-2j)(_{2i-2}^{4\hat{m}-3-j})
             -(4\hat{m}-3-j)(_{2i-2}^{4\hat{m}-4-j})\xi^2,\\
&R^{[3]}_{i,\hat{m}-1}=(_{2i-2}^{3\hat{m}-2})-(_{2i-2}^{3\hat{m}-3})\xi.
\end{split}
\end{equation}
Thus we get
\begin{equation}
\begin{split}
\det(R^+)=C_1(-\xi)^{C_0}\det \left(
\begin{array}{cccc}
R^{[3]}_{11}  &\cdots  & R^{[3]}_{1,2\hat{m}-2}  &
(_{0}^{3\hat{m}-2})-\xi (_{0}^{3\hat{m}-3})\\
\vdots &\ddots  &\vdots &\vdots\\
R^{[3]}_{2\hat{m}-1,1}  &\cdots  & R^{[3]}_{2\hat{m}-1,2\hat{m}-2} &
(_{4\hat{m}-4}^{3\hat{m}-2})-\xi (_{4\hat{m}-4}^{3\hat{m}-3})
\end{array}
\right).
\end{split}
\end{equation}
Here  $C_0=\sum\limits_{i=1}^{\hat{m}-1}(2i-1)$. Set
$$
R^{[4]}_{i,j}=\frac{1}{4\hat{m}-3-j}R^{[3]}_{ij},
R^{[4]}_{i,2\hat{m}-1}=R^{[3]}_{i,2\hat{m}-1}\ \text{for}\ 1\leq
i\leq 2\hat{m}-1,\ 1\leq j\leq 2\hat{m}-2.
$$
Then for $2\leq i\leq 2\hat{m}-1$ and $1\leq j\leq 2\hat{m}-2$, we
have
\begin{equation}\label{comp}
\begin{split}
R^{[4]}_{i,j}=&\frac{2\hat{m}-2+j}{4\hat{m}-3-j}\big((^{4\hat{m}-2-j}_{2i-2})-(^{4\hat{m}-3-j}_{2i-2})\big)
+(^{4\hat{m}-3-j}_{2i-2})-(^{4\hat{m}-4-j}_{2i-2})\xi^2\\
=&\frac{2\hat{m}-2+j}{4\hat{m}-3-j}(^{4\hat{m}-3-j}_{2i-3})+(^{4\hat{m}-4-j}_{2i-3})
+(^{4\hat{m}-4-j}_{2i-2})\theta\\
=&\frac{2\hat{m}-2+j}{2i-3}(^{4\hat{m}-4-j}_{2i-4})
+\frac{4\hat{m}-4-j-2i+4}{2i-3}(^{4\hat{m}-4-j}_{2i-4})+(^{4\hat{m}-4-j}_{2i-2})\theta\\
=&\frac{6\hat{m}-2-2i}{2i-3}(^{4\hat{m}-4-j}_{2i-4})+(^{4\hat{m}-4-j}_{2i-2})\theta.
\end{split}
\end{equation}
Set
\begin{equation}
\begin{split}
R^{[5]}_{2\hat{m}-1,j}=R^{[4]}_{2\hat{m}-1,j},\
R^{[5]}_{i,j}=R^{[4]}_{i,j}-\frac{(2i-1)\theta}{6\hat{m}-2i-4}R^{[5]}_{i+1,j}\
\text{for}\ 1\leq i\leq 2\hat{m}-2.
\end{split}
\end{equation}
Then
\begin{equation}\label{d4}
\begin{split}
R^{[5]}_{1,j}=0,
R^{[5]}_{i,j}=\frac{6\hat{m}-2-2i}{2i-3}(^{4\hat{m}-4-j}_{2i-4})\
\text{for}\ 2\leq i\leq 2\hat{m}-1,\ 1\leq j\leq 2\hat{m}-2.
\end{split}
\end{equation}
Hence we obtain
\begin{equation}
\begin{split}
\det(R^+)&=C_2\xi^{C_0}\det \left(
\begin{array}{cc}
0 &R_{1,2\hat{m}-1}^{[5]}\\
\big(\frac{6\hat{m}-2i-4}{2i-1}(_{2i-2}^{4\hat{m}-4-j})\big)_{1\leq
i,j\leq 2\hat{m}-2}&
*
\end{array}
\right)\\
&=C_3\xi^{C_0}\det\big((_{2i-2}^{4\hat{m}-4-j})_{1\leq i,j\leq
2\hat{m}-2}\big)R_{1,2\hat{m}-1}^{[4]}.
\end{split}
\end{equation}
By Lemma \ref{lem4}, $\big((_{2i-2}^{4\hat{m}-4-j})_{1\leq i,j\leq
2\hat{m}-2}\big)$ is nonsingular. Now we only need to prove that
$R_{1,2\hat{m}-1}^{[5]}\neq 0$ for $\xi\neq 0,1$, which follows from
the following claim:
$$
R_{1,2\hat{m}-1}^{[5]}=\a^{3\hat{m}-2}\ \text{with}\ \a=(1-\xi)/{2}.
$$

Notice that $(_{2k-2}^{3\hat{m}-2})=0$ when $k\geq
[\frac{3\hat{m}}{2}]+1$. By (\ref{d4}), we inductively get
$$
R^{[5]}_{1,2\hat{m}-1}=R^{[4]}_{1,2\hat{m}-1}+\sum\limits_{k=2}^{[\frac{3\hat{m}}{2}]}
  \frac{(-1)^{k-1}\prod_{j=2}^{k}(2j-3)\theta^{k-1}}{\prod_{j=1}^{k-1}(6\hat{m}-4-2j)}R^{[4]}_{k,2\hat{m}-1}.
$$
Recall that $\theta=1-\xi^2=(1-\xi)(1+\xi)=2^2\a(1-\a)$. Hence we
have
\begin{equation}
\begin{split}
&\frac{(-1)^{k-1}\prod_{j=2}^{k}(2j-3)\theta^{k-1}}
{\prod_{j=1}^{k-1}(6\hat{m}-4-2j)}R^{[4]}_{k,2\hat{m}-1}\\
=&\frac{(-1)^{k-1}2^{2k-2}\a^{k-1}(1-\a)^{k-1}}{2^{k-1}(k-1)!
(_{k-1}^{3\hat{m}-3})}\frac{(2k-3)!}{2^{k-2}(k-2)!}
\big((_{2k-3}^{3\hat{m}-3})+2\a(_{2k-2}^{3\hat{m}-3})\big)\\
=&2\a^{k-1}(\a-1)^{k-1}\frac{(_{k-1}^{2k-3})}{(_{k-1}^{3\hat{m}-3})}(_{2k-3}^{3\hat{m}-3})
    \big(1+\frac{3\hat{m}-2k}{k-1}\a\big)\\
=&2(\a-1)^{k-1}\a^{k-1}(_{k-2}^{3\hat{m}-2-k})\big(1+\frac{3\hat{m}-2k}{k-1}\a\big)\\
=&2(\a-1)^{k-1}\a^{k-1}\big((_{k-2}^{3\hat{m}-2-k})+(_{k-1}^{3\hat{m}-2-k})\a\big).
\end{split}
\end{equation}
Hence we get
$$
R^{[5]}_{1,2\hat{m}-1}=2\a+\sum\limits_{k=1}^{[\frac{3\hat{m}}{2}]-1}
  2(\a-1)^{k}\a^{k}\big((_{k-1}^{3\hat{m}-3-k})+(_{k}^{3\hat{m}-3-k})\a\big).
$$
Next we prove by induction the following:

\begin{equation}\label{a-1}
\begin{split}
\a^{3\hat{m}-2}-\a&-\sum\limits_{k=1}^{k_0}
  (\a-1)^{k}\a^{k}\big((_{k-1}^{3\hat{m}-3-k})+(_{k}^{3\hat{m}-3-k})\a\big)\\
  =&(\a-1)^{k_0+1}\a^{k_0+1}\sum\limits_{t=k_0}^{3\hat{m}-4-k_0}(_{k_0}^{t})\a^{3\hat{m}-4-k_0-t}.
\end{split}
\end{equation}
Notice that $
\a^{3\hat{m}-2}-\a=\a(\a-1)\cdot\sum_{i=0}^{3\hat{m}-4}\a^i. $ This
proves (\ref{a-1}) for $k_0=0$. Suppose that (\ref{a-1}) holds for
$k_0$, then
\begin{equation*}
\begin{split}
&\a^{3\hat{m}-2}-\a-\sum\limits_{k=1}^{k_0+1}
  (\a-1)^{k}\a^{k}\big((_{k-1}^{3\hat{m}-3-k})+(_{k}^{3\hat{m}-3-k})\a\big)\\
=&(\a-1)^{k_0+1}\a^{k_0+1}\Big\{\sum\limits_{t=k_0}^{3\hat{m}-4-k_0}(_{k_0}^{t})\a^{3\hat{m}-4-k_0-t}
  -\big((_{k_0}^{3\hat{m}-4-k_0})+(_{k_0+1}^{3\hat{m}-4-k_0})\a\big)\Big\}\\
=&(\a-1)^{k_0+1}\a^{k_0+1}\Big\{\sum\limits_{t=k_0}^{3\hat{m}-5-k_0}(_{k_0}^{t})\a^{3\hat{m}-4-k_0-t}
  -(_{k_0+1}^{3\hat{m}-4-k_0})\a\Big\}.
  \end{split}
\end{equation*}
Notice that
$(_{k_0+1}^{3\hat{m}-4-k_0})=\sum\limits_{t=k_0}^{3\hat{m}-5-k_0}(_{k_0}^{t})$.
Hence we have
\begin{equation}\label{a-2}
\begin{split}
&\a^{3\hat{m}-2}-\a-\sum\limits_{k=1}^{k_0+1}
  (\a-1)^{k}\a^{k}\big((_{k-1}^{3\hat{m}-3-k})+(_{k}^{3\hat{m}-3-k})\a\big)\\
  =&(\a-1)^{k_0+1}\a^{k_0+1}\sum\limits_{t=k_0}^{3\hat{m}-5-k_0}(_{k_0}^{t})\big(\a^{3\hat{m}-4-k_0-t}
  -\a\big)\\
  =&(\a-1)^{k_0+1}\a^{k_0+1}\sum\limits_{t=k_0}^{3\hat{m}-5-k_0}(_{k_0}^{t})(\a-1)\a
  \sum\limits_{i=0}^{3\hat{m}-5-k_0-t}\a^i\\
  =&(\a-1)^{k_0+2}\a^{k_0+2}\sum\limits_{t=k_0+1}^{3\hat{m}-5-k_0}(_{k_0+1}^{t})\a^{3\hat{m}-5-k_0-t}.
  \end{split}
\end{equation}
This proves (\ref{a-1}) for $k=k_0+1$ and  completes the proof of
(\ref{a-1}). Setting $k_0=[\frac{3\hat{m}}{2}]$, we obtain
$R^{[5]}_{1,2\hat{m}-1}=2\a^{3\hat{m}-2}$. Hence
$R^{\pm}=C_1\xi^{C_0}(1\mp\xi)^{3\hat{m}-2}$ for some $C_1\neq 0$.
This finishes the proof of Lemma \ref{oddnon}.
\end{proof}

\begin{lem}\label{evennon}
Assume that $\xi\neq 0$. Then  the matrices $N$ and $T$ defined by
(\ref{matrixnt}) are nonsingular.
\end{lem}

\begin{proof} For $1\leq t\leq \hat{m}-2$ and $\hat{m}\leq t'\leq
2\hat{m}-2$, we set
\begin{equation}
\begin{split}
&N^{[1]}_{i,2\hat{m}-1}=\frac{1}{3\hat{m}-1}N_{i,2\hat{m}-1}=(_{2i-1}^{3\hat{m}-1}),\\
&N^{[1]}_{i,\hat{m}-1}=-\frac{1}{\xi^2}\big(N_{i,\hat{m}-1}-N^{[1]}_{i,2\hat{m}-1}\big)=(_{2i-1}^{3\hat{m}-2}),\\
&N^{[1]}_{it}=-\frac{1}{\xi^{2(\hat{m}-t)}}\big(N_{it}-N^{[1]}_{i,t+\hat{m}}
    +N^{[1]}_{i,t+\hat{m}+1}\xi^2-\xi^{2\hat{m}-2t-2}N^{[1]}_{i,t+1}\big)=(_{2i-1}^{2\hat{m}-1+t}),\\
&N^{[1]}_{it'}=-\frac{1}{t'+\hat{m}}\big(N_{it'}-(5\hat{m}-2-t')N^{[1]}_{i,t'-m+1}\xi^{4\hat{m}-2-2t'}\big)
=(_{2i-1}^{5\hat{m}-2-t'}).
\end{split}
\end{equation}
Set
\begin{equation}
\begin{split}
&N^{[2]}_{t}=N^{[1]}_{t+\hat{m}-1}\ \text{for}\ 1\leq t\leq
\hat{m},\ N^{[2]}_{t}=N^{[1]}_{2\hat{m}-t} \ \text{for}\
\hat{m}+1\leq t\leq 2\hat{m}-1.
\end{split}
\end{equation}
Then $N^{[2]}_{ij}=(_{2i-1}^{4\hat{m}-1-j})$. By Lemma \ref{lem4},
the matrix $\big((_{2i-1}^{2(2\hat{m}-1)+1-t})\big)_{1\leq t\leq
2\hat{m}-1}$ is nonsingular.

Next we  calculate the determination of the matrix $T$, which is
done by a similar argument as that for $R^{\pm}$ (And  the proof
now, in fact, is much simpler). For the convenience of the reader,
we include the following details.

 Set
$$
T^{[1]}_{ij}=(2i+1)T_{ij}\ \text{for}\ 1\leq j\leq \hat{m}-1,\
T^{[1]}_{ij}=\frac{2i+1}{(5\hat{m}-2-j)(\hat{m}+j)}T_{ij}\
\text{for}\ \hat{m}\leq j\leq 2\hat{m}-2.
$$
 Corresponding to
(\ref{r1}) and (\ref{r2}), we  get
\begin{equation}
\begin{split}
T^{[1]}_{ij}=&(4\hat{m}-1-j)(_{2i}^{4\hat{m}-2-j})-(4\hat{m}-2-j)(_{2i}^{4\hat{m}-3-j})\xi^2\\
&-\xi^{2\hat{m}-2j}\big\{(2\hat{m}-1+j)(_{2i}^{2\hat{m}-2+j})-(2\hat{m}-2+j)(_{2i}^{2\hat{m}-3+j})\xi^2\big\}\
\text{for}\ 1\leq j\leq \hat{m}-1,\\
T^{[1]}_{ij}=&(_{2i}^{5\hat{m}-3-j})-(-\xi)^{4\hat{m}-2-2j}(_{2i}^{\hat{m}-1+j})\
\text{for}\ \hat{m}\leq j\leq 2\hat{m}-2.
\end{split}
\end{equation}
Set
\begin{equation}
\begin{split}
T^{[2]}_{ij}=&\frac{1}{-\xi^{2\hat{m}-2j}}\big(T^{[1]}_{ij}-(4\hat{m}-1-j)T^{[1]}_{i,\hat{m}-1+j}
              +(4\hat{m}-j-2)T^{[1]}_{i,\hat{m}+j}\xi^2\big)\ \text{for}\ 1\leq i\leq \hat{m}-2,\\
T^{[2]}_{\hat{m}-1}=&\frac{1}{-\xi^{2}}\big(T^{[1]}_{i,\hat{m}-1}-3\hat{m}T^{[1]}_{i,2\hat{m}-2}\big),\\
T^{[2]}_{ij}=&(j+\hat{m})T^{[1]}_{ij}+(4\hat{m}-2j-4)T^{[1]}_{j+1}-(5\hat{m}-j-4)T^{[1]}_{i,j+2}\xi^2
-\xi^{4\hat{m}-2j-4}T^{[2]}_{i,j-\hat{m}+2}\\
 &\ \text{for}\ \hat{m}\leq j\leq
2\hat{m}-4,\\
T^{[2]}_{i,2\hat{m}-3}=&(3\hat{m}-3)T^{[1]}_{i,2\hat{m}-3}+2T^{[1]}_{i,2\hat{m}-2},\
T^{[2]}_{i,2\hat{m}-2}=(3\hat{m}-2)T^{[1]}_{i,2\hat{m}-2}.
\end{split}
\end{equation}
By exactly the same argument as that in (\ref{rij1})-(\ref{rij2}),
we get
\begin{equation}
\begin{split}
T^{[2]}_{ij}=&(4\hat{m}-j-2)(_{2i}^{2\hat{m}-1+j})-(2\hat{m}-2j)(_{2i}^{2\hat{m}+j-2})
-(2\hat{m}-2+j)(_{2i}^{2\hat{m}+j-3})\xi^2\\
&\hskip 3cm\ \text{for}\ 1\leq i\leq \hat{m}-1,\\
T^{[2]}_{i,\hat{m}-1}=&(3\hat{m}-1)(_{2i}^{3\hat{m}-2})-2(_{2i}^{3\hat{m}-3})
-(3\hat{m}-3)(_{2i}^{3\hat{m}-4})\xi^2,\\
T^{[2]}_{ij}=&(j+\hat{m})(_{2i}^{5\hat{m}-3-j})+(4\hat{m}-4-2j)(_{2i}^{5\hat{m}-4-j})
-(5\hat{m}-4-j)(_{2i}^{5\hat{m}-5-j})\xi^2\\
 &\hskip 3cm\  \text{for}\ \hat{m}\leq j\leq
2\hat{m}-4,\\
T^{[2]}_{i,2\hat{m}-3}=&(3\hat{m}-3)(_{2i}^{3\hat{m}})+2(_{2i}^{3\hat{m}-1})
-(3\hat{m}-1)(_{2i}^{3\hat{m}-2})\xi^2,\\
T^{[2]}_{i,2\hat{m}-2}=&(3\hat{m}-2)(_{2i}^{3\hat{m}-1})-(3\hat{m}-2)(_{2i}^{3\hat{m}-3})\xi^2.
\end{split}
\end{equation}
Set
\begin{equation}
\begin{split}
&T^{[3]}_{ij}=T^{[2]}_{i,\hat{m}-1+j}\ \text{for}\ 1\leq j\leq
\hat{m}-1,\ N^{[3]}_{ij}=N^{[2]}_{i,2\hat{m}-1-j}\ \text{for}\
\hat{m}\leq j\leq 2\hat{m}-2.
\end{split}
\end{equation}
Then  for $1\leq i\leq 2\hat{m}-2$, corresponding to (\ref{rijf}),
we have
\begin{equation}
\begin{split}
T^{[3]}_{ij}=&(2\hat{m}+j-1)(_{2i}^{4\hat{m}-2-j})-(2j+2-2\hat{m})(_{2i}^{4\hat{m}-j-3})
-(4\hat{m}-3-j)(_{2i}^{4\hat{m}-j-4})\xi^2.
\end{split}
\end{equation}
By the same computation as that used in (\ref{comp}), we obtain
\begin{equation}
\begin{split}
T^{[4]}_{ij}:=&\frac{1}{4\hat{m}-3-j}T^{[3]}_{ij}
=\frac{6\hat{m}-3-2i}{2i-1}(_{2i-2}^{4\hat{m}-4-j})+(_{2i}^{4\hat{m}-4-j})\theta.
\end{split}
\end{equation}
Set
$$
T^{[5]}_{(2\hat{m}-2)j}=\frac{4\hat{m}-5}{2\hat{m}+1}T^{[4]}_{(2\hat{m}-2)j},\
T^{[5]}_{ij}=\frac{2i-1}{6\hat{m}-3-2i}\big(T^{[4]}_{ij}-\theta
T^{[5]}_{(i+1)j}\big)\ \text{for}\ 1\leq i\leq 2\hat{m}-3.
$$
Then $T^{[5]}_{ij}=\big((_{2i-2}^{4\hat{m}-4-j})\big)_{1\leq i,j\leq
2\hat{m}-2}$.  By lemma \ref{lem4}, $T^{[5]}$ is non-singular. Thus
$T$ is non-singular. This completes the proof of Lemma
\ref{evennon}. \end{proof}

\section {Holomorphic flattening, proofs of Theorems \ref{thmm2},
\ref{thmm3}}

Our proof of Theorem \ref{thmm2} is fundamentally  based on Theorem
\ref{thm1} and the two dimensional results  in Kenig-Webster [KW1]
and Huang-Krantz [HK]. First, as we observed already in $\S 1$, when
$M$ can be holomorphically flattened near $p=0$, all CR points of
$M$ near $0$ must be non-minimal. Hence, in Theorem \ref{thmm2}, we
need only to prove the converse. The proof of Theorem \ref{thmm2} is
an immediate consequence of Theorem \ref{thm1}  and the following
result:

\begin{thm}\label {new-1}
Let $M$ be a real analytic hypersurface with a  CR singularity at
$0$. Suppose that for any $N\ge 3$, there is a holomorphic change of
coordinates of the special form $(z',w'):=(z,
w+O(|zw|+|w|^2+|z|^3))$ such that  $M$ in the new coordinates (which
for simplicity we still write as $(z,w)$) is defined by an equation
of the form:
\begin{equation} \label{00-00}
w=G(z,\-{z})+iE(z,\-{z})=O(|z|^2),\
G(z_1,0,\-{z_1},0)=|z_1|^2+\ld_1(z_1^2+\-{z_1}^2)+o(|z_1|^2),
E(z,\-{z})=O(|z|^N).
\end{equation}
Here the constant $\ld_1$ is such that $0\le \ld_1 <\frac{1}{2}$ and
the real analytic functions  $G, E$ are real-valued. Then $M$ can be
holomorphically flattened near $0$,
\end{thm}

\begin{proof} We now proceed to the proof of Theorem
\ref{new-1}.
The special form for the change of coordinates in the theorem
suggests us to slice $M$ along the
$t:=(z_2,\cdots,z_n)=const$--direction and apply the two dimensional
result in [HK]. By the stability of the elliptic tangency (see [For]
for instance), we get a family of elliptic Bishop surfaces
parametrized by $t$. By the work in Kenig-Webster [KW1] and
Huang-Krantz [HK], each surface bounds a three dimensional
real-analytic Levi-flat manifold. Putting these manifolds together
and tracing the construction of these manifolds through the Bishop
disks, we will obtain a real-analytic hypersurface $\wh{M_N}$. A
major feature for $\wh{M_N}$ is that it has an order $O(N)$ of
vanishing for its Levi-form at $0$. Now, the crucial point is that
the assumption in the theorem and the uniqueness in Kenig-Webster
[KW1] assures that $\wh{M_N}$ will be biholomorphically transformed
to each other near $0$ when  making $N$ larger and larger.
 Hence, we see that $\wh{M_N}$ is a real-analytic
hypersurface with its Levi-form vanishing to the infinite order at
$0$. Thus  the Levi-form of $\wh{M_N}$  vanishes everywhere. Hence
$\wh{M_N}$ is Levi-flat. This then completes the proof of the
theorem. We next give the details on these.


In the following, we write $t=(z_2,\cdots,z_{n})=$ and  write
$u=\Re{w},\ v=\Im{w}$. For $|t|$ small, define $M_t=\{(z,w)\in M:
(z_2,\cdots,z_{n})=t\}$. Then $M_t$ is a small deformation of the
original $M_0$, which has a unique elliptic complex tangent at
$z_1=0$ for $|z_1|<\e_0<<1$. Since a small deformation of the
surface will only move the complex tangent point to a nearby point
and elliptic complex tangency is stable under small deformation,
intuitively, $M_t$ must have an elliptic complex tangent near
$z_1\approx 0$, which is completely determined by the equation:
$$\frac{\p w}{\p
\-{z_1}}=2\lambda_1\-{z_1}+{z_1}+\frac{\p (p+iE)}{\p \-{z_1}}(z_1,t,
\-{z_1},\-{t})=0.$$ Here, we also write
$p(z,\-{z})=G(z,\-{z})-\left(|z_1|^2+\ld_1(z_1^2+\-{z_1}^2)\right)$.
 By the implicit function theorem, one can solve
uniquely $z_1=a(t,\-{t})=O(|t|)$, which is $C^\o$ in $t$. Then
$$P(t)=\left(a(t,\-{t}),t,(G+\sqrt{-1}E)(a,t,\-{a},\-{t})\right )$$ is
the elliptic complex tangent point over $M_t$ obtained by deforming
the $0$ on $M_0$ to $M_t$. Next, we expand (\ref{00-00}) at
$(a(t,\-{t}),t)$:
\begin{equation}\label{flat-02}
\begin{split}
&w=w_0(t,\-{t})+b(t,\-{t})(z_1-a(t,\-{t}))+2\Re\left (
c(t,\-{t})(z_1-a(t,\-{t}))^2\right ) +d(t,\-{t})|z_1-a(t,\-{t})|^2+
\\ & h^*(z_1-a(t,\-{t}),t,\-{z_1-a(t,\-{t})},\-{t})
+\sqrt{-1}G^*\left( z_1-a(t,\-{t}),t,\-{z_1-a(t,\-{t})},\-{t}\right
)
\end{split}
\end{equation}
Here, all functions appeared above depend $C^\o$-smoothly on their
variables with $w_0(0,0)=0,\ d(0,0)=1, \ b(0,0)=0$,
$c(0,0)=\lambda_1$. Moreover,
$h^*(\eta,t,\-{\eta},\-{t})=O(|\eta|^3)$,
$G^*(\eta,t,\-{\eta},\-{t})=O(|\eta|^2)\cap O(|\eta|^N+|t|^N)$
 and $d(t,\-{t})$
are all real-valued. By  continuity, for $|t|$ small, we have
$A(\eta,\-{\eta},t,\-{t}):=2\Re\left( c(t,\-{t})\eta^2\right
)+d(t,\-{t})|\eta|^2\ge C|\eta|^2$ for a certain positive constant
$C$ independent of  $|t|$.
Hence, for $|t|$ small and for a real number $r$ with $|r|<<1$, the
following defines a simply connected (convex) domain $D_{t}$ in
${\mathbb C}$ with a real analytic boundary:
$$D_t:=\{\eta\in {\mathbb C}:\ \ 2\Re\left( c(t,\-{t})\eta^2\right)
+d(t,\-{t})|\eta|^2+r^{-2}h^*(r\eta,r\-{\eta},t,\-{t})\le 1\}.$$

Let $\sigma(\xi,t,\-{t},r)$ be the Riemann mapping from the unit
disk to $D_t$  preserving the origin. By [Lemma 2.1, Hu1],
$\sigma(\xi,t,\-{t},r)$ depends $C^\o$ on its variables and is
holomorphic in $\xi$ in a fixed neighborhood of $\-{\D}$.
(See also [Lemma 4.1, Hu2] for a detailed proof on this.)

Now, we construct a family of holomorphic disks with parameter
$(t,r)$ for $|t|,|r|<<1$ attached to $M$, which takes the following
form:
\begin{equation}\label{flat-03}
\begin{split}
& z_1(\xi,t,\-{t},r)=a(t,\-{t})+r\sigma(\xi,t,\-{t},r)(1+\psi_1(\xi,t,\-{t},r)),\\
& (z_2,\cdots,z_{n})=t,\\
& w(\xi,t,\-{t},r)=w_0(t,\-{t})+b_1(t,\-{t})\cdot
r\sigma(\xi,t,\-{t},r)(1+\psi_1(\xi,t,\-{t},r))+
r^2(1+\psi_2(\xi,t,\-{t},r)),\\& \Re{\psi_1(0,t,\-{t}, r)}=0,\ \
\Im{\psi_2(0,t,\-{t}, r)}=0,\\
& \psi=(z_1(\xi,t,\-{t},r),t,w(\xi,t,\-{t},r))
\end{split}
\end{equation}
Here $\psi_1,\psi_2$ are holomorphic functions  in $\xi\in \D$,  and
are $C^\o$  on $(\xi,t,r)$ over $\-{\D}\times \{t\in {\mathbb
C}^{n-2}:|t|<\e_0 \}\times \{ r\in {\mathbb R}: |r|<\e_0\}$.
Substituting (\ref{flat-03}) into (\ref{flat-02}) with $|\xi|=1$, we
get the following:
\begin{equation}
\psi_2(\xi,t,\-{t},r)= \O_1+\O_2+\sqrt{-1}\O_3.
\end{equation}
Here $\O_1=2\Re\left(\{\frac{\p A}{\p
\eta}(\sigma,t,\-{\sigma},\-{t})\sigma +\sigma r^{-1}\frac{\p
h^*}{\p \eta}(r\sigma,t,r\-{\sigma},\-{t})\}\psi_1\right)$,
$\Omega_2=O(|\psi_1|^2)$,  and $\Omega_3=O(|t|^{N-2}+|r|^{N-2})$ are
all real-valued. Moreover, $\O_j$ $(j=1,2,3)$ depend $C^\o$ on there
variables $(\psi_1, t,r)$ in a certain  suitable Banach space
 defined in [$\S 5$, Hu1]. Write
$g(\xi,\-{\xi},t,\-{t},r)=2\sigma\{\frac{\p A}{\p
\eta}(\sigma,t,\-{\sigma},\-{t})\sigma +r^{-1}\frac{\p h^*}{\p
\eta}(r\sigma,t,r\-{\sigma},\-{t})\}.$
 Then we similarly have $\Re{g}>0$, which makes results in [Lemma
5.1, Hu1] applicable in our setting. Write $\cal H$ for the standard
Hilbert transform, we obtain the following singular Bishop equation:
\begin{equation}\label{flat-05}
\Re\{g(\xi,\-{\xi},t,\-{t},r)\psi_1\}+\O_2(\psi_1,\-{\psi_1},t,\-{t},r)=-{\cal
H}(\O_3).
\end{equation}
Now, write $\psi_1=U(\xi,\-{\xi},t,\-{t},r)+\sqrt{-1}{\cal
H}(U(\xi,\-{\xi},t,\-{t},r))$ for $|\xi|=1$. By the argument in [$\S
5$, Hu1], from (\ref{flat-05}), one can uniquely solve
$U(\xi,\-{\xi},t,\-{t},r)$ for $|t|,|r|<<1$. Moreover,
$U(\xi,\-{\xi},t,\-{t},r)$ depends $C^\o$ on
$(\xi,\-{\xi},t,\-{t},r)$ and
$U(\xi,\-{\xi},t,\-{t},r)=O(|t|^{N-2}+|r|^{N-2}).$ Hence
$U(\xi,\-{\xi},t,\-{t},r)+\sqrt{-1}{\cal
H}(U(\xi,\-{\xi},t,\-{t},r))$ extends to a holomorphic function in
$\xi$ which also depends $C^\o$ on its variables
$(\xi,\-{\xi},t,\-{t},r)$ with $|\xi|\le 1$. Moreover, we have the
estimates
\begin{equation}\label{fla-01}
\psi_1,\psi_2=O(|t|^{N-2}+|r|^{N-2}).
\end{equation}

Next, we let $\wh{M_N}=\bigcup_{0\le r<<1,\ |t|<<1,\xi\in
\-{\D}}\psi(\xi,t,\-{t},r).$ Let $\wt{M}=\pi(\wh{M_N})$ where $\pi$
is the projection from ${\mathbb C}^{n+1}$ into the $(z,u)$-space.
By the results in Kenig-Webster [KW1] and Huang-Krantz [HK], for
each fixed $t$,  $\wh{M_{N,t}}=\wh{M_N}\cap \{z'=t\}\cap
B_{P(t,\-{t})}(r_0)$ must be  the local hull of holomorphic of
$M_t$, that is a manifold
 $C^\o$-regular up to the boundary $M_t$. Here $B_{a(t,\-{t})}(r_0)$ is
 the ball centered at $P(t,\-{t})$ with
 a certain fixed radius $r_0>0$. Also, since
$v=G(z_1,t,\-{z_1},\-{t})$ defines a strongly pseudoconvex
hypersurface in ${\mathbb C}^2$ for each fixed $t$, we see that
$\pi(\wh{M_{N,t}})\subset \wt{M^*}_{t}$, where $\wt{M^*}:= \{(z,w):
u\ge G(z,\-{z})\}$ and
$\wt{M^*_t}=\wt{M^*}\cap\{(z_2,\cdots,z_{n})=t\}.$ Indeed, $\pi$,
when restricted to $\wh{M_{N,t}}$ is a $C^\o$-diffeomorphism to
$\wt{M^*_t}$ in the intersection of $\wh{M_{N,t}}$ with the ball
centered at $P(t)$ with a certain fixed radius $1>>r_0>0$. To see
this, by our normalization presented in the previous section or  by
the Kenig-Webster [KW], we have a change of variables in $(z_1,w)$:
$$z_1'=z_1-a(t,\-{t}),\ \
w'=w_0(t,\-{t})+\sum_{j=1}^{m}b_j(t,\-{t})(w-w_0(t,\-{t}))^j,$$
where $w_0, \ b_j$ depend smoothly on $t$ and takes values $0$ at
$0$. In this coordinates, $M_t$ is mapped to $M_t'$ that is
flattened to order $m$ at $0$. Hence, for $m>>1$, the holomorphic
hull of  $M'_t$ near $0$ now is tangent to $(z_1',u')$-space (See
[KW] [HK]), in particular, must be transversal to the $v'-$axis.
Since the hull is a biholomorphic invariant, we see that
$\wh{M_{N,t}}$ has to be transversal to the $v$-axis when $|t|$ is
small. Hence,
 $\pi$ is a one to one and onto map from $\wh{M_N}$ to $\wt{M^*}$
near $0$. Write the  inverse map of $\pi$ as
$v(z_1,\-{z_1},t,\-{t})$, which is defined  over $\wt{M^*}$ near
$0$. Notice that it is the graph function of $\wh{M_N}$ near $0$ and
has to be $C^\o$-regular  for each fixed $t$.

 Next, we solve $v(z_1,\-{z_1},t,\-{t})$ from (\ref{flat-03}). For
this, we use  the computation in [HK]. First, we let
\begin{equation}\label{flat-07}
\begin{split}
&z'_1=z_1-a(t,\-{t})\\
& w'=w-\left (w_0(t,\-{t})+b_1(t,\-{t})(z_1-a(t,\-{t}))\right )
\end{split}
\end{equation}
Then (\ref{flat-03}) can be rewritten as
\begin{equation}\label{flat-08}
\begin{split}
& z'_1(\xi,t,\-{t},r)=r\sigma(\xi,t,\-{t},r)(1+\psi_1(\xi,t,\-{t},r)),\\
& (z¡¯_2,\cdots,z¡®_{n})=t,\\
& w'(\xi,t,\-{t},r)=r^2(1+\psi_2(\xi,t,\-{t},r)).
\end{split}
\end{equation}

Write $w'=u'+\sqrt{-1}v'$. Now, by  the proof in [HK, pp 225], we
see that for each $(z_1', u',t)$, there is a unique $v'$ satisfying
(\ref{flat-08}). Moreover $v'$, as a function in $(z_1', u',t)$, has
the following generalized Puiseux expansion:
$$v'=\sum_{i,j,s,\a,\b\ge 0}S_{ijs\a\b}u'^{\frac{i-j-s}{2}}z^j_1\-{z_1}^s
t^\a\-{t^\b},$$ where $|S_{ijk\a\b}|\ale C^{i+j+k}$ for some
positive constant $C$.
 By the regularity of $\wh{M_{N,t}}$ as mentioned above, we know
that $\frac{\p^{j+s} v'}{\p {z'}_1^j\p\-{z'_1}^s}|_{z_1=0}$ must be
smooth for each $|t|$ small and $u'\ge 0$.  This shows that
$S_{ijs\a\b}=0$ when $\frac{i-j-s}{2}$ is  not a positive integer.
As in [HK, pp 227], we see that $v'$ is a real analytic function in
$(z'_1,u', t)$ near $0$. By (\ref{flat-07}), we see that $v$ is
analytic function in $(z_1,u,t)$. Hence, we proved that $\wh{M_{N}}$
is a real analytic manifold, which can be represented as a graph
over $\wt{M^*}$ in $(z, u)$-space. Moreover the analytic graph
function $v=\rho=O(|t|^{N-2}+|z_1|^{N/2}).$ To see this, by
(\ref{flat-03}) (\ref{fla-01}), we need only to explain that
$\Im(w_0)=O(|t|^N)$ and $b_1(t,\-{t})=O(|t|^{N-1})$. Indeed,
$\Im(w_0)=E(a(t,\-{t}),\-{a(t,\-{t})},t,\-{t})=O(|t|^N)$ and
$$b_1=\frac{\p (G+\sqrt{-1}E)}{\p
z_1}(a(t,\-{t}),\-{a(t,\-{t})},t,\-{t}).$$ Since
$$\frac{\p
(G+\sqrt{-1}E)}{\p \-{z_1}}(a(t,\-{t}),\-{a(t,\-{t})},t,\-{t})=0,$$
 we get
$$b_1=2\sqrt{-1}\frac{\p E}{\p
z_1}(a(t,\-{t}),\-{a(t,\-{t})},t,\-{t})=O(|t|^{N-1}).$$
\bigskip

Still let $v=\rho$ be the defining function as mentioned above.
Since $\rho$ is real analytic, we can extend  $\wh{M_{N}}$ to a real
analytic hypersurface $M^{\#}_N$ near the origin by using the graph
of $\rho$. Now, we let
$$\theta=\sqrt{-1}\p (-\frac{w-\-{w}}{2\sqrt{-1}}+\rho),  L_j=\frac{\p }{\p
z_j}+\frac{2\sqrt{-1}\rho_{z_j}}{1-2\sqrt{-1}\rho_w}\frac{\p}{\p
w}.$$ Then $\theta$ is a contact form along $M^{\#}_N$  and
$\{L_j\}_{j=1}^{n}$ forms a basis of real analytic tangent vector
fields of type $(1,0)$ along $M^{\#}_N$ near $0$.
 With respect to such a contact form and a basis of tangent vector
fields of type $(1,0)$,  we obtain the following Levi-matrix, which
is a real analytic $n\times n$-matrix near $0$:
$${\cal L}_N=\left((\sqrt{-1}d\theta, L_j\wedge \-{L_k})\right)_{1\le j,k\le n}.$$
Since $\rho=O(|z|^{N/2})$, we see that ${\cal L}=O( |z|^{N/2-3})$ as
$z\ra 0$. Now, for $N'>N$, by the existence of the special change of
coordinates as in the hypothesis of the theorem, we have a
transformation of the form $z'=z, w'=w+h(z,w)=w+o(1)$, which further
flattens $M$ to $M'$ near $0$ to the order of $N'$. For $M'$, we
similarly have $\wh{M'_{N'}}$, which, by the special property that
$z'=z$ of our transformation, can be seen to be precisely the image
of $\wh{M_{N}}$, near $0$, under the transformation. (Here, it
suffices to use the two dimensional uniqueness  result of
Kenig-Webster [KW]). Next, we can similarly define $\theta'$,
$L_j'$, as well as, the Levi matrix ${\cal L'}_{N'}$. Now ${\cal
L'}_{N'}=O(|z|^{N'/2-3})$. By the transformation formula of the Levi
-form, we see that
 there is an invertible real analytic matrix $B$ near $0$ and a
positive real analytic function $\kappa$ near $0$ such that
$${\cal L}_N=\kappa A{\cal L'}_{N'}\-{A^t}.$$
Hence, ${\cal L}_N=O(|z|^{N'/2-3})$ for any $N'>N.$ By the
analyticity of ${\cal L}_N$. We see that  ${\cal L}_N\equiv 0$ and
thus $\wh{M_{N}}$ is Levi-flat. Next, by the classical theorem of
Cartan, we see that $\wh{M_{N}}$ can be biholomorphically mapped to
an open piece in ${\mathbb C}^n\times {\mathbb R}$. This completes
the proof of the theorem.

\end{proof}

{\it Proof of Theorem \ref{thmm2} and Theorem \ref{thmm3}}: The
proof of Theorem \ref{thmm2} and Theorem \ref{thmm3} is an immediate
consequence of Theorem \ref{thm1} and Theorem \ref{new-1}. Here, we
only need to mention that when $M$ is already flattened, it is
obvious that the $\wh{M_N}$ constructed  is the local hull of
holomorphy of $M$ near $0$. By the invariant property of holomorphic
hull, we conclude that this is also the case when $M$ is not
flattened yet.
$\endpf$

\bigskip

\begin{example}\label {example-new}{\rm
Define $M\subset {\mathbb C}^3$ by the following equation near $0$:
$$w=q(z,\-{z})+p(z,\-{z})+iE(z,\-{z}).$$
Here as before
$q=|z_1|^2+\ld_1(z_1^2+\-{z_1^2})+|z_2|^2+\ld_2(z_2^2+\-{z_2^2})$
with $0\le \ld_1,\ld_2<\infty$, and $p, E=O(|z|^3)$ are real-valued.
Also $G(z,\-{z}):=q(z,\-{z})+p(z,\-{z})$.
 For any $c\in {\mathbb R}\sm \{0\}$, define the real hypersurface
 $K_c$ by the equation $q(z,\-{z})=c$. Then  $K_c$ intersects  transversally $M$
 along a submanifold $L_c$ of real dimension $3$. Then $L_c$ is a
 CR submanifold of CR dimension $1$ if and only if
$L(q)\equiv 0$ along $L_c$. Here

\begin{equation}\begin{split}\label{325eq22}
L=&(G_2-iE_2)\frac{\p}{\p z_1}-(G_1-iE_1)\frac{\p}{\p
z_2}+2i(G_2E_1-G_1E_1)\frac{\p}{\p w},
\end{split}\end{equation}
that is non-zero and tangent to $M\sm \{0\}$.

Write $\Psi=p(z,\-{z})-iE(z,\-{z})$. Then the above is equivalent to
the equation $\Psi_2\cdot(\-{z_1}+2\ld_1z_1)=\Psi_1\cdot
(\-{z_2}+2\ld_2z_2)$. Namely, $M$ is  non-minimal at its CR points
near $0$ if and only if the just mentioned equation holds.

One solution is given by
$\Psi=p(z,\-{z})-iE(z,\-{z})=\mu_1(|z_1|^2\-{z_1}+\ld_1|z_1|^2z_1)
+\mu_2(|z_2|^2\-{z_2}+\ld_2|z_2|^2z_2)
+\mu_1\-{z_1}(|z_2|^2+\ld_2z_2^2)+\mu_2\-{z_2}(|z_1|^2+\ld_1z_1^2),$
with $\mu_1,\mu_2\in{\mathbb C}$. Then
$\Psi_1=(\mu_1\-{z_1}+\mu_2\-{z_2})(\-{z_1}+2\ld_1z_1)$ and
$\Psi_2=(\mu_2\-{z_2}+\mu_1\-{z_1})(\-{z_2}+2\ld_1z_2)$. Thus
$\Psi_2\cdot (\-{z_1}+2\ld_1z_1)=\Psi_1\cdot (\-{z_2}+2\ld_2z_2)$
holds trivially.
}
\end{example}

Finally, We also mention  a  recent preprint [Bur2] for some
generalization of the work in Kenig-Webster [KW] and Huang-Krantz
[HK].

 \vfill \eject
\section{Appendix}

In this appendix, for convenience of the reader, we give a detailed
proof of Theorem \ref{thmm1} for the  case  $n=2$, $m=3$ to
demonstrate the basic ideas of the complicated calculations for the
proof  of Theorem \ref{thmm1} performed in Sections 4-6 of this
paper.
\medskip

In Section $3$, we have showed, by making use of the non-minimality,
the following:
\begin{equation} \label{0005}
\begin{split}
&\big(|w_n|^2+|w_1|^2\big)\cdot
\big(\-{w_n}\Psi_1-\-{w_1}\Psi_n\big)+\big(2\lambda_nw_n\-w_1-2\lambda_1w_1\-w_n\big)\cdot
\Psi=0, \ \ \hbox{where}\\
&\Psi=w_n\-{w_n} \Phi_1-{w_n}\-{w_1}\Phi_n+\-{w_1}\cdot \Phi,\
\Phi=w_nH_{\-1}-w_1H_{\-n},\\
&  H=E^{(m)},\
 \  w_j=z_j+2\lambda_j\-z_j\  \text{for}\ 1\leq j\leq n=2.
\end{split}
\end{equation}
We  use  the following notations
\begin{equation}\label{xieta00}
\begin{split}
&\xi=2\lambda_n,\ \eta=2\lambda_1,\ \theta=1-\xi^2,\\
&H_{[tsrh]}=H_{(te_n+se_1,re_n+he_1)}\ \text{for}\ t+s+r+h=m,\\
&\Phi_{[tsrh]}=\Phi_{(te_n+se_1,re_n+he_1)}\ \text{for}\ t+s+r+h=m,\\
&\Psi_{[tsrh]}=\Psi_{(te_n+se_1,re_n+he_1)}\ \text{for}\ t+s+r+h=m+1
\end{split}
\end{equation}
Here, for any homogeneous polynomial $\chi(z,\-{z})$ of degree $k\ge
1$, we write
$$\chi=\sum_{\a\ge 0,\b\ge 0, |\a|+|\b|=k}H_{(\a,\b)}z^{\a}\-{z^{\b}}.$$

We first set up more notations and establish formulas which are
crucial in the general case discussed in $\S 4-\S 6$.

 By (\ref{0005}), we have
\begin{equation}\label{atsrh00}
\begin{split}
\Phi_{[tsrh]}=&\xi(h+1)H_{[ts(r-1)(h+1)]}+(h+1)H_{[(t-1)sr(h+1)]}\\
         &-(r+1)H_{[t(s-1)(r+1)h]}-\eta(r+1)H_{[ts(r+1)(h-1)]},
\end{split}
\end{equation}
and
\begin{equation}\label{btsrh00}
\begin{split}
\Psi_{[tsrh]}=&(s+1)\big\{\xi\Phi_{[t(s+1)(r-2)h]}+(1+\xi^2)\Phi_{[(t-1)(s+1)(r-1)h]}
+\xi\Phi_{[(t-2)(s+1)rh]}\big\}\\
&-\xi(t+1)\Phi_{[(t+1)s(r-1)(h-1)]}
-t\Phi_{[tsr(h-1)]}-\xi\eta(t+1)\Phi_{[(t+1)(s-1)(r-1)h]}\\
& -\eta t\Phi_{[t(s-1)rh]}+\Phi_{[tsr(h-1)]}+\eta\Phi_{[t(s-1)rh]}.
\end{split}
\end{equation}

Collecting the coefficients of
$z_n^tz_1^{s-1}\-{z_n}^{r+3}\-{z_1}^h$ for $t\geq 0$, $s\geq 1$,
$r\geq -3$ and $h=m+1-t-s-r\geq 0$ in (\ref{0005}), we get
\begin{equation*}
\begin{split}
s\big\{&\xi \Psi_{[tsrh]}+(2\xi^2+1)\Psi_{[(t-1)s(r+1)h]}+(\xi^3+2\xi)\Psi_{[(t-2)s(r+2)h]}\\
&+\xi^2\Psi_{[(t-3)s(r+3)h]}\big\}+
\mathcal{F}\{(\Psi_{[t's'r'h']})_{s'+h'\leq s+h-2,s'\leq s,h'\leq
h}\}=0.
\end{split}
\end{equation*}
Here, for a set of complex numbers (or polynomials) $\{a_j,
b\}_{j=1}^{k}$, we say $b\in {\cal F}\{a_1,\cdots, a_k\}$ if
$b=\sum_{j=1}^{k}(c_ja_j+d_j\-{c_j})$ with $c_j,d_j\in
{\mathbb{C}}$. Also, we set up the convention that $\chi_{[tsrh]}=0$
if one of the indices is negative.

Thus for $s\geq 1$, we can inductively get
\begin{equation}\label{beq00}
\begin{split}
\Psi_{[tsrh]}=\mathcal{F}\{(\Psi_{[t's'r'h']})_{s'+h'\leq
s+h-2,s'\leq s,h'\leq h}\}.
\end{split}
\end{equation}
Substituting (\ref{btsrh00}) into (\ref{beq00}), we get, for $s\geq
1$, the following
\begin{equation*}
\begin{split}
&(s+1)\big\{\xi\Phi_{[t(s+1)(r-2)h]}+(1+\xi^2)\Phi_{[(t-1)(s+1)(r-1)h]}
+\xi\Phi_{[(t-2)(s+1)rh]}\big\}\\
&=\mathcal{F}\{(\Phi_{[t's'r'h']})_{s'+h'\leq s+h-1,s'\leq
s+1,h'\leq h}\}.
\end{split}
\end{equation*}
Hence for $s\geq 2$, we can inductively obtain
\begin{equation}\label{11100}
\begin{split}
\Phi_{[tsrh]}=\mathcal{F}\{(\Phi_{[t's'r'h']})_{s'+h'\leq
s+h-2,s'\leq s,h'\leq h}\}.
\end{split}
\end{equation}
Substituting (\ref{atsrh00}) into (\ref{11100}), we get, for $s\geq
2$ and $h\geq0$, the following
\begin{equation*}
\begin{split}
&\xi(h+1)H_{[ts(r-1)(h+1)]}+(h+1)H_{[(t-1)sr(h+1)]}=\mathcal{F}\{(\Phi_{[t's'r'h']})_{s'+h'\leq
s+h-2,s'\leq s,h'\leq h}\}\\
&+(r+1)H_{[t(s-1)(r+1)h]}+\eta(r+1)H_{[ts(r+1)(h-1)]}.\\
&=\mathcal{F}\big\{(H_{[t's'(m-t'-s'-h')h']})_{s'+h'\leq
s+h-1,s'\leq s,h'\leq h+1}\big\}.\\
\end{split}
\end{equation*}
Hence for $s\geq 2$ and $h\geq 1$, we can inductively get that
\begin{equation*}\begin{split}
H_{[ts(m-t-s-h)h]}=\mathcal{F}\big\{(H_{[t's'(m-t'-s'-h')h']})_{s'+h'\leq
s+h-2,s'\leq s,h'\leq h}\big\}.
\end{split}\end{equation*}
Notice that $H_{[tsrh]}=\-{H_{[rhts]}}$.  Keeping applying the above
until the assumption that $s\ge 2$ and $h\ge 1$ do not hold anymore,
  we can  inductively get the following crucial formula:
\begin{equation}\begin{split}\label{111ind00}
H_{[ts(m-t-s-h)h]}=\mathcal{F}\big\{(H_{[t'1(m-t'-2)1]})_{1\leq
t'\leq m-2},(H_{[t'0(m-t'-i)i]})_{i\leq \max(s,h),0\leq t'\leq
m-i}\big\}.
\end{split}\end{equation}

\bigskip

Now, we assume $m=3$. We first normalize $H:=E^{(3)}$ without using
the non-minimality condition.

Let $z'=z, w'=w+B(z,w)$ be a holomorphic transformation that
transforms $w=G(z,\-z)+iE(z,\-z)$ to $w'=G'(z',\-z')+iE'(z',\-z')$.
Then we get
$$
{\Im}B(z,w)=E'(z,\-z)-E(z,\-z).
$$
Here $B(z,w)$ is a weighted holomorphic homogeneous polynomial in
$(z,w)$ of degree $3$, with $wt(z)=1$ and $wt(w)=2$.

\medskip
{\bf Sub-appendix I:} In this part, we first prove the following:
\begin{lem}
After a holomorphic transformation, we can have  $E(z,\-z)$ defined
in
(\ref{429eq1}) to  satisfy  the following normalization:\\
{\bf (1)} When $\lambda_n=0$, then
 \begin{equation}\begin{split}\label{norm30}
E_{(3e_n,0)}=E_{(2e_n+e_1,0)}=E_{(e_n+2e_1,0)}=E_{(3e_1,0)}=E_{(2e_n,e_n)}=E_{(e_n+e_1,e_n)}=0.
\end{split}\end{equation}
{\bf (2)} When $\lambda_n\neq 0$, then
 \begin{equation}\begin{split}\label{norm3}
E_{(3e_n,0)}=E_{(2e_n+e_1,0)}=E_{(e_n+2e_1,0)}=E_{(3e_1,0)}=E_{(e_n+e_1,e_1)}=E_{(e_n+e_1,e_n)}=0.
\end{split}\end{equation}
\end{lem}

\begin{proof}
First, notice that the real dimension of  the space of all such
$B^{(3)}$ is
 \begin{equation}\begin{split}\label{bdim00}
&2\cdot \sharp\big\{(i_1,i_{n},j)\in \mathbb{R}^{3}:\
i_1,i_{n},j\geq 0,\  i_1+i_{n}+2j=3
\big\}\\
=&2\cdot \sharp\big\{(i_1,i_{n},j)\in \mathbb{R}^{3}:\ i_1\neq 0,\
i_1+i_{n}+2j=3 \big\}+4.
\end{split}\end{equation}

{\bf (1)} Assume that $\lambda_n=0$. Set
$$
\hat{P}^{(3)}=\big\{\text{polynomials  of the form}\
2{\Re}\sum_{i+j+2k=3}
    a_{(ijk)}z_1^iz_n^j|z_n|^{2k}\big\}.
$$
To get the normalization condition (\ref{norm30}), we only need to
prove that
\begin{equation}\label{227eq004}
{\Im}\big(B^{(3)}(z,q(z,\-z))\big)\big|_{\hat{P}^{(3)}}=Q^{(3)}(z,\-z)
\end{equation}
is solvable for any  $Q^{(3)}(z,\-z)\in \hat{P}^{(3)}$.  Here, we
choose an orthonormal basis $\{z^\a\-{z^\b}\}$ for the space of (not
necessarily holomorphic)   polynomials. Then for any subspace $P$
and a polynomial $A$, we write $A|_{P}$ for the orthogonal
projection of $A$ to $P$.

Notice that $\hat{P}^{(3)}$ and the space
$\{B^{(3)}(z,q(z,\-{z}))\}$ have the same dimension. Hence to prove
(\ref{227eq004}),  we  need to show that
$$
 {\Im}\big(B^{(3)}(z,{q}(z,\-z))\big)\big|_{\hat{P}^{(3)}}=0\
 \Longleftrightarrow B=0.
$$
By considering the terms involving only $z_n$ and $\-z_n$, we get
$$
{\Im}\big(b_{(011)} z_n|z_n|^{2}+b_{(030)} z_n^3\big)=0.
$$
Thus we get $b_{(011)}=b_{(030)}=0$. Hence, if
${\Im}\big(B^{(3)}(z,{q}(z,\-z))\big)\big|_{\hat{P}^{(3)}}=0$, we
have   that $B^{(3)}(z,|z_n|^2+\lambda_1z_1^2)=0$. Namely, we have
\begin{equation}
\begin{split}
b_{(120)}z_1z_n^2+b_{(101)}z_1|z_n|^2+b_{(210)}z_1^2z_n+(b_{(300)}+\lambda_1
b_{101})z_1^3=0.
\end{split}
\end{equation}
Hence we get $b_{(ijk)}=0$. Furthermore, we obtain
$B^{(3)}(z,\-z)=0$.
\bigskip

{\bf (2)} In this case, we assume $\lambda_n\neq 0$. Write
 \begin{equation*}\begin{split}
\hat{P}_{-3}^{(3)}=\{&\text{homogeneous polynomials of the form: }\\
&2{\Re}\big(\sum\limits_{i\geq
1,i+j+2k=3}a_{ijk}z_1^iz_n^j|z_n|^{2k}+bz_n^3+cz_n|z_1|^2\big)\}.
\end{split}\end{equation*}

To get the normalization condition (\ref{norm3}), we only need to
prove that
\begin{equation}\label{227eq-300}
{\Im}\big(B^{(3)}(z,q(z,\-z))\big)\big|_{\hat{P}_{-3}^{(3)}}=Q^{(3)}(z,\-z)
\end{equation}
is solvable for any real valued polynomial $Q^{(3)}(z,\-z)\in
\hat{P}_{-3}^{(3)}$.
 Notice that the real dimension of $\hat{P}_{-3}^{(3)}$ is
$$
2\cdot \sharp\big\{(i_1,i_{n},j)\in \mathbb{R}^{3}:\ i_1\neq 0,\
i_1+i_{n}+2j=3 \big\}+4,
$$
which is the same as the real dimension of all such  $B(z,w)'s$.
Hence to prove (\ref{227eq-300}), we need to show that
$$
 {\Im}\big(B^{(3)}(z,{q}(z,\-z))\big)\big|_{\hat{P}_{-3}^{(3)}}=0
 \Longleftrightarrow B^{(3)}=0.
$$
Notice that
\begin{equation}\label{3eq1}
\begin{split}
0&={\Im}B^{(3)}(z,w)\big|_{\hat{P}^{(3)}_{-3}}\\
&={\Im}\Big(\sum\limits_{i\geq
1,i_j+2k=3}b_{(ijk)}z_1^iz_n^jw^k+b_{(030)}z_n^3+b_{(011)}z_n(|z_n|^2
+\lambda_nz_n^2+\lambda_n\-{z_n}^2+\l_1z_1^2+|z_1|^2)\Big)\big|_{\hat{P}^{(3)}_{-3}}.
\end{split}
\end{equation}
Collecting the coefficients of $z_n|z_1|^2$ and $z_n^3$,
respectively, in (\ref{3eq1}), we get $b_{(011)}=0$ and
$b_{(030)}+\lambda_nb_{(011)}=0$. Thus we get
$b_{(030)}=b_{(011)}=0$. Hence
${\Im}B^{(3)}(z,w)\big|_{\hat{P}^{(3)}_{-3}}=0$ implies that
$$
b_{(300)}z_1^3+b_{(210)}z_1^2z_n+b_{(120)}z_1z_n^2+b_{(101)}z_1(|z_n|^2
+\lambda_nz_n^2+z_1^2)=0.
$$
Now it is obvious that $b_{(101)}=b_{(120)}=b_{(300)}=b_{(210)}=0$.
Hence we have get $B^{(3)}=0$. This completes the proof of
(\ref{norm3}).
\end{proof}

\medskip

 {\bf Sub-appendix II}: Now we proceed to prove Theorem \ref{thmm1} for  $n=2$ and $m=3$.

\medskip
{\bf Case I:}  In this case, we assume that $\lambda_n=\lambda_1=0$.
Then (\ref{0005}) has the following form:
\begin{equation}\label{00500}
\begin{split}
\-z_n\Psi_1=\-{z_1}\Psi_n.
\end{split}
\end{equation}
By considering the coefficients of
$z_n^tz_1^{s-1}\-{z_n}^{r+1}\-{z_1}^{h}$ for $t\geq 0,$ $s\geq 1$,
$r\geq 0$ and $h=(m+1)-t-s-r\geq 0$ in (\ref{00500}), we get
\begin{equation}\label{00100}
\begin{split}
s\Psi_{[tsrh]}=(t+1)\Psi_{[(t+1)(s-1)(r+1)(h-1)]}.
\end{split}
\end{equation}
Setting $h=0$ in (\ref{00100}), we get $\Psi_{[tsr0]}=0$ for $s\geq
1$. Combining this with (\ref{00100}), we  inductively get
$\Psi_{[tsrh]}=0$ for $s\geq h+1$. Now, we will apply
(\ref{btsrh00}). Notice that we now have $\xi=\eta=0$.
We thus obtain:
\begin{equation}\label{00200}
\begin{split}
(s+1)\Phi_{[(t-1)(s+1)(r-1)h]}=(t-1)\Phi_{[tsr(h-1)]}\ \text{for}\
s\geq h+1.
\end{split}
\end{equation}
Setting $h=0$ in (\ref{00200}), we get $ \Phi_{[tsr0]}=0\
\text{for}\ s\geq 2.$ Combining this with (\ref{00200}), we
inductively get $ \Phi_{[tsrh]}=0$ for $s\geq h+2$. Together with
(\ref{atsrh00}), we get
\begin{equation}\label{00300}
\begin{split}
(h+1)H_{[(t-1)sr(h+1)]}=(r+1)H_{[t(s-1)(r+1)h]}\ \text{for}\ s\geq
h+2.
\end{split}
\end{equation}
Setting $t=0$, we get $H_{[0srh]}=0$ for $s\geq h+1, \ r\geq 1$.
Then we  inductively get $H_{[tsrh]}=0$ for $s\geq h+1, \ r\geq
t+1$. When $s\geq h+1, r\leq t$,   from (\ref{00300}), we
inductively get $H_{[tsrh]}=\mathcal{F}\{(H_{[t's'r'0]})_{t'\geq
r'}\}$, which is $0$ by our normalization in (\ref{norm30}). Thus we
have proved
\begin{equation}\label{00400}
\begin{split}
H_{[tsrh]}=0\ \text{for}\ s\geq h+1.
\end{split}
\end{equation}
Next we  will prove that $H_{[tsrs]}=0$. Setting $s=h\geq 1$, $t\geq
0$ and $r=-1$ in (\ref{00100}), we get $ \Psi_{[ts0s]}=0\
\text{for}\ t\geq 1. $ Substituting it back to (\ref{00100}), we
inductively get
$$
\Psi_{[tsrs]}=0\ \text{for}\ t\geq r+1.
$$
Substituting  (\ref{btsrh00}) into this equation, we get
\begin{equation}\label{005000}
\begin{split}
(s+1)\Phi_{[(t-1)(s+1)(r-1)s]}=(t-1)\Phi_{[tsr(s-1)]}\ \text{for}\
t\geq r+1.
\end{split}
\end{equation}
Setting $s=0$, we get $\Phi_{[t1r0]}=0$ for $t\geq r+1$.
Substituting this back to (\ref{005000}), we get
$\Phi_{[t(s+1)rs]}=0$ for $t\geq r+1$. Together with
(\ref{atsrh00}), we get
\begin{equation*}
\begin{split}
(s+1)H_{[(t-1)(s+1)r(s+1)]}=(r+1)H_{[ts(r+1)s]}\ \text{for}\ t\geq
r+1.
\end{split}
\end{equation*}
Notice that $H_{[t0r0]}=0$ by our normalization. Hence we
inductively get
\begin{equation}\label{00700}
\begin{split}
H_{[tsrs]}=0\ \text{for}\ t\geq r.
\end{split}
\end{equation}
Since $H_{[tsrh]}=\-{H_{[rhts]}}$, (\ref{00400}) and (\ref{00700})
imply $H\equiv 0$ for the case $\lambda_n=\lambda_1=0$.

\bigskip

{\bf Step II:}  In this step, we assume that $\lambda_n=0$ and
$\lambda_1\neq 0$. Theorem \ref{thmm1} with $m=3$ in this setting is
an immediate consequence of the following lemma:

\begin{lem}\label{lem0100}
  Suppose that $\lambda_n=0$ and $\lambda_1\neq 0$. Assume that there exists an $h_0\geq
  -1$ such that
\begin{equation}\label{01a00}
\Psi_{[tsrh]}=\Phi_{[tsrh]}=0\ \text{for}\ h\leq h_0,\ H_{[tsrh]}=0\
\text{for}\ \max(s,h)\leq h_0+1.
\end{equation}
Then we have
\begin{equation}\label{01c00}
\Psi_{[tsrh]}=\Phi_{[tsrh]}=0\ \text{for}\ h\leq h_0+1,\
H_{[tsrh]}=0\ \text{for}\ \max(s,h)\leq h_0+2.
\end{equation}
\end{lem}

Once we have Lemma \ref{lem0100} at our disposal, since
(\ref{01a00}) holds for $h_0=-1$ by our normalization, hence
(\ref{01c00}) holds for $h_0=-1$. Then by an induction,  we see that
(\ref{01c00}) holds for all $h_0\leq m-2$. This  completes the proof
of Theorem \ref{thmm1} in this setting.
\medskip

\begin{proof}[Proof of Lemma \ref{lem0100}]
First, notice that  (\ref{0005}) has the following form:
\begin{equation*}
\begin{split}
&(|z_n|^2+\eta z_1^2)(\-{z_n}\Psi_1-\eta z_1\Psi_n)-\eta
z_1\-{z_n}\Psi+\-{z_1}\mathcal{F}\{\Psi_1,\Psi_n,\Psi\}=0.
\end{split}
\end{equation*}
Collecting the coefficients of
$z_n^tz_1^{s-1}\-{z_n}^{r+3}\-{z_1}^{h_0+1}$ in the above equation
and making use of  the assumptions in Lemma \ref{lem0100}, we have:
\begin{equation*}
\begin{split}
&s\Psi_{[(t-1)s(r+1)(h_0+1)]}+(s-2)\eta\Psi_{[t(s-2)(r+2)(h_0+1)]}-t\eta\Psi_{[t(s-2)(r+2)(h_0+1)]}\\
&-(t+1)\eta^2\Psi_{[(t+1)(s-4)(r+3)(h_0+1)]}-\eta\Psi_{[t(s-2)(r+2)(h_0+1)]}=0.
\end{split}
\end{equation*}
Namely, we have
\begin{equation}\label{01200}
\begin{split}
&s\Psi_{[(t-1)s(r+1)(h_0+1)]}=(t+3-s)\eta\Psi_{[t(s-2)(r+2)(h_0+1)]}
+(t+1)\eta^2\Psi_{[(t+1)(s-4)(r+3)(h_0+1)]}.
\end{split}
\end{equation}
By setting $r=-3$ in (\ref{01200}), we get $\Psi_{[ts0(h_0+1)]}=0$
for $t\geq 1$. Substituting this back to (\ref{01200}), we
inductively get that  $\Psi_{[tsr(h_0+1)]}=0$ for $t\geq r+1$.
Combining this with (\ref{btsrh00}) and the hypothesis,
we obtain
\begin{equation}\label{01400}
\begin{split}
(s+1)\Phi_{[(t-1)(s+1)(r-1)(h_0+1)]}=(t-1)\eta\Phi_{[t(s-1)r(h_0+1)]}\
\text{for}\ t\geq r+1.
\end{split}
\end{equation}
Setting $r=0$ in (\ref{01400}), we get $\Phi_{[ts0(h_0+1)]}=0$ for
$t\geq 2$. Hence we  inductively get  $\Phi_{[tsr(h_0+1)]}=0$ for
$t\geq r+2$. In particular, we have $\Phi_{[t0r(h_0+1)]}=0$ for
$t\geq r+2$. Combining this with (\ref{atsrh00}),the hypothesis and
$\lambda_n=0$, we get $ (h_0+2)H_{[(t-1)0r(h_0+2)]}=0$ for $t\geq
r+2$. Namely, we obtain $ H_{[t0r(h_0+2)]}=0$ for $t\geq r+1$.
Together with our normalization (\ref{norm30}) and the reality of
$H$, we obtain:
\begin{equation}\label{01600}
\begin{split}
H_{[t0r(h_0+2)]}=0.
\end{split}
\end{equation}

{\bf (1)} When $h_0=-1$, setting $s=0$ in (\ref{01400}), we get
$\Phi_{[(t-1)1(r-1)0]}=0$ for $t\geq r+1$. Together with
(\ref{atsrh00}) and (\ref{norm30}), we get  that $
H_{[(t-1)1r1]}=(r+1)H_{[t0(r+1)0]}=0$ for $t\geq r+1$.  Namely, we
obtain $H_{[t1r1]}=0$ for $t\geq r$. By the reality of $H$, we get
for all $t,r$ the following:
\begin{equation}\label{001600}
\begin{split}
H_{[t1r1]}=0.
\end{split}
\end{equation}
From (\ref{atsrh00}), (\ref{01600}) and (\ref{001600}), we obtain
$\Phi_{[t0r0]}=\Phi_{[t1r0]}=0$. Together with (\ref{btsrh00}), we
see that $\Psi_{[t0r0]}=0$.

Setting $h=0$ in (\ref{beq00}) and making use of $\Psi_{[t0r0]}=0$,
we first get $\Psi_{[t1r0]}=\Psi_{[t2r0]}=0$, then
 inductively get $\Psi_{[tsr0]}=0$. Combining this with
(\ref{btsrh00}), we get
\begin{equation}\label{001700}
(s+1)\Phi_{[(t-1)(s+1)(r-1)0]}=(t-1)\eta \Phi_{[t(s-1)r0]},
\end{equation}
Setting $s=0$ in (\ref{001700}), we obtain $\Phi_{[t1r0]}=0$. By an
induction argument, we  get $\Phi_{[tsr0]}=0$. This proves
(\ref{01c00}) for the case $h_0=-1$.  \medskip

{\bf (2)} When $h_0\geq 0$, from (\ref{111ind00}),(\ref{01600}) and
(\ref{001600}), we inductively get $H_{[tsr(h_0+2)]}=0$ for $s\leq
h_0+2$. Combining this with (\ref{atsrh00}) and (\ref{001600}), we
get $\Phi_{[t0r(h_0+1)]}=\Phi_{[t1r(h_0+1)]}=0$. Substituting this
back to (\ref{btsrh00}), we obtain $\Psi_{[t0r(h_0+1)]}=0$. Together
with (\ref{beq00}), we inductively get $\Psi_{[tsr(h_0+1)]}=0$.
Combining this with (\ref{btsrh00}), we obtain
\begin{equation*}
(s+1)\Phi_{[(t-1)(s+1)(r-1)(h_0+1)]}=(t-1)\eta
\Phi_{[t(s-1)r(h_0+1)]}.
\end{equation*}
As in Case I, we inductively get $\Phi_{[tsr(h_0+1)]}=0$. This
proves (\ref{01c00}) for the case $h_0\geq 0$ and  thus completes
the proof of Theorem \ref{thmm1} for the case $\lambda_n=0$ and
$\lambda_1\neq 0$.
\end{proof}

{\bf Case III:} In this case, we assume $\lambda_n\neq 0$ and
$\lambda_1\neq 0$. Considering the coefficients of $z_1^3\-{z_n}^3$,
$z_1^3z_n\-{z_n}^2$, $z_1^3z_n^2\-{z_n}$ and $z_1^3z_n^3$,
respectively, in (\ref{0005}), we get
\begin{align*}
&4\xi\Psi_{[0400]}+2\eta\Psi_{[0220]}-\eta\xi\Psi_{[1210]}-\eta^2\Psi_{[1030]}-\eta\theta
\Psi_{[0220]}=0,                     \\
&4(2\xi^2+1)\Psi_{[0400]}+2\eta(\Psi_{[1210]}+\xi\Psi_{[0220]})
-\eta(2\xi\Psi_{[2200]}                \\
&\hskip 1cm+(1+\xi^2) \Psi_{[1210]})-2\eta^2\Psi_{[2020]}-\eta\theta
\Psi_{[1210]}=0,                     \\
&4(\xi^3+2\xi)\Psi_{[0400]}+2\eta(\Psi_{[2200]}+\xi\Psi_{[1210]})-\eta(2(1+\xi^2)\Psi_{[2200]}     \\
&\hskip 1cm +\xi \Psi_{[1210]})-3\eta^2\Psi_{[3010]}-\eta\theta
\Psi_{[2200]}=0,                \\
&4\xi^2\Psi_{[0400]}+2\eta\xi\Psi_{[2200]}-\eta2\xi\Psi_{[2200]}-4\eta^2\Psi_{[4000]}=0.
\end{align*}
Simplifying the above from the last equation to the first one, we
get
\begin{align}
&4\xi^2\Psi_{[0400]}=4\eta^2\Psi_{[4000]},     \label{n004}      \\
&4(\xi^3+2\xi)\Psi_{[0400]}+\eta\big\{(-2\xi^2-\theta)\Psi_{[2200]}+\xi\Psi_{[1210]}\big\}=
3\eta^2\Psi_{[3010]},                  \label{n003}\\
&4(2\xi^2+1)\Psi_{[0400]}+\eta\big\{-2\xi\Psi_{[2200]}+2\xi\Psi_{[0220]}\big\}
=2\eta^2\Psi_{[2020]},          \label{n002}              \\
&4\xi\Psi_{[0400]}+\eta\big\{-\xi\Psi_{[1210]}+(2-\theta)\Psi_{[0220]}\big\}=\eta^2
\Psi_{[1030]}.       \label{n001}
\end{align}
A direct computation shows that
\begin{equation}\label{3eq3}
\begin{split}
&(2-3\theta)\xi^2-(2-2\theta)\xi(\xi^3+2\xi)+(2-\theta)\xi^2
(2\xi^2+1)-2\xi^3 \xi\\
=&2\xi^2-2\xi(\xi^3+2\xi)+2\xi^2 (2\xi^2+1)-2\xi^3
\xi+\theta(-3\xi^2+2\xi(\xi^3+2\xi)-\xi^2 (2\xi^2+1))=0.
\end{split}
\end{equation}
We also have the following computation:
\begin{equation}\label{3eq004}
\begin{split}
&-(2-2\theta)\xi\big\{(-2\xi^2-\theta)\Psi_{[2200]}+\xi\Psi_{[1210]}\big\}
+(2-\theta)\xi^2
\big\{-2\xi\Psi_{[2200]}+2\xi\Psi_{[0220]}\big\}\\
&-2\xi^3
\big\{-\xi\Psi_{[1210]}+(2-\theta)\Psi_{[0220]}\big\}\\
=&(-\xi)\big\{(2-2\theta)(\theta-2)+2\xi^2(2-\theta)\big\}\Psi_{[2200]}
+\xi^2\big\{-(2-2\theta)+2\xi^2\big\}\Psi_{[1210]}\\
&-\xi^3\big\{-2(2-\theta)+2(2-\theta)\big\}\Psi_{[0220]}\\
=&0.
\end{split}
\end{equation}

Computing
$(2-3\theta)(\ref{n004})-(2-2\theta)\xi(\ref{n003})+(2-\theta)\xi^2(\ref{n002})-2\xi^3(\ref{n001})$
and making use of (\ref{3eq3})-(\ref{3eq004}), we get
\begin{equation}\label{3eq7}
\begin{split}
(2-3\theta)4\Psi_{[4000]}-(2-2\theta)\xi3\Psi_{[3010]}+(2-\theta)\xi^22\Psi_{[2020]}-2\xi^3\Psi_{[1030]}=0.
\end{split}
\end{equation}
Substituting (\ref{btsrh00}) into (\ref{3eq7}), we get
\begin{equation}\label{3eq8}
\begin{split}
&(2-3\theta)4\Psi_{[4000]}-(2-2\theta)\xi3\Psi_{[3010]}+(2-\theta)\xi^22\Psi_{[2020]}-2\xi^3\Psi_{[1030]}\\
=&(2-3\theta)4\xi\Phi_{[2100]}-(2-2\theta)\xi3\big\{(1+\xi^2)\Phi_{[2100]}+\xi\Phi_{[1110]}\big\}\\
&+(2-\theta)\xi^22\big\{\xi\Phi_{[2100]}+(1+\xi^2)\Phi_{[1110]}+\xi\Phi_{[0120]}\big\}\\
&-2\xi^3\big\{\xi\Phi_{[1110]}+(1+\xi^2)\Phi_{[0120]}\big\}\\
=&\xi\big\{4(2-3\theta)-3(2-2\theta)(1+\xi^2)+2\xi^2(2-\theta)\big\}\Phi_{[2100]}\\
&+\xi^2\big\{-3(2-2\theta)+2(2-\theta)(1+\xi^2)-2\xi^2\big\}\Phi_{[1110]}\\
&+\xi^3\big\{2(2-\theta)-2-2\xi^2\big\}\Phi_{[0120]}\\
=&-4\xi\theta^2\Phi_{[2100]}+\xi^22\theta^2\Phi_{[1110]}=-2\xi\theta^2(2\Phi_{[2100]}+(-\xi)\Phi_{[1110]}).
\end{split}
\end{equation}
Since $\theta\not =0$ with the assumption that $\l_n\not
=\frac{1}{2}$.
Hence we get $2\Phi_{[2100]}+(-\xi)\Phi_{[1110]}=0$.
From (\ref{atsrh00}), we get
\begin{equation}\label{3eq9}
\begin{split}
0=2\Phi_{[2100]}+(-\xi)\Phi_{[1110]}=2(H_{[1101]}-H_{[2010]})-\xi\big(\xi
H_{[1101]}+H_{[0111]}-2H_{[1020]}\big).
\end{split}
\end{equation}
 By our normalization
(\ref{norm3}), we have $H_{[1101]}=H_{[0111]}=0$. Thus we get
$2H_{[2010]}-2\xi H_{[1020]}=0$. By the reality of $H$, we obtain:
$$
2(1-\xi){\Re}H_{[2010]}+\sqrt{-1}\cdot2(1+\xi) {\Im}H_{[2010]}=0.
$$
When $\lambda_n\neq 1/2$, then $1-\xi\neq 0$. Hence we get
$H_{[2010]}=0$.

\medskip


Collecting the terms of the form $z_n^t\-{z_n}^{6-t}$ $(0\leq t\leq
6)$ in (\ref{0005}), we get
$$
|w_n|^2\-{w_n}\sum\limits_{t'+r'=3}\Psi_{[t'1r'0]}z_n^{t'}\-{z_n}^{r'}=0.
$$
Thus we get
\begin{equation}\label{new1}
\Psi_{[t1r0]}=0\ \text{for}\ t+r=3.
\end{equation}
Combining this with (\ref{btsrh00}), we get
\begin{equation}\label{3eq900}
\begin{split}
0=&\Psi_{[3100]}+(-\xi)\Psi_{[2110]}+(-\xi)^2\Psi_{[1120]}+(-\xi)^3\Psi_{[0130]}\\
=&2\xi\Phi_{[1200]}-2\eta\Phi_{[3000]}+(-\xi)\big\{2(1+\xi^2)\Phi_{[1200]}+2\xi
\Phi_{[0210]}-3\xi\eta\Phi_{[3000]}-\eta\Phi_{[2010]}\big\}\\
&+(-\xi)^2\big\{2\xi
\Phi_{[1200]}+2(1+\xi^2)\Phi_{[0210]}-2\xi\eta\Phi_{[2010]}\big\}\\
&+(-\xi)^3\big\{2\xi
\Phi_{[0210]}-\xi\eta\Phi_{[1020]}+\eta\Phi_{[0030]}\big\}\\
=&(-2+3\xi^2)\eta\Phi_{[3000]}+(\xi-2\xi^3)\eta\Phi_{[2010]}+\xi^4\eta\Phi_{[1020]}-\xi^3\eta\Phi_{[0030]}\\
=&(1-3\theta)\eta\Phi_{[3000]}+(-\xi)(1-2\theta)\eta\Phi_{[2010]}+\xi^2(1-\theta)\eta\Phi_{[1020]}
+(-\xi)^3\eta\Phi_{[0030]}.
\end{split}
\end{equation}
Substituting (\ref{atsrh00}) into the equation above, we get
\begin{equation}
\begin{split}
& (1-3\theta)\eta H_{[2001]}+(-\xi)(1-2\theta)\eta\big(\xi
H_{[2001]}+H_{[1011]}\big)\\
&+(-\xi)^2(1-\theta)\eta\big(\xi
H_{[1011]}+H_{[0021]}\big)+(-\xi)^3\eta\xi H_{[0021]}=0.
\end{split}
\end{equation}
 By (\ref{norm3}), we have
$H_{[1011]}=H_{[0021]}=0$. Hence we get $-2\theta^2H_{[2001]}=0.$
When $\theta\not =0$ or $\xi\not = \frac{1}{2}$, we conclude that
$H_{[2001]}=0$. Together with $H_{[2010]}$, the normalization in
(\ref{atsrh00}) and the reality of $H$,   we see that $H_{[tsrh]}=0\
\text{for}\ s,h\leq 1. $ Substituting this back to (\ref{atsrh00})
and (\ref{btsrh00}), we get $\Phi_{[t'0r'0]}=\Phi_{[t''1r''0]}=0$
and thus also $\Psi_{[t0r0]}=0$ for $t+r=4$ and $t'+r'=3$,
$t''+r''+1=3$.

Collecting  terms of the form $z_n^t\-{z_1}\cdot\-{z_n}^{5-t}$
$(0\leq t\leq 5)$ in (\ref{0005}), and making use of
$\Psi_{[t'0r'0]}=0$, we get
$$
|w_n|^2\-{w_n}\sum\limits_{t'+r'=3}\Psi_{[t'1r'1]}z_n^{t'}\-{z_n}^{r'}=0.
$$
Thus we get $\Psi_{[t1r1]}=0$.  Combining this with (\ref{btsrh00}),
we get
\begin{equation}
\begin{split}
0=&\Psi_{[2101]}+(-\xi)\Psi_{[1111]}+(-\xi)^2\Psi_{[0121]}\\
=&2\xi\Phi_{[0201]}-\eta\Phi_{[2001]}+(-\xi)\big\{2(1+\xi^2)\Phi_{[0201]}-2\xi\eta\Phi_{[2001]}\big\}\\
&+(-\xi)^2\big\{2\Phi_{[0201]}-\xi\eta\Phi_{[1011]}+\eta\Phi_{[0021]}\big\}\\
=&(-1+2\xi^2)\eta\Phi_{[2001]}-\xi^3\eta\Phi_{[1011]}+\xi^2\eta\Phi_{[0021]}\\
=&(1-2\theta)\eta\Phi_{[2001]}+(-\xi)(1-\theta)\eta\Phi_{[1011]}+\xi^2\eta\Phi_{[0021]}.
\end{split}
\end{equation}
Substituting (\ref{atsrh00}) into this equation, we get
\begin{equation}
\begin{split}
 (1-2\theta)\eta2H_{[1002]}+(-\xi)(1-\theta)\eta\big(2\xi
H_{[1002]}+ 2H_{[0012]}\big)+\xi^2\eta\xi 2 H_{[0012]}=0.
\end{split}
\end{equation}
By (\ref{norm3}), we have $H_{[0012]}=0$. Hence we get
$-\theta^2H_{[1002]}=0$. Since $\theta\not = 0$, we see that
$H_{[1002]}=0$.

By (\ref{new1}), we have $\Psi_{[0130]}=0$. Combining this with
(\ref{btsrh00}), we get
\begin{equation}
\begin{split}
\Psi_{[0130]}=&2\xi\Phi_{[0210]}-(\xi\eta\Phi_{[1020]}-\eta\Phi_{[0030]})\\
=&2\xi(\xi H_{[0201]}-2H_{[0120]})=2\xi^2H_{[0201]}.
\end{split}
\end{equation}
Here, we used the fact that $\Phi_{[1020]}=\Phi_{[0030]}=0$ and
$H_{[tsrh]}=0$ for $s,h\le 1$. Thus we get $H_{[0201]}=0$. Now,
combing the normalization in (\ref{btsrh00}) with $H_{[0201]}=0,
H_{[2010]}=0, H_{[tsrh]}=0$ for $s, h\le 1$, we conclude  that $H
\equiv 0$. This completes the proof of Theorem \ref{thmm1} for the
case of $\lambda_n\neq 0$, $n=2$ and $m=3$.
\bigskip

\vfill \eject

\end{document}